\documentclass[bj,numbers]{imsart}
\usepackage{mathtools,amsthm,algorithm}
\usepackage[noend]{algpseudocode}
\makeatletter
\def\BState{\State\hskip-\ALG@thistlm}
\makeatother
\usepackage{float}
\usepackage{booktabs}
\usepackage[numbers,sort&compress]{natbib}
\usepackage[capitalize]{cleveref}
\usepackage{multirow}
\usepackage{subcaption}
\startlocaldefs
\theoremstyle{plain}
\newtheorem{theorem}{Theorem}
\newtheorem{corollary}{Corollary}
\newtheorem{proposition}{Proposition}
\newtheorem{lemma}{Lemma}

\theoremstyle{remark}

\newtheorem{assumption}{Assumption}

\newtheorem{remark}{Remark}

\usepackage{amsmath,amssymb,mathtools}
\usepackage{xifthen}
\usepackage{dsfont}
\newcommand{\cA}{\mathcal{A}}
\newcommand{\cB}{\mathcal{B}}
\newcommand{\cC}{\mathcal{C}}

\newcommand{\cE}{\mathcal{E}}
\newcommand{\cF}{\mathcal{F}}

\newcommand{\cK}{\mathcal{K}}

\newcommand{\cN}{\mathcal{N}}
\newcommand{\cO}{\mathcal{O}}
\newcommand{\cP}{\mathcal{P}}

\newcommand{\cS}{\mathcal{S}}

\newcommand{\cX}{\mathcal{X}}
\newcommand{\cY}{\mathcal{Y}}

\newcommand{\EE}{\mathbb{E}}

\newcommand{\PP}{\mathbb{P}}

\newcommand{\RR}{\mathbb{R}}

\newcommand{\1}{\mathds{1}}

\newcommand*{\kl}[3][]{%
\ifthenelse{\isempty{#1}}{\operatorname{D}(#2\,\|\,#3)}%
{\operatorname{D}(#2\,\|\,#3\mid#1)}%
}
\newcommand*{\tv}[2]{\mathrm{d_{TV}}(#1, #2)}

\DeclarePairedDelimiter{\norm}{\|}{\|}
\DeclarePairedDelimiter{\triplenorm}{\vert\kern-0.25ex\vert\kern-0.25ex\vert}{\vert\kern-0.25ex\vert\kern-0.25ex\vert}

\newcommand*{\E}{\mathbb E}

\newcommand*{\ep}{\varepsilon}
\newcommand*{\defeq}{\coloneqq}
\newcommand*{\rd}{\mathrm{d}}
\newcommand*{\dd}{\, \rd}
\DeclareMathOperator*{\argmin}{argmin}
\DeclareMathOperator*{\argmax}{argmax}

\newcommand*{\todo}[1][]{%
\ifthenelse{\equal{#1}{}}{\textcolor{red}{[\textbf{TODO}]}}%
{\textcolor{red}{[\textbf{TODO:} #1]}}%
}

\newcommand{\conv}{\mathrm{conv}}

\def\bI{\bm{I}}
\def\bX{\bm X}
\newcommand{\rs}{}
\newcommand{\wh}{\widehat}
\newcommand{\wt}{\widetilde}
\newcommand{\rI}{\mathrm{I}}
\newcommand{\rII}{\mathrm{II}}
\newcommand{\rIII}{\mathrm{III}}
\newcommand{\rIV}{\mathrm{IV}}
\newcommand{\be}{\mathbf{e}}
\newcommand{\diam}{\mathrm{diam}}
\newcommand{\rank}{\mathrm{rank}}
\newcommand{\diag}{\mathrm{diag}}
\newcommand{\T}{\top}

\newcommand{\op}{\mathrm{op}}

\newcommand{\balpha}{\boldsymbol{\alpha}}
\newcommand{\bbeta}{\boldsymbol{\beta}}
\def\I{\mathcal{I}}
\newcommand{\Tm}{\alpha_{\min}}
\def\oJ{\overline{J}} 
\def\uJ{\underline{J}}
\def\ok{\overline{\kappa}}
\def\uk{\underline{\kappa}} 
\def\i{\infty}
\def\supp{{\rm supp}}
\newcommand{\swd}{\operatorname{SWD}}
\def\bS{\mathbb{S}}
\def\trace{{\rm tr}}
\endlocaldefs

\begin{document}
	\begin{frontmatter}
		\title{Estimation and inference for the Wasserstein distance between mixing measures in topic models}
		
		\runtitle{Estimation and inference for the  Wasserstein Distance in Topic Models}
		
		\begin{aug}
			\author[A]{\fnms{Xin} \snm{Bing}  \ead[label=e1]{xin.bing@utoronto.ca}}
			\author[B]{\fnms{Florentina} \snm{Bunea} 
				\ead[label=e2]{fb238@cornell.edu}} 
			\author[C]{\fnms{Jonathan} \snm{Niles-Weed} 
				\ead[label=e3]{jnw@cims.nyu.edu}}
			
			\runauthor{Bing, Bunea, Niles-Weed }
			
			
			\address[A]{
				Department of Statistical Sciences,
				University of Toronto,
				\printead{e1}
			}
			\address[B]{
				Department of Statistics and Data Science,
				Cornell University,
				\printead{e2}}
			\address[C]{
				{Courant Institute for Mathematical Sciences and Center for Data Science,
					New York University,}
				\printead{e3}}
		\end{aug}
		
		\begin{abstract}
			The Wasserstein distance between mixing measures has come to occupy a central place in the statistical analysis of mixture models.
			This work proposes a new canonical interpretation of this distance and provides tools to perform inference on the Wasserstein distance between mixing measures in topic models. We consider the general setting of an identifiable mixture model consisting of mixtures of distributions from a set $\cA$ equipped with an arbitrary metric $d$, and show that the Wasserstein distance between mixing measures is uniquely characterized as the most discriminative convex extension of the metric $d$ to the set of mixtures of elements of $\cA$.
			The Wasserstein distance between mixing measures has been widely used in the study of such models, but without axiomatic justification.
			Our results establish this metric to be a canonical choice. Specializing our results to topic models, we consider estimation and inference of this distance. Although upper bounds for its estimation have been recently established elsewhere, we prove the first minimax lower bounds for the estimation of the Wasserstein distance between mixing measures, in topic models, when both the mixing  weights and the mixture components need to be estimated. 
			Our second main contribution is the provision of  fully data-driven inferential tools for estimators of the  Wasserstein distance between potentially sparse mixing measures of high-dimensional discrete probability distributions on $p$ points, in the topic model context.
			These results allow us to obtain the first asymptotically valid, ready to use,  confidence intervals for the Wasserstein distance in (sparse) topic models with potentially growing ambient dimension $p$. 
		\end{abstract}
		
		
	\end{frontmatter}

	
	\section{Introduction}
	
	
	This work is devoted to estimation and inference for the Wasserstein distance between  the  mixing measures associated with mixture distributions, in  topic models.

	
	
	Our motivation stems from the general observation that  standard distances (Hellinger, Total Variation, Wasserstein) between {\it mixture distributions} do not capture the possibility that similar distributions may arise from mixing completely different mixture components,  and have therefore different {\it mixing  measures}.  Figure~\ref{crab} illustrates this possibility. 
	
	
	\begin{figure}[ht]
		\centering
		\includegraphics[width=.25\textwidth]{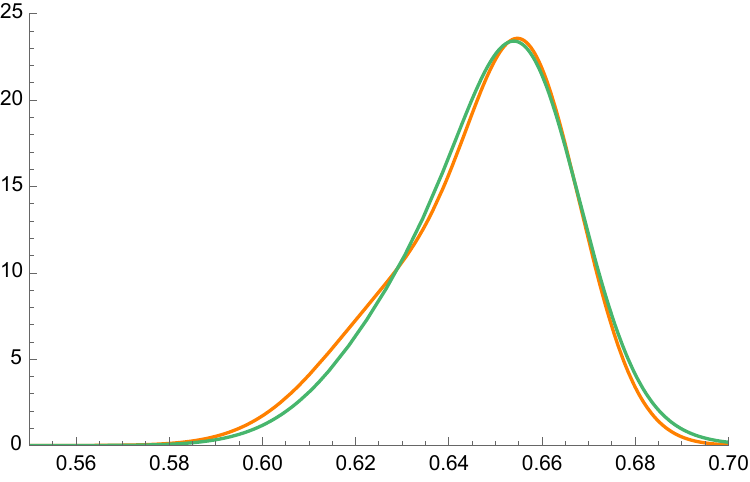}\includegraphics[width=.25\textwidth]{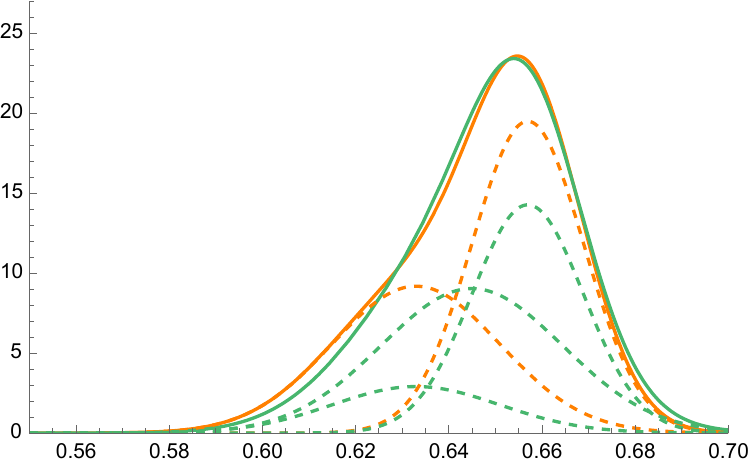}
		\caption{Pearson's crab mixture~\cite{Pearson1894}, in orange, and an alternative mixture, in green. Although these mixtures are close in total variation (left plot), they arise from very different mixture components. 
			\label{crab}}
	\end{figure}
	
	In applications of mixture models in sciences and data analysis,  discovering differences between  mixture components is crucial for model interpretability. As a result, the statistical literature on finite mixture models largely focuses on the underlying \emph{mixing measure} as the main object of interest, and compares mixtures by measuring distances between their corresponding mixing measures.
	This perspective has been adopted in the analysis of estimators for mixture models~\cite{Leroux1992,Ngu13,HeiKah18,Chen1995,McLPee00}, in Bayesian mixture modeling~\cite{MacEachern1998,Lo1984}, and in the design of model selection procedures for finite mixtures~\cite{Manole2021}.
	The Wasserstein distance is a popular choice for this comparison~\cite{Ngu13,HeiKah18}, as it is ideally suited for comparing discrete measures supported on different points of the common space on which they are defined.
	
	Our first main contribution is to give an axiomatic justification for this distance choice, for {\it any} identifiable  mixture models.

	Our main statistical contribution consists in  providing fully data driven inferential tools for the Wasserstein distance between mixing measures, in topic models. Optimal estimation of mixing measures in particular mixture models is an area of active interest (see, e.g.,~\cite{Leroux1992,Ngu13,HeiKah18,Chen1995,McLPee00,Ho2020,WuYan20}).
	In parallel, estimation and inference of distances between discrete distributions has begun to receive increasing attention~\cite{HanJiaWei20,Wang2021}.
	This work complements both strands of literature and we provide solutions to  the following open problems:  \\
	(i) Minimax lower bounds for the  estimation of the Wasserstein distance between mixing measures in  topic models, when both the mixture weights and the mixture components are estimated from the data;\\
	(ii)  Fully data driven  inferential tools  for the Wasserstein distance between topic-model based, potentially sparse,  mixing measures. 
	
	Our solution to  (i) complements recently derived upper bounds on this distance \cite{bing2021likelihood},  thereby confirming that near-minimax optimal estimation of the distance is possible,  while  the inference problem (ii), for topic models,  has not been studied elsewhere, to the best of our knowledge. We give below specific necessary background and discuss  our contributions, as well as the organization of the paper.

	\section{Our contributions}
	
	\subsection{The Sketched Wasserstein Distance (SWD) between mixture distributions equals the Wasserstein distance between their mixing measures}

	Given a set $\cA \subseteq \cP(\cY)$ of probability distributions on a Polish space $\cY$, consider the set $\cS$ of finite mixtures of elements of $\cA$.
	As long as mixtures consisting of elements of $\cA$ are identifiable, we may associate any $r^{(1)} = \sum_{k=1}^K \alpha_k^{(1)} A_k \in \cS$ with a unique \emph{mixing measure} $\balpha^{(1)} = \sum_{k=1}^K \alpha_k^{(1)} \delta_{A_k}$, which is a probability measure on $\cA \subseteq \cY$.
	
	We assume that $\cP(\cY)$ is equipped with a metric $d$.
	This may be one of the classic distances on probability measures, such as the Total Variation (TV) or Hellinger distance, or it may be a more exotic distance, for instance, one computed via a machine learning procedure or tailored to a particular data analysis task.
	As mentioned above, given mixing measures $\balpha^{(1)}$ and $\balpha^{(2)}$,  corresponding to mixture distributions $r^{(1)}$ and $r^{(2)}$, respectively, it has become common in the literature to compare them by computing their Wasserstein distance $W(\balpha^{(1)}, \balpha^{(2)}; d)$ with respect to the distance $d$ (see, e.g.,~\cite{Ngu13,HeiKah18}).
	It is intuitively clear that computing a distance between the mixing measures will indicate, in broad terms, the similarity or dissimilarity between the mixture distributions; however, it is not clear whether using the Wasserstein distance between the mixing measures best serves this goal, nor whether such a comparison has any connection to a bona fide distance between the mixture distributions themselves.
	

	As our first main contribution, we show that, indeed, the  Wassersetin distance between mixing measures is a fundamental quantity to consider and study, but we arrive at it from a different perspective. In Section~\ref{setup} we introduce a new distance, the Sketched Wasserstein Distance  ($\swd$) on $\cS$, between mixture distributions which is defined and uniquely characterized as the largest---that is, most discriminative---jointly convex metric on $\cS$ for which the inclusion $\iota: (\cA, d) \to (\cS, \swd)$ is an isometry.
	Despite the abstractness of this definition, we show, in \cref{embedding} below  that 
	\begin{equation*}
		\swd(r^{(1)}, r^{(2)}) = W(\balpha^{(1)}, \balpha^{(2)}; d)\,.
	\end{equation*}
	This fact justifies our name \emph{Sketched Wasserstein Distance} for the new metric---it is a distance on mixture distributions that reduces to a Wasserstein distance between mixing measures, which may be viewed as ``sketches'' of the mixtures. We use the results of Section~\ref{setup} as further motivation for developing inferential tools for the Wasserstein distance between mixing measures, in the specific setting provided by topic models.  We present this  below.

	\subsection{Background on topic models}\label{sec_intro_background}

	To state our results, we begin by 
	introducing the topic model, as well as notation that will be used throughout the paper. We also give in this section relevant existing estimation results for the model.  They serve as partial motivation for the results of this paper, summarized in the following sub-sections.


	Topic models are  widely used modeling tools for linking  $n$ discrete, typically high dimensional, probability vectors. 
	Their motivation and  name stems from text analysis, where they were first introduced \cite{BleiLDA}, but their applications go beyond that, for instance to biology \citep{cisTopic,Chen2020}.
	Adopting a familiar jargon, topic models are used to model a corpus of $n$ documents, each assumed to have been generated from a maximum of $K$ topics.  Each document $i \in [n] := \{1, \ldots, n\}$ is modeled as a set of $N_i$ words drawn from a  discrete distribution $r^{(i)} \in \Delta_p := \{v\in \RR_+^p: v^\T \1_p = 1\}$ supported on the set of words  $\mathcal{Y} := \{y_1, \ldots, y_p\}$, where $p$  is the dictionary  size. 
	One observes $n$ independent, $p$-dimensional word counts $Y^{(i)}$ for $i \in [n]$ and assumes that 
	\begin{equation} \label{multinomial} 
		Y^{(i)} \sim \text{Multinomial}_p(N_i, r^{(i)}).
	\end{equation} 
	For future reference, we let $X^{(i)}  := {Y^{(i)} / N_i}$ denote the associated word frequency vector.
	
	The topic model assumption is that the $p \times n$ matrix $R$, with columns  $r^{(i)}$, the expected word frequencies of each document, can be factorized as 
	$ R = A T$, 
	where the $p \times K $ matrix $A := [A_{1}, \ldots, A_{K}]$ has columns $A_{k}$ in the probability simplex $\Delta_p$, corresponding to the conditional probabilities of word occurrences, given a topic. The $K \times n$ matrix $T  := [\alpha^{(1)}, \ldots,\alpha^{(n)}]$ collects probability vectors $\alpha^{(i)} \in \Delta_K$
	with entries corresponding to the probability with which each topic $k \in [K]$ is covered by document $i$, for each $i \in [n]$. The only assumption  made in this factorization is that $A$ is common to the corpus, and  therefore document independent: otherwise the matrix factorization above always holds,  by a standard application of Bayes's theorem. Writing out the matrix factorization one column $i \in [n]$ at a time  \begin{equation}\label{topic-basic}  
		r^{(i)}  = \sum_{k = 1}^{K} \alpha^{(i)}_k A_{k}, \quad \text{with }  \alpha^{(i)}\in \Delta_K \text{ and }  A_{k} \in \Delta_p \text{ for all }k\in [K],
	\end{equation}  
	one recognizes that each (word) probability vector $r^{(i)} \in \Delta_p$ is a discrete mixture of the (conditional) probability vectors $A_{k} \in \Delta_p$, with potentially sparse mixture weights $\alpha^{(i)} $, as not all topics are expected to be covered by a single document $i$.
	Each probability vector $r^{(i)}$ can be written uniquely as in (\ref{topic-basic}) if uniqueness of the topic model factorization holds.
	The latter has been the subject of extensive study and is closely connected to the uniqueness of non-negative matrix factorization. Perhaps the most popular  identifiability assumption for  topic models, owing in part to its interpretability and constructive nature, is the so called anchor word assumption \cite{DonohoStodden,arora2012learning}, coupled with the assumption that the matrix $T\in\RR^{K\times n}$ has rank $K$. We state these assumptions formally in Section \ref{sec_background_topic_models}. 
	For the remainder of the paper, we assume that  we can write $r^{(i)}$ uniquely as in (\ref{topic-basic}), and refer to $\alpha^{(i)}$ as the true mixture weights, and to $A_{k}$, $k \in [K]$, as the mixture components. Then the collection of mixture components defined in the previous section  becomes $\mathcal{A} = \{ A_1, \dots, A_K \}$ and the mixing measure giving rise to  $r^{(i)}$ is 
	$$
	\balpha^{(i)} = \sum_{k=1}^K \alpha_{k}^{(i)} \delta_{A_{k}},
	$$ 
	which we view as a probability measure on the metric space $\cX := (\Delta_p, d)$ obtained by equipping $\Delta_p$ with a metric $d$. 
	
	Computationally tractable estimators $\wh A$ of the mixture components collected in $A$ have been proposed by \cite{arora2012learning,arora2013practical,Tracy,rechetNMF,bing2020fast,bing2020optimal,wu2022sparse,klopp2021assigning}, with finite sample performance, including minimax adaptive rates relative to  $\|\wh A - A\|_{1,\i} = \max_{1\le k\le K}\|\wh A_k - A_k\|_1$ and $\|\wh A - A\|_1 = \sum_{k=1}^K\sum_{j=1}^p|\wh A_{jk} - A_{jk}|$ obtained by \cite{Tracy,bing2020fast,bing2020optimal}.  In  \cref{sec_background_topic_models} below we give further details  on these results.
	
	In contrast, minimax optimal, sparsity-adaptive estimation of the potentially sparse mixture weights has only been studied very recently in \cite{bing2021likelihood}.  Given an estimator $\wh A$ with $\wh A_k \in \Delta_p$ for all $k\in [K]$, this work  advocates the usage of an MLE-type estimator  $\wh \alpha^{(i)}$, 
	\begin{equation}\label{MLE} 
		\wh \alpha^{(i)}=  \argmax_{\alpha \in \Delta_K} ~ N_i\sum_{j=1}^p X_j^{(i)} \log (\wh A_{j\cdot}^\T \alpha),\quad 
		\text{for each  $i \in [n]$}.
	\end{equation}
	If $\wh A$ were treated as known, and replaced by  $A$, the estimator in \eqref{MLE} reduces to  $\wh \alpha_A^{(i)}$, the MLE  of $\alpha^{(i)}$ under the multinomial model \eqref{multinomial}.
	
	Despite being of MLE-type, the analysis of   $\wh \alpha^{(i)}$, for $i\in [n]$, is non-standard, and  \cite{bing2021likelihood}  conducted it under the general regime that we will also adopt in this work. We state it below,  to facilitate  contrast with the  classical set-up.   

	\begin{table}[H]
		\centering
		\renewcommand{\arraystretch}{1.3}{
				\begin{tabular}{lll}
					{\bf Classical Regime} &  & {\bf General Regime}  \\\hline
					(a)     $A$ is known     &  &    (a')   $A$ is estimated from $n$ samples of size $N_i$ each  \\
					(b)  $p$ is independent of $N_i$ &  & (b')  $p$ can grow with $N_i$ (and $n$). \\
					(c)   $ 0 < \alpha_{k}^{(i)} < 1$, for all $k \in [K]$ &  &  (c') $\alpha^{(i)}$ can be sparse.\\
					(d)  $r_{j}^{(i)} > 0$, for all $j \in [p]$ & &  (d')  $r^{(i)}$ can be sparse.\\
					\hline
				\end{tabular}
			}
	\end{table} 
	
	The General Regime stated above is natural in the topic model context.  The fact that $A$ must be estimated from the data is embedded in the definition of the topic model. Also, although one can work with a fixed dictionary size, and fixed number of topics, one can expect  both to grow as more documents are added to the corpus. As already mentioned, each document  in a  corpus will typically cover only a few topics, 
	so $\alpha^{(i)}$ is typically sparse. When $A$ is also sparse, 
	as it is generally the case \cite{bing2020optimal}, then $r^{(i)}$ can also  be sparse.  To see this,  take $K = 2$ and write 
	$r^{(i)}_1 = \alpha^{(i)}_1A_{11} +\alpha^{(i)}_2A_{12}$. If document $i$ covers topic 1  ($\alpha^{(i)}_1
	> 0$), but not topic 2 ($\alpha^{(i)}_2
	= 0$), and  the probability of using  word 1, given that topic 1 is covered, is zero ($A_{11} = 0$), then $r^{(i)}_1 = 0$.

	
	Under the General Regime above, \cite[Theorem 10]{bing2021likelihood} shows that, whenever $\alpha^{(i)}$ is sparse,   the estimator  $\wh \alpha^{(i)}$ is also sparse, with high probability, and derive the  sparsity adaptive finite-sample rates for  $\|\wh \alpha^{(i)} -   \alpha^{(i)}\|_1$ summarized in \cref{sec_background_topic_models} below. Furthermore, \cite{bing2021likelihood} suggests estimating  $W(\balpha^{(i)}, \balpha^{(j)}; d )$ by  $W( \wh \balpha^{(i)}, \wh \balpha^{(j)}; d)$, with 
	\begin{equation}\label{mm_est_ij}
		\wh \balpha^{(i)} = \sum_{k=1}^{K} \wh \alpha_{k}^{(i)} \delta_{\wh A_k},\quad \quad \wh \balpha^{(j)} = \sum_{k=1}^{K} \wh \alpha_{k}^{(j)} \delta_{\wh A_k},
	\end{equation}
	for given  estimators $\wh A$, $\wh \alpha^{(i)}$ and $\wh \alpha^{(j)}$ of their population-level counterparts, and $i, j \in [n]$. Upper bounds for these distance estimates are derived \cite{bing2021likelihood}, and re-stated in a  slightly more general form in \cref{thm_upper_bound_SWD} of \cref{sec_upper_bound_est}.
	The  work of \cite{bing2021likelihood}, however, does not study the optimality of their upper bound, and motivates our next contribution. 

	\subsection{Lower bounds for the estimation of the Wasserstein distance between mixing measures  in topic models} 
	
	In \cref{thm_lower_bound_SWD} of \cref{sec_lower_bound_est},  we establish  a minimax optimal bound for estimating $W(\balpha^{(i)}, \balpha^{(j)}; d)$, for any pair $i, j \in [n]$, from a collection of $n$ independent multinomial vectors whose collective cell probabilities satisfy the topic model assumption. Concretely, for any $d$ on $\Delta_p$ that is bi-Lipschitz equivalent to the TV distance (c.f. \cref{lip_d}), we have 
	\[
	\inf_{\wh W} \sup_{\balpha^{(i)}, \balpha^{(j)}} \EE |\wh W  - W(\balpha^{(i)},\balpha^{(j)};d)| 
	\gtrsim   \sqrt{\tau  \over N\log^2(\tau)} +  \sqrt{pK \over nN\log^2(p)}
	\]
	where $\tau = \max\{\|\alpha^{(i)}\|_0,\|\alpha^{(j)}\|_0\}$, and $\|\cdot\|_0$ denotes the $\ell_0$ norm of a vector, counting the number of its non-zero elements, For simplicity of presentation we assume that $N_i = N_j = N$, although this is not needed for either theory or practice. The two terms in the above lower bound quantify, respectively, 
	the smallest error of estimating the distance when the mixture components are known, and the smallest error when the mixture weights are known.
	
	We also obtain a slightly sharper lower bound for the estimation of the mixing measures $\balpha^{(i)}$ and $\balpha^{(j)}$ themselves in the Wasserstein distance:
	\begin{equation*}
		\inf_{\wh \balpha} \sup_{\balpha} \, \EE W(\wh \balpha,\balpha;d)| 
		\gtrsim   \sqrt{\tau  \over N} +  \sqrt{pK \over nN}\,.
	\end{equation*}
	As above, these two terms reflect the inherent difficulty of estimating the mixture weights and mixture components, respectively.
	As we discuss in \cref{sec_lower_bound_est}, a version of this lower bound with additional logarithmic factors follows directly from the lower bound presented above for the estimation of $W(\balpha^{(i)},\balpha^{(j)};d)$.
	For completeness, we give the simple argument leading to this improved bound in \cref{sec_upper_bound_mm_est} of \cite{supplement}.
	
	
	The upper bounds obtained in \cite{bing2021likelihood} match the above  lower bounds, up to possibly sub-optimal logarithmic terms.
	Together, these results show that the plug-in approach to estimating the Wasserstein distance is nearly optimal, and complement a rich literature on the performance of plug-in estimators for non-smooth functional estimation tasks~\cite{LepNemSpo99,CaiLow11}.
	Our lower bounds complement those recently obtained for the minimax-optimal estimation of the TV distance for discrete measures~\cite{Jiao2018}, but apply to a wider class of distance measures and differ by an additional logarithmic factor; whether this logarithmic factor is necessary is an interesting open question. 
	The results we present are the first nearly tight lower bounds for estimating the class of distances we consider, and our proofs require developing new techniques beyond those used for the total variation case.
	More details can be found in \cref{sec_lower_bound_est}.
	{ The results of this section are valid for any $K, p, n,  N$. In particular,  the ambient dimension $p$, and the number of mixture components $K$ are allowed to grow with the sample sizes $n, N$. }

	\subsection{Inference for the  Wasserstein distance between sparse mixing measures in  topic models}
	
	
	To the best of our knowledge, despite an active interest in  inference for  the Wasserstein  distance in general~\cite{ManBalNil21,BarGorLou19,BarLou19,BarGinMat99,del2024central,del2024centralgeneral} and the Wasserstein distance between discrete measures, in particular \cite{sommerfeld2018inference,tameling2019empirical}, inference for the Wasserstein between  sparse discrete mixing measures, in general, and in particular in topic models, has not been studied.
	
	Our main statistical contribution is the provision of inferential tools  for $ W(\balpha^{(i)}, \balpha^{(j)}; d)$, for any pair $i, j\in [n]$, in topic models.  To this end: 
	\begin{enumerate}
		\item[(1)] In \cref{sec_LT_SWD}, \cref{thm_limit_distr},  we derive   a $\sqrt{N}$  distributional limit for an appropriate estimator of the distance. 
		\item[(2)]  In \cref{sec_est_LT}, \cref{thm_cdf}, we estimate consistently the parameters of the limit, thereby providing  theoretically justified, fully data driven,  estimates for inference.  We also contrast our suggested approach with a number of possible bootstrap schemes, in our simulation study of \cref{sec_sim_CI} in \cite{supplement}, revealing the benefits of our method. 
	\end{enumerate}

	Our approach  to (1) is given in Section \ref{sec_inference}. In Proposition \ref{LT_SWD}
	we show that, as soon as  that error in estimating  $A$
	can be appropriately controlled in probability, for instance by employing the estimator given  in  \cref{sec_background_topic_models}, the desired limit of a distance estimator hinges on inference for the potentially sparse mixture weights $\alpha^{(i)}$.  Perhaps surprisingly, this problem has not been studied elsewhere,  under the  General Regime introduced in \cref{sec_intro_background} above. Its  detailed treatment is of interest in its own right, and  constitutes one of our main contributions to the general problem  of inference for the Wasserstein distance between mixing measures in topic models. 
	{Since our asymptotic distribution results for the distance estimates rest on determining the distributional limit of  $K$-dimensional mixture weight estimators, the  asymptotic analysis  will be conducted for  $K$ fixed. However, we do allow the ambient dimension $p$ to depend on the sample sizes $n$ and $N$.} 
	
	\subsubsection{Asymptotically normal estimators for  sparse mixture weights in topic models, under general conditions}  
	
	We construct estimators  of the sparse mixture weights that admit a $\sqrt{N}$ Gaussian limit under the {General Regime},  and are asymptotically efficient under the {Classical Regime}. The challenges associated with this problem under the {General Regime} inform our estimation strategy, and can be best seen in contrast to the existing solutions offered in the {Classical Regime}.  In what follows we concentrate on inference of $\alpha^{(i)}$ for some arbitrarily picked $i\in [n]$. 

	
	In the {Classical Regime}, 
	the estimation and inference for the MLE   of   $\alpha^{(i)}$, in a low dimensional parametrization $r^{(i)} = g(\alpha^{(i)})$ of   a multinomial probability vector $r^{(i)}$, when $g$ is a {\it known} function and $\alpha^{(i)}$ lies in an open set of $\RR^K$, with $K < p$, have been studied for over eighty years. Early results go back to  \cite{Rao55,Rao58,birch}, and are reproduced in updated forms in  standard text books, for instance Section 16.2 in \cite{Agresti2012}.  Appealing to these results  for $g(\alpha^{(i)}) = A\alpha^{(i)}$, with $A$ known, the asymptotic distribution of  the  MLE $\wh \alpha_A^{(i)}$ of $\alpha^{(i)}$  under the Classical Regime  is given by
	\begin{equation}\label{goal_classical} 
		\sqrt{N}    (\wh \alpha_{A}^{(i)} - \alpha^{(i)}) \overset{d}{\to} \cN_K(0, \Gamma^{(i)}),\quad \text{as }N\to \i,
	\end{equation} 
	where the asymptotic covariance matrix is 
	\begin{equation}\label{def_Gamma_i}
		\Gamma^{(i)}=
		\Bigl(\sum_{j=1}^p {A_{j\cdot}A_{j\cdot}^\T \over r_j^{(i)}}\Bigr)^{-1} - \alpha^{(i)}\alpha^{(i)\T},
	\end{equation} 
	a $K \times K$ matrix of rank $K - 1$. 
	By standard MLE theory, under the Classical Regime,   any $(K-1)$ sub-vector of $\wh \alpha_A^{(i)}$ is asymptotically efficient,  for any interior point of $\Delta_{K}$.
	

	However,  even if conditions (a) and (b) of the {Classical Regime} hold, but conditions (c) and (d) are violated, in that the model parameters are on the boundary of their respective simplices,  the standard MLE theory, which is devoted to the analysis of estimators  of interior points, no longer applies.  In general,  an estimator $\bar{\alpha}^{(i)}$, restricted to  lie in $\Delta_K$,   cannot be expected to have a Gaussian limit around a boundary point $\alpha^{(i)}$ (see, e.g., \cite{Andrews2002} for a  review of  inference for boundary points).
	
	Since inference for  the sparse  mixture weights is only an intermediate step towards inference on the Wasserstein distance between mixing distributions, we aim at estimators of  the weights  that have the simplest possible limiting distribution, Gaussian in this case.  Our solution is given in \cref{cor_asn_fixed} of \cref{sec_ASN}. We show, in the General Regime, that an appropriate  one-step update  $\wt \alpha^{(i)}$, that removes the asymptotic bias of  the (consistent)  MLE-type estimator $\wh \alpha^{(i)}$ given in (\ref{MLE}) above,  is indeed asymptotically normal: 
	\begin{equation}\label{goal} 
		\sqrt{N} ~ \bigl(\Sigma^{(i)}\bigr)^{+\frac12}   \bigl(\wt \alpha^{(i)} - \alpha^{(i)}\bigr) \overset{d}{\to} \cN_K\left(0, \begin{bmatrix}
			\bI_{K-1} & 0 \\ 0 & 0
		\end{bmatrix}\right),\quad \text{as }n,N\to \i,
	\end{equation} 
	where the asymptotic covariance matrix is 
	\begin{equation}\label{def_Sigma_i}
		\Sigma^{(i)}=
		\Bigl(\sum_{j\in \oJ} {A_{j\cdot}A_{j\cdot}^\T \over r_j^{(i)}}\Bigr)^{-1} - \alpha^{(i)}\alpha^{(i)\T}
	\end{equation} 
	with $\oJ = \{j\in [p]: r_j^{(i)}>0$\} and $(\Sigma^{(i)})^{+\frac12}$ being the square root of its generalized inverse.
	We note  although the dimension $K$ of  $\Sigma^{(i)}$ is fixed, its entries are allowed to grow with $p, n$ and $N$, under the General Regime. 
	
	We benchmark the behavior of our proposed estimator against that of the MLE, in the Classical Regime. We observe that, under the  Classical Regime,  we have $\oJ = [p]$,  and $\Sigma^{(i)}$ reduces to $ \Gamma^{(i)}$, the limiting covariance matrix of the MLE, given by \eqref{def_Gamma_i}. 
	This shows that, in this regime,  our proposed estimator $\wt \alpha^{(i)}$ enjoys the same asymptotic efficiency properties as the MLE.
	


	In \cref{rem_LS} of \cref{sec_ASN} we discuss other natural estimation possibilities. It turns out that a weighted least squares estimator of the mixture weights is also asymptotically normal, however it is not asymptotically efficient in the Classical Regime (cf. Theorem \ref{thm_limit_distr_LS} in Appendix \ref{app_least_squares}, where we show that the covariance matrix of the limiting distribution of the weighted least squares does not equal $\Gamma^{(i)}$). As a consequence, the  simulation results in  \cref{sec_sim_MLE_LS} show that the resulting distance confidence intervals are slightly wider than those corresponding to our proposed estimator.  
	For completeness, we do however  include a full theoretical treatment of this  natural estimator in \cref{app_least_squares}.

	\subsubsection{The limiting distribution of the distance estimates and  consistent quantile estimation} 
	Having constructed estimators of the mixture weights admitting $\sqrt{N}$ Gaussian limits, we obtain as a consequence the limiting distribution of our proposed estimator for $W(\balpha^{(i)}, \balpha^{(j)}; d)$.
	The limiting distribution is non-Gaussian and involves an auxiliary optimization problem based on the Kantorovich--Rubinstein dual formulation of the Wasserstein distance: namely, under suitable conditions we have
	\begin{equation*}
		\sqrt N(\widetilde{W} - W(\balpha^{(i)}, \balpha^{(j)}; d)) \overset{d}{\to} \sup_{f \in \cF'_{ij}} f^\top Z_{ij}\,,\quad \text{as }n,N\to \i,
	\end{equation*}
	where $Z_{ij}$ is a Gaussian vector and $\cF'_{ij}$ is a certain polytope in $\RR^K$.
	The covariance of $Z_{ij}$ and the polytope $\cF'_{ij}$ depend on the parameters of the problem. 
	
	To use this theorem for practical inference, we provide a consistent method of estimating the limit on the right-hand side, and thereby obtain the first theoretically justified, fully data-driven confidence intervals for $W(\balpha^{(i)}, \balpha^{(j)}; d)$. We refer to Theorem \ref{thm_cdf} and  Section \ref{sec_est_LT_plugin} for details. 
	
	Furthermore, a detailed simulation study presented in \cref{app_sim} of \cite{supplement}, shows that, numerically, (i) our method compares favorably with the  $m$-out-of-$N$ bootstrap~\cite{Dumbgen1993}, especially in small samples, while also avoiding the delicate choice of $m$; (ii)  our method has comparable performance to the much more computationally expensive derivative-based  bootstrap~\cite{fang2019inference}.  This investigation, combined with our theoretical guarantees,  provides strong support in favor of the  potential of our proposal for inference problems involving the Wasserstein distance between potentially sparse mixing measures in large topic models.

	\section{A new distance for mixture distributions} \label{setup}
	In this section, we begin by focusing on a \emph{given} subset $\cA = \{A_1, \dots, A_K\}$ of probability measures on $\cY$, and consider mixtures relative to this subset.
	The following assumption guarantees that, for a mixture $r = \sum_{k=1}^K \alpha_k A_k$, the mixing weights $\alpha = (\alpha_1, \dots, \alpha_k)$ are identifiable.
	\begin{assumption}\label{ass_id}
		The set $\cA$ is affinely independent. 
	\end{assumption}
	As discussed in \cref{sec_intro_background} above, in the context of topic models, it is necessary to simultaneously identify the mixing weights $\alpha$ and the mixture components $A_1, \dots, A_K$; for this reason, the conditions for identifiability of topic models (\cref{ass_anchor} in \cref{sec_background_topic_models}) is necessarily \emph{stronger} than \cref{ass_id}.
	However, the considerations in this section apply more generally to finite mixtures of distributions on an arbitrary Polish space.
	In the following sections, we will apply the general theory developed here to the particular case of topic models, where the set $\cA$ is unknown and has to be estimated along with the mixing weights.
	
		%
	
	We assume that $\cP(\cY)$ is equipped with a metric $d$.
	Our goal is to define a metric based on $d$ and adapted to $\cA$ on the space $\cS = \conv(\cA)$, the set of mixtures arising from the mixing components $A_1, \dots, A_K$.
	That is to say, the metric should retain geometric features of the original metric $d$, while also measuring distances between elements of $\cS$ relative to their unique representation as convex combinations of elements of $\cA$.
	Specifically, to define a new distance $\swd(\cdot, \cdot)$ between elements of $\cS$, we propose that it satisfies two desiderata:
	
	%
	%

	\begin{enumerate}
		\item[\textbf{R1}:] The function $\swd$ should be a jointly convex function of its two arguments. 
		\item[\textbf{R2}:] The function $\swd$ should agree with the original distance $d$ on the mixture components  i.e., $\swd(A_k, A_\ell) = d(A_k, A_\ell)$ for $k, \ell \in [K]$. 
	\end{enumerate}
	We do \emph{not} specifically impose the requirement that $\swd$ be a metric, since many useful measure of distance between probability distributions, such as the Kullback--Leibler divergence, do not possess this property; nevertheless, in Corollary~\ref{metric}, we show that our construction does indeed give rise to a metric on $\cS$.
	
	\textbf{R1} is motivated by both mathematical and practical considerations.
	Mathematically speaking, this property is enjoyed by a wide class of well behaved metrics, such as those arising from norms on vector spaces, and, in the specific context of probability distributions, holds for any $f$-divergence.
	Metric spaces with jointly convex metrics are known as Busemann spaces, and possess many useful analytical and geometrical properties~\cite{Kirk2014,Takahashi1970}.
	Practically speaking, convexity is crucial for both computational and statistical purposes, as it can imply, for instance, that optimization problems involving $\swd$ are computationally tractable and that minimum-distance estimators using $\swd$ exist almost surely.
	\textbf{R2} says that the distance should accurately reproduce the original distance $d$ when restricted to the original mixture components, and therefore that $\swd$ should capture the geometric structure induced by $d$.
	
	To identify a unique function satisfying \textbf{R1} and \textbf{R2}, note that the set of such functions is closed under taking pointwise suprema.
	Indeed, the supremum of convex functions is convex~\cite[see, e.g.,][Proposition IV.2.1.2]{HirLem93}, and given any set of functions satisfying \textbf{R2}, their supremum clearly satisfies \textbf{R2} as well. 
	For $r^{(1)},r^{(2)}\in \cS$, we therefore define $\swd$ by
	\begin{equation}\label{wa_def}
		\swd(r^{(1)},r^{(2)}) \defeq \sup_{\phi(\cdot, \cdot): \text{ $\phi$ satisfies \textbf{R1}\&\textbf{R2}}} \phi(r^{(1)},r^{(2)})\,.
	\end{equation}
	This function is the largest---most discriminative---quantity satisfying both \textbf{R1} and \textbf{R2}.

	Under \cref{ass_id}, we can uniquely associate a mixture $r = \sum_{k=1}^K \alpha_k A_k$ with the mixing measure $\balpha = \sum_{k=1}^K \alpha_k \delta_{A_k}$, which is a probability measure on the metric space $(\cP(\cY), d)$.
	Our first main result shows that $\swd$ agrees precisely with the Wasserstein distance on this space.
	%
	%
	%


	\begin{theorem}\label{embedding}
		Under Assumption~\ref{ass_id}, any all $r^{(1)}, r^{(2)} \in \cS$,
		\begin{align}\nonumber
			\swd(r^{(1)}, r^{(2)}) & =  W(\balpha^{(1)}, \balpha^{(2)}; d) 
		\end{align}
		where $W(\balpha^{(1)}, \balpha^{(2)}; d)$ denotes the Wasserstein distance between the mixing measures $\balpha^{(1)}, \balpha^{(2)}$ corresponding to $r^{(1)}, r^{(2)}$, respectively.
	\end{theorem}
	\begin{proof}
		The proof can be found in \cref{app_proof_embedding}.
	\end{proof}

	%
	
	Theorem~\ref{embedding} reveals a surprising characterization of the Wasserstein distances for mixture models.
	On a convex set of measures, the Wasserstein distance is uniquely specified as the largest jointly convex function taking prescribed values at the extreme points.
	This characterization gives an axiomatic justification of the Wasserstein distance for statistical applications involving mixtures.
	
	The study of the relationship between the Wasserstein distance on the mixing measures and the classical probability distances on the mixtures themselves was inaugurated by \cite{Ngu13}.
	Clearly, if the Wasserstein distance between the mixing measures is zero, then  the mixture distributions agree as well, but the converse is not generally true.
	This line of research has therefore sought conditions under which the Wasserstein distance on the mixing measures is comparable with a classical probability distance on the mixtures.
	In the context of mixtures of continuous densities on $\RR^d$, \cite{Ngu13} showed that a strong identifiability condition guarantees that the total variation distance between the mixtures is bounded below by the squared Wasserstein distance between the mixing measures, as long as the number of atoms in the mixing measures is bounded.
	Subsequent refinements of this result (e.g.,~\cite{Ho2016,Ho2016a}) obtained related comparison inequalities under weaker identifiability conditions.\\

	This work adopts a different approach, viewing the Wasserstein distance between mixing measures as a bona fide metric, the Sketched Wasserstein Distance, on the mixture distributions  themselves.
	
	Our Theorem~\ref{embedding} can also be viewed as an extension of a similar result in the nonlinear Banach space literature: Weaver's theorem~\cite[Theorem 3.3 and Lemma 3.5]{Weaver2018} identifies the norm defining the \emph{Arens--Eells space}, a Banach space into which the Wasserstein space isometrically embeds, as the largest seminorm on the space of finitely supported signed measures which reproduces the Wasserstein distance on pairs of Dirac measures.
	
	\begin{corollary}\label{metric}
		Under \cref{ass_id},   $(\cS, \swd)$ is a complete metric space.
	\end{corollary}
	The proof appears in Appendix~\ref{app_proof_metric}. 
	
	\begin{remark}
		If we work with models that are not affinely independent, we pay a price: the quantity defined by (\ref{wa_def}) is no longer a metric. This happens even if   Assumption \ref{ass_id} is only  slightly relaxed to   
		\begin{assumption}\label{extreme_points}
			For each $k \in [K]$, $A_k \notin \conv(\cA \setminus A_k)$.
		\end{assumption}
		The proof of \cref{embedding} shows that under Assumption \ref{extreme_points} we have,  for all $r^{(1)}, r^{(2)}  \in \cS$
		\begin{align} \label{d_equivalence}
			\swd(r^{(1)}, r^{(2)}) & =  \inf_{\balpha^{(1)}, \balpha^{(2)} \in \Delta_K: r^{(i)} = A \alpha^{(i)},\,\, i \in \{1, 2\} } W(\balpha^{(1)}, \balpha^{(2)} ; d)
		\end{align}
		However, $\swd$ is only a semi-metric in this case: it is symmetric, and $\swd(r^{(1)}, r^{(2)}) = 0$  iff $r^{(1)} = r^{(2)}$, but it no longer satisfies the triangle inequality in general, as we show by example in Appendix~\ref{triangle_counterexample}.
	\end{remark}

	\section{Near-Optimal estimation of the Wasserstein distance between  mixing measures in  topic models} \label{sec:discrete_swd}  
	
	Finite sample upper bounds for estimators of  $W(\balpha^{(i)}, \balpha^{(j)}; d)$, when $d$ is either the total variation distance or the Wasserstein  distance on discrete probability distributions supported on $p$ points have been recently established in \cite{bing2021likelihood}.  In this section we  investigate the optimality of these bounds, by deriving minimax lower bounds for estimators of this distance, a problem  not previously studied in the literature. 
	
	For clarity of exposition, we  begin with a review on estimation in topic models  in \cref{sec_background_topic_models}, state existing  upper bounds on estimation in \cref{sec_upper_bound_est}, then prove the lower bound in \cref{sec_lower_bound_est}.

	\subsection{Review of  topic model identifiability conditions  and of related  estimation results}\label{sec_background_topic_models}

	In this section, we begin by reviewing (i) model identifiability and the minimax optimal estimation of the mixing components $A$, and (ii)  estimation and finite sample error bounds for  the mixture weights $\alpha^{(i)}$ for $i\in [n]$. Necessary notation, used in the remainder of the paper,  is also introduced here.

	\subsubsection{Identifiability and estimation of $A$}    
	Model identifiability in topic models is a sub-problem   of the  larger  class devoted to the study of the existence and uniqueness of non-negative matrix factorization. The latter has been studied extensively in the past two decades. One popular set of identifiability conditions involve 
	the following assumption on $A$: 
	\begin{assumption}\label{ass_anchor}
		For each $k\in [K]$, there exists at least one $j\in [p]$ such that $A_{jk} > 0$ and $A_{jk'}=0$ for all $k'\ne k$.
	\end{assumption}
	It is   informally referred to as the anchor word assumption, as it postulates that every topic $k \in [K]$ has at least one word $j \in [p]$ solely associated with it. 
	\cref{ass_anchor} in conjunction with $\rank(T) = K$ ensures that both $A$ and $T$ are identifiable \cite{DonohoStodden}, therefore guaranteeing  the full model identifiability. 
	In particular, \cref{ass_anchor} implies \cref{ass_id}.
	\cref{ass_anchor}
	is sufficient, but not necessary, for model identifiability, see, e.g., \cite{Fu19}. Nevertheless, all known provably fast algorithms make constructive usage of it.  Furthermore,  \cref{ass_anchor} improves 
	the interpretability of the matrix factorization, and the extensive empirical study in \cite{Ding2015}
	shows that it is supported  by a wide range of data sets where topic models can be employed.

	We therefore adopt this condition in our work.    Since the paper \cite{BleiLDA}, estimation of $A$ has been studied extensively from the Bayesian perspective (see, \cite{blei-intro}, for a comprehensive review of this class of techniques). A recent list of papers \cite{arora2012learning,arora2013practical,bansal2014provable,Tracy,bing2020fast,bing2020optimal,klopp2021assigning,wu2022sparse} have proposed computationally fast algorithms of estimating $A$,  from a frequentist point of view, under  
	\cref{ass_anchor}. Furthermore, minimax optimal estimation of $A$ under various metrics has also been well understood \cite{Tracy,bing2020fast,bing2020optimal,klopp2021assigning,wu2022sparse}. In particular, under  \cref{ass_anchor}, by letting $m$ denote the number of words $j$ that satisfy \cref{ass_anchor}, the work of  \cite{bing2020fast} has established the following minimax lower bound 
	\begin{equation}\label{disp_lower_bound_A}
		\inf_{\wh A} \sup_{A\in \Theta_A} \EE 
		\|\wh A - A\|_{1,\infty} \gtrsim  \sqrt{pK\over nN},
	\end{equation}    
	for all estimators $\wh A$,  over the parameter space 
	\begin{align}\label{def_Theta_A}
		\Theta_A  = \{A:  A_k \in \Delta_p,   \forall  k\in [K], ~ \min_{j\in[p]}\|A_{j\cdot}\|_\i \ge c'/p, ~ \textrm{\cref{ass_anchor} holds with }  m\le c ~ p\}
	\end{align}
	with $c\in (0,1)$, $c'>0$ being some absolute constant. Above  $N_i = N$ for all $i\in [n]$ is assumed for convenience of notation, and will be adopted in this paper as well, while noting that all the related theoretical results  continue to hold in general.

	The estimator $\wh A$ of $A$ proposed in \cite{bing2020fast}  is shown to be minimax-rate adaptive, and computationally efficient.
	Under the conditions collected in \cite[Appendix K.1]{bing2021likelihood}), this estimator $\wh A$ satisfies 
	\begin{equation}\label{rate_A}
		\EE \|\wh A - A\|_{1,\infty} \lesssim  \sqrt{pK\log(L)\over nN}, 
	\end{equation}
	and is therefore   minimax optimal up to a logarithmic factor of $L := n\vee p\vee N$. 
	This $\log(L)$ factor comes from a union bound over the whole corpus and all words.
	Under the same set of conditions, we also have the following normalized sup-norm control (see, \cite[Theorem K.1]{bing2021likelihood})   
	\begin{equation}\label{rate_A_sup}
		\EE \left[\max_{j\in [p]}{\|\wh A_{j\cdot}  - A_{j\cdot} \|_\i \over \|A_{j\cdot}\|_\i}\right]  \lesssim \sqrt{pK\log(L)\over nN}.
	\end{equation}

We will use this estimator of $A$ as our running example for the remainder of this paper. For the convenience of the reader, Section \ref{Aconstruct} of \cite{supplement} gives its detailed construction.  
All our theoretical results are still applicable when $\wh A$ is replaced by any other minimax-optimal 
estimator such as \cite{Tracy,bing2020optimal,klopp2021assigning,wu2022sparse} coupled with a sup-norm control as in \eqref{rate_A_sup}.

\subsubsection{Mixture weight estimation}\label{sec_est_weights}
We review existing finite sample error bounds for  the MLE-type estimator $\wh \alpha^{(i)}$ defined above in \eqref{MLE},  for $i \in [n]$, arbitrary fixed. Recall that   $r^{(i)} = A\alpha^{(i)}$. The estimator $\wh \alpha^{(i)}$ is valid and  analyzed  for a generic $\wh A$  in \cite{bing2021likelihood},  and the results pertain, in particular,  to all estimators mentioned in \cref{sec_background_topic_models}. For brevity, we only state below the results  for  an estimator $\wh \alpha^{(i)}$ calculated relative to the estimator $\wh A$ in \cite{bing2020fast}.





The conditions under which $\|\wh \alpha^{(i)} - \alpha^{(i)}\|_1$ is analyzed are lengthy, and fully explained in \cite{bing2021likelihood}. To improve readability, while keeping this paper self-contained, we opted 
for stating them in  \cref{app_sec_alpha}, and we introduce here only the minimal amount of notation that will be used first in \cref{sec_lower_bound_est}, and also in  the remainder of  the paper. 

Let $1\le \tau\le K$ be any integer and consider the following parameter space associated with  the sparse mixture weights,
\begin{equation}\label{def_space_alpha}
	\Theta_\alpha(\tau) = \left\{
	\alpha \in \Delta_K: \|\alpha\|_0 \le \tau
	\right\}.
\end{equation}
For  $r  = A\alpha $ with any $\alpha  \in \Theta_\alpha(\tau)$, let  $X$ denote  the  vector of observed frequencies  of $r$,  with $$J := \supp(X) = \{ j \in [p]: X_j > 0\}. $$ Define
\begin{equation}\label{def_oJ}
	\oJ  := \left\{j\in [p]: r_j  > 0\right\},\qquad \uJ := \left\{j\in [p]: r_j > {5\log(p) /  N}\right\}.
\end{equation}
These sets are chosen such that (see, for instance,  \cite[Lemma I.1]{bing2021likelihood})  
$$
\PP\left\{\uJ  \subseteq J \subseteq \oJ \right\} \ge 1-2p^{-1}.
$$ 
Let   $S_\alpha = \supp(\alpha)$ and define
\begin{equation}\label{def_xi}
	\Tm := \min_{k\in S_\alpha }  \alpha_k,\qquad \xi \coloneqq \max_{j\in \oJ} {\max_{k\notin S_\alpha } A_{jk} \over \sum_{k\in S_\alpha } A_{jk}}.
\end{equation} 
Another important quantity appearing in the results of \cite{bing2021likelihood} is the restricted $\ell_1\to\ell_1$ condition number of $A$ defined, for any integer $k\in [K]$, as
\begin{equation}\label{def_kappa_A}
	\kappa(A, k) \coloneqq \min_{S\subseteq[K]: |S|\le k}\min_{v\in \cC(S)}{\|A v\|_1 \over \|v\|_1},
\end{equation}
with
$\cC(S) \coloneqq \{ v \in \RR^{K} \setminus \{0\}: \|v_S\|_1 \ge \|v_{S^c}\|_1\}.$ Let
\begin{align}\label{def_kappas}
	\ok_\tau = \min_{\alpha \in \Theta_\alpha(\tau)}\kappa(A_{\oJ}, \|\alpha\|_0), \qquad \uk_\tau = \min_{\alpha\in \Theta_\alpha(\tau)}\kappa(A_{\uJ}, \|\alpha\|_0).
\end{align}  
For future reference, we have 
\begin{equation}\label{ineq_kappas}
	\uk_\tau \le \ok_\tau,\quad \ok_{\tau'} \le \ok_\tau, \quad  \text{for any $\tau \le \tau'$.}
\end{equation}
Under conditions in \cref{app_sec_alpha} for  $\alpha^{(i)}\in \Theta_\alpha(\tau)$ and $\wh A$, Theorems 9 \& 10 in \cite{bing2021likelihood} show that 
\begin{equation}\label{rate_alpha}
	\EE \| \wh \alpha^{(i)} - \alpha^{(i)}\|_1   ~ \lesssim  ~ {1 \over \ok_\tau}  
	\sqrt{\tau\log(L) \over N} + {1 \over  \ok_\tau^2} \sqrt{pK\log(L) \over nN}.
\end{equation}
The statement of the above result, included in \cref{app_sec_alpha}, of \cite[Theorems 9 \& 10]{bing2021likelihood} is valid for any generic estimator of $A$ that is minimax adaptive, up to a logarithmic factor of $L$. 

We note that the first term (\ref{rate_alpha}) reflects the error of estimating the mixture weights, if $A$ were known, whereas the second one is the error in estimating $A$.

\subsection{Review on finite sample upper bounds for $W(\alpha^{(i)}, \alpha^{(j)}; d)$}\label{sec_upper_bound_est}
For any estimators $\wh \alpha^{(i)}, \wh \alpha^{(j)}\in \Delta_K$ and $\wh A$, plug-in estimators $W(\wh \balpha^{(i)}, \wh \balpha^{(j)};d)$ of $W(\balpha^{(i)}, \balpha^{(j)}; d)$ with $\wh \balpha^{(i)}$ and $\wh \balpha^{(j)}$ defined in
\eqref{mm_est_ij}
can be analyzed using the fact that   
\begin{align}\label{basic_ineq}
	\EE \left|W(\wh \balpha^{(i)}, \wh \balpha^{(j)};d) -W(\balpha^{(i)}, \balpha^{(j)}; d) \right| & ~ \le ~  2\epsilon_A +  \diam(\cA)~  \epsilon_{\alpha, A}, 
\end{align}
with $\diam(\cA) \defeq \max_{k, k'} d(A_k, A_{k'})$, whenever  the following hold 
\begin{equation}\label{one}  
	\EE \left[\| \wh  \alpha^{(i)}  - \alpha^{(i)}\|_1 +  \| \wh \alpha^{(j)}  - \alpha^{(j)}\|_1\right]  \leq  2\epsilon_{\alpha, A}
\end{equation}
and  
\begin{equation}\label{two}
	\EE\left[ \max_{k\in[K]} d(\wh  A_k, A_k)\right] \le \epsilon_A,
\end{equation}
for some  deterministic sequences $\epsilon_{\alpha, A}$ and $\epsilon_A$.  The notation $\epsilon_{\alpha, A}$ is mnemonic of the fact that the estimation of mixture weights depends on the estimation of $A$.

Inequality \eqref{basic_ineq} follows by several appropriate applications of the triangle inequality and, for completeness, we give the details  in \cref{app_proof_general_SWD_rate}. We also note that \eqref{basic_ineq} holds for  $\wh\balpha^{(i)}$ based on any estimator $\wh\alpha^{(i)}\in \Delta_K$. 

When $\wh \alpha^{(i)}$ and $\wh \alpha^{(j)}$  are the  MLE-type estimators given by (\ref{MLE}), \cite{bing2021likelihood} used this strategy to bound the left hand side in (\ref{basic_ineq}) for any estimator $\wh A$,  and then for a minimax-optimal estimator of $A$, when  $d$ is either the total variation  distance or the Wasserstein distance on $\Delta_p$.

We reproduce the result  \cite[Corollary 13]{bing2021likelihood} below, but we state it  in the slightly more general framework in which the distance  $d$  on $\Delta_p$ only needs to satisfy 
\begin{equation}\label{lip_d_upper}
	2d(u, v) \le  C_d \|u - v\|_1,\qquad  \text{for any $u,v \in \Delta_p$}
\end{equation} 
for some $C_d>0$. 
This holds, for instance, if $d$ is an $\ell_q$ norm with $q\ge 1$ on $\RR^p$ (with $C_d = 2$), or the Wasserstein distance on $\Delta_p$ (with $C_d = \diam(\cY)$).  Recall that $\Theta_A$ and $\Theta_\alpha(\tau)$ are defined in \eqref{def_Theta_A} and \eqref{def_space_alpha}. Further recall $\ok_{\tau}$ from \eqref{def_kappas} and $L = n\vee p\vee N$.

\begin{theorem}[Corollary 13, \cite{bing2021likelihood}]\label{thm_upper_bound_SWD}
	Let $\wh\alpha^{(i)}$ and $\wh\alpha^{(j)}$ be obtained from \eqref{MLE} by using
	the estimator $\wh A$ in \cite{bing2020fast}.  Grant  conditions in \cref{app_sec_alpha} for  $\alpha^{(i)},\alpha^{(j)}\in \Theta_\alpha(\tau)$ and $\wh A$. Then for any $d$ satisfying \eqref{lip_d_upper} with $C_d>0$, we have 
	\[
	\EE \left|W(\wh \balpha^{(i)}, \wh  \balpha^{(j)}; d) -W(\balpha^{(i)}, \balpha^{(j)}; d)  \right|  ~ \lesssim  ~ {C_d \over \ok_\tau} \left(
	\sqrt{\tau\log(L) \over N} + {1 \over  \ok_\tau} \sqrt{pK\log(L) \over nN}
	\right).
	\]
\end{theorem} 

In the next section we investigate the optimality of this upper bound by establishing a matching minimax lower bound on estimating $W(\balpha^{(i)}, \balpha^{(j)}; d)$.

\subsection{Minimax lower bounds for the Wasserstein distance  between topic model  mixing measures}\label{sec_lower_bound_est}

\cref{thm_lower_bound_SWD} below  establishes the first lower bound on estimating  $W(\balpha^{(i)},  \balpha^{(j)}; d)$, the Wasserstein distance between two mixing measures,  under topic models. Our lower bounds are valid for any distance $d$ on $\Delta_p$ satisfying 
\begin{equation}\label{lip_d}
	c_d \|u-v\|_1 \le 2 d(u,v) \le C_d \|u-v\|_1,\qquad \text{for all }u,v \in \Delta_p,
\end{equation}
with some $0<c_d\le C_d <\i$.

\begin{theorem}\label{thm_lower_bound_SWD}
	Grant topic model assumptions and assume $1<\tau\le cN$ and $pK \le c'(nN)$ for some universal constants $c,c'>0$. Then, for any $d$ satisfying \eqref{lip_d} with some $0<c_d\le C_d<\i$, we have 
	\begin{align*}
		\begin{split}
			&\inf_{\wh W} \sup_{\substack{ A\in \Theta_A\\\alpha^{(i)}, \alpha^{(j)} \in \Theta_\alpha(\tau)}} \EE \left|\wh W  - W(\balpha^{(i)},\balpha^{(j)};d)\right| 
			\gtrsim  ~ 
			{   c_d^2 \over C_d} \left( \ok_{\tau}^2  \sqrt{\tau  \over N\log^2(\tau)} +  \sqrt{pK \over nN\log^2(p)}\right).
		\end{split}
	\end{align*}
	Here the infimum is taken over all distance estimators. 
\end{theorem}
\begin{proof}
	The proof can be found in \cref{app_proof_thm_lower_bound_SWD}.
\end{proof}


The lower bound in \cref{thm_lower_bound_SWD} consists of two terms that quantify, respectively, the estimation error of estimating the distance when the mixture components, $A_1,\ldots, A_K$, are known, and that when the mixture weights, $\alpha^{(1)},\ldots, \alpha^{(n)}$, are  known. 
Combined with \cref{thm_upper_bound_SWD}, this lower bound establishes that the simple plug-in procedure  can lead to near-minimax optimal estimators of the Wasserstein distance between mixing measures in topic models, up to the $\ell_1\to \ell_1$ condition numbers of $A$, the factor $C_d/c_d$ and  multiplicative logarithmic factors of $L$.

\begin{remark}[Optimal estimation of the mixing measure in Wasserstein distance]\label{rem:mm_lb}
	The proof of \cref{thm_upper_bound_SWD} actually shows that the same upper bound holds for $\E W(\wh \balpha^{(i)}, \balpha^{(i)}; d)$; that is, that the rate in \cref{thm_upper_bound_SWD} is also achievable for the problem of estimating the mixing measure in Wasserstein distance.
	Our lower bound in \cref{thm_lower_bound_SWD} indicates that the plug-in estimator is also nearly minimax optimal for this task.
	Indeed, if we let $\wh \balpha^{(i)}$ and $\wh \balpha^{(j)}$ be arbitrary estimators of $\balpha^{(i)}$ and $\balpha^{(j)}$, respectively, then 
	\begin{equation*}
		|W(\wh \balpha^{(i)}, \wh \balpha^{(j)}; d) - W(\balpha^{(i)},\balpha^{(j)};d)| \leq W(\wh \balpha^{(i)}, \balpha^{(i)}; d) + W(\wh \balpha^{(j)}, \balpha^{(j)}; d)\,.
	\end{equation*}
	Therefore, the lower bound in \cref{thm_lower_bound_SWD} also implies a lower bound for the problem of estimating the measures $\balpha^{(i)}$ and $\balpha^{(j)}$ in Wasserstein distance.
	In fact, a slightly stronger lower bound, without logarithmic factors, holds for the latter problem.
	We give a simple argument establishing this improved lower bound in \cref{sec_upper_bound_mm_est}.
\end{remark}

\begin{remark}[New minimax lower bounds for estimating a general metric $D$ on discrete probability measures]
	The proof of \cref{thm_lower_bound_SWD} relies on a new result, stated in \cref{thm:lower_bound} of \cref{app_general_lower_bound}, which is of interest on its own. It establishes minimax lower bounds for estimating a distance $D(P, Q)$ between two probability measures $P,Q \in \cP(\cX)$ on a finite set $\cX$ based on $N$ i.i.d. samples from $P$ and $Q$. Our results are valid for any distance $D$ satisfying 
	\begin{equation}\label{d_sandwich}
		c_* ~ \textrm{TV}(P,Q) \le D(P,Q) \le C_* ~  \textrm{TV}(P,Q),\quad \text{for all }P,Q\in \cP(\cX),
	\end{equation}
	with some $0< c_* \le C_* <\infty$, including the Wasserstein distance as a particular case. 
	
	Prior to our work, no such bound even for estimating the Wasserstein distance was known except in the special case where the metric space associated with the Wasserstein distance is a tree~\cite{ba2011sublinear,Wang2021}.
	Proving lower bounds on the rate of estimation of the Wasserstein distance 
	for more general metrics is a long standing problem, and the first nearly-tight bounds in the continuous case were given only recently~\cite{NilesWeed2019,liang2019minimax}.

	More generally, the lower bound in \cref{thm:lower_bound} adds to the growing literature on minimax rates of functional estimation for discrete probability measures (see, e.g.,~\cite{Jiao2018,Wu2019}).
	The rate we prove is smaller by a factor of $\sqrt{\log|\cX|}$ than the optimal rate of estimation of the total variation distance~\cite{Jiao2018}; on the other hand, our lower bounds apply to any metric which is bi-Lipschitz equivalent to the total variation distance.
	We therefore do not know whether this additional factor of $\sqrt{\log |\cX|}$ is spurious or whether there in fact exist metrics in this class for which the estimation problem is strictly easier.
	At a technical level, our proof follows a strategy of~\cite{NilesWeed2019} based on reduction to estimation in total variation, for which we prove a minimax lower bound based on the method of ``fuzzy hypotheses"~\cite{Tsy09}.
	The novelty of our bounds involves the fact that we must design priors which exhibit a large \emph{multiplicative} gap in the functional of interest, whereas prior techniques \cite{Jiao2018,Wu2019} are only able to control the magnitude of the {\em additive} gap between the values of the functionals.
	Though additive control is enough to prove a lower bound for estimating total variation distance, this additional multiplicative control is necessary if we are to prove a bound for any metric satisfying~\eqref{d_sandwich}.
\end{remark}

	\section{Inference for the Wasserstein distance between mixing measures in topic models, under  the General Regime}\label{sec_inference}  
	
	We follow the program outlined in the Introduction and begin by giving our general strategy towards obtaining a $\sqrt{N}$ asymptotic limit for a distance estimate. 
	{ Throughout this section $K$ is considered fixed, and does not grow with $n$ or $N$. However, the dictionary size $p$ is allowed to depend on both $n$ and $N$.}   As in the Introduction, we let $N_i = N_j = N$ purely for notational convenience.  
	We let  $\widetilde \alpha^{(i)}$ and $\widetilde \alpha^{(j)}$ be (pseudo) estimators of  $\alpha^{(i)}$ and $\alpha^{(j)}$,  that we will motivate below,  and give formally in \eqref{def_alpha_td}. Using the Kantorovich-Rubinstein dual formulation of  $W(\balpha^{(i)}, \balpha^{(j)}; d)$ (see, \cref{whynot}),  we propose distance estimates of the type 
	\begin{equation}\label{SWD_bar}
		\widetilde  W := \sup_{f \in \widehat{\mathcal{F}}}f^\top(\widetilde \alpha^{(i)}  - \widetilde \alpha^{(j)}),
	\end{equation}
	for 
	\begin{equation}\label{def_F_hat}
		\widehat \cF = \{f \in \RR^K: f_k - f_\ell \leq  d(\widehat{A}_{k}, \widehat{A}_{\ell}), ~  \forall k, \ell \in [K],~  f_1 = 0\}.
	\end{equation} 
	constructed, for concreteness,  relative to the estimator $\wh A$ described in Section \ref{sec_background_topic_models}. 
	\cref{LT_SWD},  stated below, and proved in \cref{app_proof_LT_SWD}, gives our general approach to inference for  $W(\balpha^{(i)}, \balpha^{(j)}; d)$ based on $\wt W$.
	
	\begin{proposition}\label{LT_SWD} Assume access to:  
		\begin{enumerate}
			\item [ (i)] Estimators $\widetilde \alpha^{(i)},  \wt\alpha^{(j)} \in \RR^K$  that satisfy 
			\begin{equation}\label{cond_distr}
				\sqrt{N}\Bigl((\wt\alpha^{(i)}-\wt\alpha^{(j)}) - (\alpha^{(i)} - \alpha^{(j)})\Bigr) \overset{d}{\to} X^{(ij)},\qquad \text{as }n,N\to \i.
			\end{equation}
			for $X^{(ij)} \sim  \cN_{K}(0, Q^{(ij)})$, with some positive semi-definite covariance matrix $Q^{(ij)}$.
			\item [] 
			\item [ (ii)] Estimators $\wh A$ of $A$ such that   $\epsilon_A$ defined in \eqref{two} satisfies
			\begin{equation}\label{errorA} 
				\lim_{n,N\to \i}\epsilon_A\sqrt{N} = 0. 
			\end{equation} 
		\end{enumerate}
		\noindent Then,  we have the following convergence in distribution as $n,N\to \i$,
		\begin{equation}\label{SWD_limit_distr}
			\sqrt{N} \Bigl(
			\widetilde  W -  W(\balpha^{(i)}, \balpha^{(j)}; d)
			\Bigr) \overset{d}{\to} \sup_{f\in \cF'_{ij}}f^\top X^{(ij)}, 
		\end{equation}
		where 
		\begin{equation}\label{space_F_prime}
			\cF'_{ij} := \cF\cap \left\{
			f\in\RR^K: f^\top (\alpha^{(i)} - \alpha^{(j)}) = W(\balpha^{(i)}, \balpha^{(j)}; d)
			\right\}
		\end{equation}
		with 
		\begin{equation}\label{def_F}
			\cF = \{f \in \RR^K: f_k - f_\ell \leq d(A_k, A_\ell),~  \forall k, \ell \in [K],~  f_1 = 0\}.
		\end{equation}  
	\end{proposition}  
	
	Proposition \ref{LT_SWD} shows that in order to obtain the desired  $\sqrt{N}$ asymptotic limit 
	(\ref{SWD_limit_distr}) it is sufficient to control, separately,  the  estimation of the mixture components,  in probability, and  that of the mixture weight estimators, in distribution. 
	{
		Since  estimation of $A$ is based on all  $n$ documents, each of size $N$, all the limits in \cref{LT_SWD} are taken over both $n$ and $N$ to ensure that the error in estimating $A$ becomes negligible in the distributional limit.  For our running example of $\wh A$, constructed  in Section \ref{Aconstruct} of \cite{supplement}, $$\epsilon_A = O\left( \sqrt{pK\log(L)\over nN}\right),$$ for $L := n\vee p\vee N$, when $d$ is either the Total Variation distance, or the Wasserstein distance, as shown in  \cite{bing2021likelihood}.
	}


	The proof of \cref{LT_SWD} shows that,  when (\ref{errorA}) holds,  as soon as (\ref{cond_distr}) is established, the desired (\ref{SWD_limit_distr})  follows by an application of the functional $\delta$-method, recognizing (see, \cref{prop_DHD} in \cref{app_proof_LT_SWD}) that the function $h: \RR^K \to \RR$
	defined as $h(x) = \sup_{f\in \cF} f^\T x$ is Hadamard-directionally differentiable at $u = \alpha^{(i)} - \alpha^{(j)}$ with derivative $h'_u: \RR^K\to \RR$ defined as 
	\begin{equation}\label{eq_hdd}
		h'_u(x) = \sup_{f\in \cF: f^\T u = h(u)}f^\T x = \sup_{f\in \cF'_{ij}} f^\T x. 
	\end{equation}
	The last step of the proof  has also been advocated in prior work \cite{shapiro1991asymptotic,Dumbgen1993,fang2019inference}, and closest to ours is \cite{sommerfeld2018inference}. This work estimates the  Wasserstein distance between two discrete probability vectors of fixed dimension  by the Wasserstein distance between their observed frequencies.  
	In \cite{sommerfeld2018inference}, the equivalent of (\ref{cond_distr}) is the basic central limit theorem for the empirical estimates of the cell probabilities of a multinomial distribution. In topic models, and in the  {\it Classical Regime},  that in particular makes, in the topic model context,   the unrealistic assumption that $A$ is known, one could take $\wt{\alpha} = \wh \alpha_A$, the MLE of $\alpha$. Then, invoking classical results, for instance, \cite[Section 16.2]{Agresti2012},  the convergence in (\ref{cond_distr}) holds,   with $Q^{(ij)} = \Gamma^{(i)} + \Gamma^{(j)}$  given above in \eqref{def_Gamma_i}. 
	
	In contrast, constructing estimators of the potentially sparse mixture weights for which (\ref{cond_distr}) holds, under the {\it General Regime}, in the context of topic models, requires special care,  
	is a novel contribution to the literature,  and is treated in the following section. Their analysis, coupled with \cref{LT_SWD},  will be the basis of \cref{thm_limit_distr} of \cref{sec_LT_SWD} in which we establish the asymptotic distribution of our proposed distance estimator in topic models.

	\subsection{Asymptotically normal estimators of sparse  mixture weights in topic models, under the General Regime}\label{sec_ASN}

	We concentrate on one sample, and drop the superscript $i$ from all relevant quantities. To avoid any possibility of confusion, in this section we also let $\alpha = \alpha_{*}$ denote the true population quantity. 
	
	As discussed in \cref{sec_intro_background} of the  Introduction, under the Classical Regime, the MLE estimator of $\alpha_{*}$ is asymptotically normal and efficient, a fact which cannot be expected to hold under General Regime for an estimator  restricted to $\Delta_K$, especially when $\alpha_*$ is on the boundary of $\Delta_K$.  We exploit the KKT conditions satisfied by $\wh \alpha$, defined in \eqref{MLE}, to motivate the need for its  correction, leading  to a final  estimator that  will be a ``de-biased'' version of $\wh \alpha$, in a sense that will be made precise shortly.  

	Recall that $\wh \alpha$ is the maximizer,  over $\Delta_K$, of 
	$ \sum_{j \in J} X_j \log (\wh A_{j\cdot}^\T\alpha)$, with $J = \supp(X)$, for the observed frequency vector $X \in \Delta_p$.
	Let 
	\begin{equation}\label{def_r_hat}
		\wh r = \wh A \wh \alpha,\qquad \text{with}\quad   \wh J := \supp(\wh r~ ).
	\end{equation}
	We have  $J \subseteq \wh J$ with probability equal to one. To see this, note that  if for some $j \in \wh J$ we would have $\wh r_j = \wh A_{j\cdot}^\T \wh \alpha =  0$, and  then $\wh \alpha$ would not be a maximizer. 
	

	
	Observe  that the KKT 
	conditions associated with $\wh \alpha$,  for dual variables $\lambda \in \RR^K$ satisfying $\lambda_k \ge 0$, $\lambda_k \wh \alpha_k = 0$ for all $k\in [K]$, and  arising from the non-negativity constraints on $\alpha$ in the primal problem, imply that 
	\begin{eqnarray}\label{KKThat}  
		0_K  &=&   \sum_{j \in {\wh J}}  \frac{X_j}{\wh A_{j\cdot}^{\T} \wh \alpha}\wh A_{j\cdot}  -\1_K + {\lambda \over N}. 
	\end{eqnarray} 
	By adding and subtracting terms in (\ref{KKThat}) (see also the details of the proof of  \cref{thm_asn}),  we see that $\wh \alpha$ satisfies  the  following equality
	\begin{equation}\label{bias-in}  
		\sqrt{N} ~  \wt V (\wh \alpha - \alpha_*) = \sqrt{N} ~  \Psi (\alpha_*) - \sqrt{N} ~  \Psi (\wh \alpha), 
	\end{equation} 
	where 
	\begin{equation}\label{def_V_td}
		\wt V  :=  \sum_{j\in \wh J}{X_j \over \wh A_{j\cdot}^\T \wh\alpha ~ \wh A_j^\T \alpha_{*} }\wh A_{j\cdot}\wh  A_{j\cdot}^\T,\qquad 
		\Psi ( \alpha) :=  \sum_{j \in \wh J} {X_j -  \wh A_{j\cdot}^\T  \alpha  \over  \wh A_{j\cdot}^\T  \alpha } \wh A_{j\cdot}.
	\end{equation}

	On the basis of (\ref{bias-in}), standard asymptotic principles dictate that the asymptotic normality  of $\wh \alpha$  would follow by  establishing that $\wt V$ converges in probability to a matrix that will contribute to the asymptotic covariance matrix, the term $\sqrt{N} \Psi (\alpha_*)$ converges to a Gaussian limit, and  $\sqrt{N} \Psi (\wh \alpha)$ vanishes, in probability.  The latter, however, cannot be expected to happen, in general,  and it is this term that creates the first difficulty  in the analysis. To see this, writing the right hand side in (\ref{KKThat}) as 
	\begin{eqnarray} 
		\sum_{j \in \wh J} {X_j -  \wh A_{j\cdot}^\T  \wh  \alpha  \over  \wh A_{j\cdot}^\T \wh  \alpha } \wh A_{j\cdot}  - \sum_{j\in \wh J^c}\wh A_{j\cdot} + {\lambda \over N} 
		= \Psi (\wh \alpha) - \sum_{j\in \wh J^c}\wh A_{j\cdot} + {\lambda \over N}, \nonumber 
	\end{eqnarray} 
	we obtain from (\ref{KKThat}) that 
	\begin{equation}\label{KKTuse}
		\Psi (\wh \alpha) = - {\lambda \over N}  + \sum_{j\in \wh J^c}\wh A_{j\cdot}.
	\end{equation} 
	We note that,  under the Classical Regime, $\bar{J} :=  \{j: A_{j\cdot}^\T\alpha_* > 0\}  = [p]$, which implies that, with probability tending to one,  $J =[p]$  thus $\wh J = [p]$ and $\wh J^c = \emptyset$.
	Consequently, \cref{KKTuse} yields $\sqrt{N} \Psi(\wh {\alpha}) = \cO_\PP(\lambda/\sqrt{N})$, which is expected to vanish asymptotically. This is indeed in line with classical analyses,  in which the {\it asymptotic  bias} term,  $\Psi(\wh {\alpha})$, is of order $\cO_\PP(1/N)$. 
	
	However, in the General Regime,  we do not expect this to happen as we allow $\bar J \subset  [p]$.   Since we show in \cref{lem_prod_AT} of \cref{app_proof_thm_asn} that, with high probability, $\bar{J} = \wh J$, we therefore have   $\wh J^c \neq \emptyset$. Thus,  one cannot ensure that $\sqrt{N} {\Psi}(\wh \alpha)$ vanishes asymptotically,  because the last term in (\ref{KKTuse}) is non-zero, and  we do not have direct control on the dual variables $\lambda$ either. It is worth mentioning that the usage of $\wh J$  is needed as the naive estimator $J = \supp(X)$ of $\bar J$ is not consistent in general when $p > N$, whereas $\wh J$ is; see, \cref{lem_prod_AT} of \cref{app_proof_thm_asn}. 

	
	The next natural step is to construct a new estimator, by removing the bias of $\wh \alpha$. We let $\wh V$ be a matrix that will be defined shortly, and denote by $\wh V^{+}$ its generalized inverse. Define 
	\begin{equation}\label{def_alpha_td}
		\wt \alpha =  \wh \alpha +  \wh V^{+} \Psi(\wh \alpha), 
	\end{equation}
	and observe that it satisfies 
	\begin{equation}\label{bias-out}  
		\wt \alpha  - \alpha_* = (\bI_K - \wh V^{+}\wt V  )(\wh \alpha - \alpha_*) + \wh V^{+}  \Psi (\alpha_*),
	\end{equation}
	a decomposition that  no longer contains the possibly non-vanishing  bias term $\Psi(\wh\alpha)$ in (\ref{bias-in}).  The (lengthy) proof  of \cref{thm_asn} stated in \cite{supplement} 
	shows that, indeed,  after appropriate scaling, the first term in the right hand side of (\ref{bias-out}) vanishes asymptotically, and the second  one  has a Gaussian limit, as soon as $\wh V$ is appropriately chosen.   
	
	As mentioned above, the choice of $\wh V$ is crucial for obtaining the desired asymptotic limit, and is given  by  
	\begin{equation}\label{def_V_hat}
		\wh V :=  
		\sum_{j \in \wh J} {1 \over \wh A_{j\cdot}^\T \wh \alpha} \wh A_{j\cdot} \wh A_{j\cdot}^\T.
	\end{equation} 
	This choice is motivated by using  $\Sigma:=\Sigma^{(i)}$ defined in \eqref{def_Sigma_i} as a benchmark for the asymptotic covariance matrix of the limit, as we would like that the new estimator $\wt \alpha$ not only be valid in the General Regime, but also not lose the asymptotic efficiency of the  MLE,  in the Classical Regime.  Indeed, 
	the proof of \cref{thm_asn} shows that the asymptotic covariance matrix of $\sqrt{N} \wh V^+\Psi (\alpha_*)$ is nothing but $\Sigma$. Notably, from the decomposition in \eqref{bias-out}, it would  be tempting to use  $\wt V$, given by (\ref{def_V_td}),  with $\alpha_*$ replaced by $\wh \alpha$, instead of $\wh V$. However, although this matrix has the desired asymptotic behavior, we find in our simulations that its finite sample performance relative to $\wh V$ is sub-par.

	{ 
		Our final estimator is therefore 
		$\wt \alpha$ given by (\ref{def_alpha_td}), with $\wh V$ given by (\ref{def_V_hat}). We show below that this estimator is asymptotically normal.    The construction of $\wt \alpha$ in \eqref{def_alpha_td} has the flavor  of a Newton--Raphson one-step correction of $\wh \alpha$ (see, \cite{LeCam1990}), relative to  the estimating equation   $ \Psi (\alpha) = 0$.  However, classical general results  on the asymptotic distribution of one-step corrected estimators such as  Theorem 5.45 of \cite{vaart_1998} cannot be employed directly, chiefly because, when translated to our context, they are derived 
		for deterministic  $A$. In our problem, after determining  the appropriate 
		de-biasing quantity  and the appropriate $\wh V$, the main  remaining  difficulty in  proving asymptotic normality is in controlling the (scaled) terms in the right hand side of (\ref{bias-out}).
		A key challenge is the delicate interplay between $\wh A$ and $\wh \alpha$, which are dependent, as they are  both estimated via the same sample. This difficulty is further elevated  in the General Regime, when not only $p$, but also the entries of $A$ and $\alpha$ (thereby the quantities  $\xi$, $\Tm$, $\ok_K$ and $\uk_\tau$, defined in Section \ref{sec_est_weights}), can grow with $N$ and $n$. In this case, one needs careful control of their interplay.

		We state  and prove results  pertaining to this general situation in   \cref{app_proof_thm_asn}:  \cref{thm_asn_general}  is our most general result, proved for any estimator $\wh A$ whose estimation errors, defined by  the left hand sides of \eqref{rate_A} and \eqref{rate_A_sup},  can be well controlled, in a sense made precise in the statement of \cref{thm_asn_general}.  As an immediate consequence,  Theorem \ref{thm_asn} in  \cref{app_proof_thm_asn} establishes the asymptotic normality of $\wh \alpha$ under the specific control given by  the right hand sides of \eqref{rate_A} and \eqref{rate_A_sup}.
		We recall that, for instance,  these upper bounds are valid for  the estimator given  in Section \ref{Aconstruct} of \cite{supplement}. \\
		
		We state below a version of our results, corresponding to an estimator of $A$ that satisfies \eqref{rate_A} and \eqref{rate_A_sup}, by  making the following simplifying assumptions:  the condition numbers  $\ok_K^{-1}, \uk_\tau^{-1}$, the signal strength $\xi$,  and $|\oJ \setminus \uJ|$ are bounded. The latter means that we assume that the number of positive entries in $r$ that fall below $\log p/N$ is bounded. 
		We also assume, in this simplified version of our results,  that the magnitude of  the entries of $\alpha$ does not grow with either $N$ or $n$.  We stress that we do not make any of these assumptions in  \cref{app_proof_thm_asn}, and use them  here  for clarity of exposition only.

		Recall that $\Sigma:=\Sigma^{(i)}$ is defined in \eqref{def_Sigma_i}. Let $\Sigma^+$ be the Moore-Penrose inverse of $\Sigma$ and let $\Sigma^{+\frac12}$ be its matrix square root.

		\begin{theorem}\label{cor_asn_fixed}
			Let $\wh A$ be any estimator such that \eqref{rate_A} and \eqref{rate_A_sup} hold, for  $K$  that  does not grow with $n$ and $N$.   Assume  $\log^2(p\vee n) = o(N)$ and 
			$
			p\log(L) = o(n).
			$ 
			
			Then
			\begin{equation}\label{eq_conv_distr}
				\lim_{n,N\to \i} \sqrt{N}  ~ \Sigma^{+{1\over 2}}\left(\wt \alpha - \alpha_*\right) \overset{d}{\to} \cN_K\left(0, \begin{bmatrix}
					\bI_{K-1} & 0 \\ 0 & 0
				\end{bmatrix}\right).
			\end{equation}
		\end{theorem}
	}
	{ 
		The proof of Theorem \ref{cor_asn_fixed} is obtained as a direct consequence of  the general  Theorem \ref{thm_asn}, stated and proved  in  \cref{app_proof_thm_asn}. 
		Theorem \ref{cor_asn_fixed} holds under a mild requirement on $N$, the document length, and on  a requirement on $n$, the number of documents,  that is expected to hold in topic model contexts, when  $n$ is typically (much) larger than the dictionary size $p$. 
		
	}
	{
		\begin{remark}
			In line with classical theory, joint asymptotic normality of  mixture weights estimators, as in Theorem  \ref{cor_asn_fixed}, can  only be established when $K$ does not grow with either $N$ or $n$. \\
			
			In problems in which $K$ is expected  to grow  with the ambient  sample sizes,  although joint  asymptotic results such as \cref{eq_conv_distr} become ill-posed, one can still study the marginal distributions of $\wt \alpha$, over any subset of constant dimension. In particular, a straightforward modification of our analysis yields the limiting distribution of each component of $\wt \alpha$:  for any $k\in [K]$,
			\[
			\lim_{n,N\to \i} \sqrt{N / \Sigma_{kk}} \left(\wt \alpha_k - \alpha_{*k}\right) \overset{d}{\to} \cN(0,1).  
			\]
			The above result can be used to construct confidence intervals for any  entry of $\alpha_*$, including those  with zero values.

		\end{remark}
	}

	\begin{remark}[Alternative estimation of the mixture weights]\label{rem_LS} 
		Another natural estimator of the mixture weights is the weighted least squares estimators, defined as
		\begin{equation}\label{LSK}
			\wt \alpha_{LS} := \argmin_{\alpha \in \RR^K} \|
			\wh D^{-1/2}(X - \wh A\alpha)
			\|_2^2 
		\end{equation}
		with $\wh D := \diag(\|\wh A_{1\cdot}\|_1, \ldots, \|\wh A_{p\cdot}\|_1).$
		The use of the pre-conditioner $\wh D$ in the definition of $\wh A^+$ is reminiscent of the definition of the normalized Laplacian in graph theory: it moderates the size of the $j$-th entry of each mixture  estimate, across the $K$ mixtures, for each $j \in [p]$. 
		
		The estimator $\wt \alpha_{LS}$ can also be viewed as an (asymptotically)  de-biased version of a restricted least squares estimator,  obtained as  in (\ref{LSK}), but minimizing only over  $\Delta_K$ (see \cref{rem_WLS} in \cref{app_least_squares}).  It  therefore has the same flavor as our proposed estimator  $\wt \alpha$. 
		In \cref{ASN_LS} of \cref{app_least_squares} we further prove that under suitable conditions, as $n,N\to \infty$,
		\[
		\sqrt{N} \Sigma_{LS}^{-1/2}\left(\wt \alpha_{LS} - \alpha_*\right) \overset{d}{\to} \cN_K(0, \bI_K).
		\]
		for some matrix $\Sigma_{LS}$ that does not equal the covariance matrix $\Sigma$ given by \eqref{def_Sigma_i}. 
		
		The asymptotic normality of $\wt\alpha_{LS}$ together with \cref{LT_SWD} shows that inference for the distance between mixture weights can also be conducted relative to a distance estimate \eqref{SWD_bar} based on 
		$\wt\alpha_{LS}$. 
		Note however that the potential sub-optimality of the  limiting covariance matrix  of this mixture weight estimator  also affects the length of the confidence intervals for   the Wasserstein distance, see 
		our simulation results in \cref{sec_sim_MLE_LS}.  We therefore recommend the usage of the asymptotically debiased, MLE-type estimators $\wt \alpha$ analyzed in this section, with $\wt \alpha_{LS}$   being a second best.

		
	\end{remark}

	\subsection{The limiting distribution of the proposed distance estimator }\label{sec_LT_SWD}  
	As a consequence of \cref{LT_SWD} and \cref{thm_asn}, we derive the  limiting distribution of our proposed estimator of the distance $W(\balpha^{(i)},\balpha^{(i)};d)$ in \cref{thm_limit_distr} below.  
	
	Let $\wt W$ 
	be defined as \eqref{SWD_bar} by using $\wh A$ and the de-biased estimators  $\wt \alpha^{(i)},\wt \alpha^{(j)}$ in \eqref{def_alpha_td}. Let
	\begin{equation}\label{def_Z}
		Z_{ij} \sim \cN_K(0, Q^{(ij)})
	\end{equation}
	where  we assume that
	\begin{equation}\label{def_Q_ij}
		Q^{(ij)} \defeq
		\lim_{n,N\to\infty}   (\Sigma^{(i)} + \Sigma^{(j)}) \in \RR^{K\times K}
	\end{equation}
	exists with $\Sigma^{(i)}$ defined in \eqref{def_Sigma_i}. Here we rely on the fact that $K$ is independent of $N$ and $n$ to define the limit, but allow the model parameters $A$, $\alpha^{(i)}$ and $\alpha^{(j)}$ to depend on both $n$ and $N$.

	\begin{theorem}\label{thm_limit_distr}
		Grant conditions in \cref{thm_asn} for $\alpha^{(i)},\alpha^{(j)}\in\Theta_\alpha(\tau)$. For any $d$ satisfying \eqref{lip_d_upper} with $C_d=\cO(1)$, we have the following convergence in distribution, as $n,N\to \infty$, 
		\begin{equation}\label{limit_distr_MLE}
			\sqrt{N}\left(
			\wt W - W(\balpha^{(i)},\balpha^{(j)};d)
			\right) \overset{d}{\to} \sup_{f\in \cF'_{ij}}f^\top Z_{ij}
		\end{equation}
		with $\cF'_{ij}$ defined in \eqref{space_F_prime}. 
	\end{theorem}
	The proof of \cref{thm_limit_distr} immediately follows from \cref{LT_SWD} and \cref{thm_asn} together with $\epsilon_A\le C_d \|\wh A - A\|_{1,\i}$  and \eqref{rate_A}. The requirement $\epsilon_A\sqrt{N}\to 0$ of \cref{LT_SWD} in this context reduces to $p\log(L) / n \to 0$, which holds when: 
	(i) $p$ is fixed,  and $n$ grows, a situation encountered when the dictionary size is not affected by the growing size of the corpus; (ii) $p$ is independent of $N$ but grows with, and slower than, $n$, as we expect that not many words are added to an initially fixed dictionary as more documents are added to the corpus; (iii) $p = p(N) < n$, which reflects the fact that the dictionary size can depend on the length of the document, but again should grow slower than the number of documents.

	Our results in \cref{thm_limit_distr} can be readily generalized to the cases where $N_i$ is different from $N_j$. We refer to \cref{app_proof_thm_limit_distr_nm} for the precise statement in this case.

	\subsection{Fully data driven inference for the Wasserstein distance between mixing measures}\label{sec_est_LT}

	To  utilize \cref{thm_limit_distr} in practice for inference on $W(\balpha^{(i)},\balpha^{(j)};d)$,  for instance for constructing confidence intervals or conducting hypothesis testing, we provide below  consistent quantile estimation for  the limiting distribution in \cref{thm_limit_distr}.

	Other possible data-driven inference strategies are bootstrap-based. The table below summarizes the pros and cons of these procedures, in contrast with our proposal (in green).  The symbol $\times$ means that the procedure is not available in a certain scenario, while $\checkmark$ means that it is, and   {\color{green} $\checkmark$} means that the  procedure has best performance, in  our experiments. 
	\begin{table}[H]
		\centering
		\renewcommand{\arraystretch}{1.25}{
			\resizebox{\textwidth}{!}{
				\begin{tabular}{l   c   c   c}
					Procedures   &  Make use of the form of the limiting distribution & Any $\alpha, \beta \in \Delta_K$ & Only  $\alpha = \beta \in \Delta_K$\\\hline
					(1)  Classical bootstrap   &  No &  $\times$ &  $\times$ \\
					(2)  $m$-out-of-$N$ bootstrap & No &  $\checkmark$ & $\checkmark$ \\
					(3) Derivative-based  bootstrap & Yes &  $\checkmark$  & $\checkmark$ \\
					(4) Plug-in estimation & Yes & {\color{green} $\checkmark$} & {\color{green} $\checkmark$ }\\\hline
					\multicolumn{2}{l}{Performance for constructing confidence intervals ($>$ means better)} & (4) > (2) & (3) $\approx$ (4) > (2)\\\hline
				\end{tabular}
		}}
	\end{table}
	
	\subsubsection{Bootstrap}\label{sec_BS}
	
	The bootstrap \cite{Efron1979} is a powerful tool for estimating a distribution. However, since the Hadamard-directional derivative of $W(\balpha^{(i)},\balpha^{(j)};d)$ w.r.t. ($\alpha^{(i)} - \alpha^{(j)}$)  is non-linear (see \cref{eq_hdd}), \cite{Dumbgen1993} shows that the classical bootstrap is not consistent.
	The same paper also shows that this can be corrected by  a version of $m$-out-of-$N$ bootstrap, for $m/N \to 0$ and $m\to \infty$. Unfortunately, the optimal choice of $m$ is not known, which hampers its  practical implementation,  as suggested by our simulation result in \cref{sec_sim_BS}. Moreover, our simulation result in \cref{sec_sim_CI} also shows that the $m$-out-of-$N$ bootstrap seems to have inferior finite sample performance comparing to other procedures described below. 
	
	An alternative to the $m$-out-of-$N$ bootstrap is to use a derivative-based bootstrap \cite{fang2019inference} for estimating the limiting distribution in \cref{thm_limit_distr}, 
	by plugging the bootstrap samples into the Hadamard-directional derivative (HDD) on the right hand side of \eqref{limit_distr_MLE}. As proved in \cite{fang2019inference} and pointed out by \cite{sommerfeld2018inference}, this procedure is consistent at the null $\alpha^{(i)} = \alpha^{(j)}$ in which case $\cF'_{ij} = \cF$ (see \eqref{def_F} and \eqref{space_F_prime}) hence the HDD is a known function. However, this procedure is not directly applicable when $\alpha^{(i)}\ne \alpha^{(j)}$ since the HDD depends on $\cF'_{ij}$ which needs to be consistently estimated. Nevertheless, we provide below a consistent estimator of $\cF'_{ij}$ which can be readily used in conjunction with the derivative-based bootstrap. For the reader's convenience, we provide details of both the $m$-out-of-$N$ bootstrap and the derivative-based bootstrap in \cref{app_BS_procedure}.

	\subsubsection{A plug-in estimator for the limiting distribution of the distance estimate}\label{sec_est_LT_plugin}
	
	In view of \cref{thm_limit_distr}, we propose to replace  the population level quantities  that appear in the limit by their consistent estimates. Then, we estimate the cumulative  distribution function of the limit via Monte Carlo simulations.
	
	
	To be concrete, for a specified integer $M>0$, let $Z_{ij}^{(1)},\ldots, Z_{ij}^{(M)}$ be i.i.d. samples from $\cN_K(0, \wh \Sigma^{(i)} + \wh \Sigma^{(j)})$ where, for $\ell \in \{i,j\}$,
	\begin{equation}\label{def_Sigma_hat_i}
		\wh \Sigma^{(\ell)} = 
		\Bigl(\sum_{j\in \wh J^{(\ell)}} {A_{j\cdot}A_{j\cdot}^\T \over \wh r_j^{(\ell)}}\Bigr)^{-1} 
		- \wh \alpha^{(\ell)}\wh \alpha^{(\ell)\T},
	\end{equation}
	with $\wh r^{(\ell)} = \wh A\wh\alpha^{(\ell)}$ and $\wh J^{(\ell)} = \supp(\wh r^{(\ell)})$. We then propose to estimate the limiting distribution on the right hand side of \eqref{limit_distr_MLE} by the empirical distribution of 
	\begin{equation}\label{samples}
		\sup_{f\in \wh \cF'_\delta} f^\T Z_{ij}^{(b)},\qquad \text{for }~  b = 1,\ldots, M,
	\end{equation}
	where, for some tuning parameter $\delta\ge 0$,
	\begin{equation}\label{def_F_hat_prime_delta}
		\wh \cF'_{\delta} = \wh \cF \cap \left\{
		f\in\RR^K: |f^\top (\wh\alpha^{(i)} - \wh\alpha^{(j)})- W(\wh\balpha^{(i)}, \wh\balpha^{(j)}; d)| \le \delta
		\right\}
	\end{equation}
	with $\wh \cF$ defined in \eqref{def_F_hat} and $\wh\balpha^{(i)}, \wh\balpha^{(j)}$ defined in \eqref{mm_est_ij}.
	
	Let $\wh F_{N,M}$ be the empirical cumulative density function (c.d.f.) of \eqref{samples} and write $F$ for the c.d.f. of the limiting distribution on the right hand side of \eqref{limit_distr_MLE}. The following theorem states that $\wh F_{N,M}(t)$ converges to $F(t)$ in probability for all $t\in \RR$.  Its proof is deferred to \cref{app_thm_cdf}.

	\begin{theorem}\label{thm_cdf}
		Grant conditions in \cref{cor_asn_fixed} for  $\alpha^{(i)}$ and $\alpha^{(j)}$. 
		For any $d$ satisfying \eqref{lip_d_upper} with some constant $C_d>0$, by choosing $\delta \asymp \sqrt{\log(L)/N} + \sqrt{p\log(L)/(nN)}$  in \eqref{def_F_hat_prime_delta}, we have, for any $t\in \RR$, 
		\[  
		|\wh F_{N,M}(t) -  F(t)| = o_\PP(1), \quad \text{as } M\to \i \text{ and }n,N\to \i.
		\]
	\end{theorem}

	\begin{remark}[Confidence intervals]
		\cref{thm_cdf} in conjunction with \cite[Lemma 21.2]{vaart_1998} ensures that any quantile  of the limiting distribution on the right hand side of \eqref{limit_distr_MLE} can also be consistently estimated by its empirical counterpart of \eqref{samples}, which, therefore by Slutsky's theorem, can be used to provide confidence intervals of $W(\balpha^{(i)}, \balpha^{(j)}; d)$. This is summarized in the following corollary. Let $\wh F_{N,M}^{-1}(t)$ be the quantile of \eqref{samples} for any $t\in (0,1)$.
		
		\begin{corollary}\label{cor_CI}
			Grant conditions in \cref{thm_cdf}. For any $t \in (0,1)$, we have
			\[
			\lim_{n,N\to \i}\lim_{M\to \i} \PP\left\{
			\wt W - {\wh F_{N,M}^{-1}(1-t/2)\over \sqrt N}   \le W(\balpha^{(i)}, \balpha^{(j)}; d)   \le   \wt W - {\wh F_{N,M}^{-1}(t/2)\over \sqrt N} 
			\right\} = 1-t.
			\]
		\end{corollary}
	\end{remark}

	\begin{remark}[Hypothesis testing at the null $\alpha^{(i)} = \alpha^{(j)}$]
		In applications where we are interested in inference for $\alpha^{(i)} = \alpha^{(j)}$, the above plug-in procedure can be simplified. There is no need for the tuning parameter $\delta$, since  one only needs to compute the empirical distribution of 
		\[
		\sup_{f\in \wh \cF} f^\T Z_{ij}^{(b)},\qquad \text{for }~  b = 1,\ldots, M.
		\]
		
	\end{remark}

	\begin{remark}[The tuning parameter $\delta$]
		The rate of $\delta$ in \cref{thm_cdf} is stated for the MLE-type estimators of $\alpha^{(i)}$ and $\alpha^{(j)}$ as well as the estimator $\wh A$ of $A$ satisfying \eqref{rate_A}.
		
		Recall $\epsilon_{\alpha, A}$ and $\epsilon_A$ from \eqref{one} and \eqref{two}. For a generic estimator $\wh\alpha^{(i)},\wh\alpha^{(j)}\in \Delta_K$ and $\wh A$, \cref{lem_hausdorff_alternative} of \cref{app_thm_cdf} requires the choice of $\delta$ to satisfy $\delta \ge 6\epsilon_A+2\diam(\cA)\epsilon_{\alpha, A}$ and $\delta \to 0$ as $n,N\to \i$.
		
		To ensure consistent estimation of the limiting c.d.f., we need to first estimate the set $\cF'_{ij}$ given by \eqref{space_F_prime}  consistently, for instance, in the Hausdorff distance. Although consistency of the plug-in estimator $\wh \cF$ given by \eqref{def_F_hat} of $\cF$ can be done with relative ease,   proving  the consistency of the plug-in estimator of the facet $f^\T (\alpha^{(i)}-\alpha^{(j)}) = W(\balpha^{(i)},\balpha^{(j)};d)$ requires extra care. We introduce  $\delta$ in \eqref{def_F_hat_prime_delta} mainly for technical reasons related to this consistency proof. Our simulations reveal that simply setting $\delta = 0$ yields good performance overall. 
	\end{remark}


			%
	
	\begin{supplement}
		\stitle{Supplement to ``Estimation and inference for the Wasserstein distance between mixing measures in topic models''}
		\sdescription{The supplement \cite{supplement} contains all the proofs, additional theoretical results, all numerical results and review of the error bound of the MLE-type estimators of the mixture weights.}
	\end{supplement}

	\bibliographystyle{imsart-number}
	\bibliography{ref}

	\newpage

		\cref{Aconstruct} contains the description of an estimator of $A$.
	\cref{app_proofs} contains most of the proofs. In \cref{app_general_lower_bound} we state the minimax lower bounds of estimating any metric on discrete probability measures that is bi-Lipschitz equivalent to the Total Variation Distance. In \cref{sec_upper_bound_mm_est} we state minimax lower bounds for estimation of mixing measures in topic models. In \cref{app_proof_thm_limit_distr_nm} we derive the limiting distribution of the proposed W-distance estimator for two samples with different sample sizes.  \cref{app_sim} contains all numerical results. In \cref{app_sec_alpha} we review the $\ell_1$-norm error bound of the MLE-type estimators of the mixture weights in \cite{bing2021likelihood}.  Finally,  
	\cref{app_least_squares} contains our theoretical guarantees of the W-distance estimator based on the weighted least squares estimator of the mixture weights.

	\appendix

	\section{An algorithm for estimating the mixture component  matrix $A$}\label{app_alg_A}\label{Aconstruct}
	
	We recommend the following procedure for estimating the matrix of mixture components $A$, typically referred to, in the topic model literature, as the  word-topic matrix $A$. Our procedure requires that Assumption \ref{ass_anchor} holds, and it consists of two parts: (a) estimation of  the partition of anchor words, and (b) estimation of the word-topic matrix $A$.  
	Step (a) uses the procedure proposed in \cite{bing2020fast}, stated in Algorithm \ref{alg_I},  while step (b) uses the procedure proposed in \cite{bing2020optimal}, summarized in Algorithm \ref{alg_1}.

	Recall that $\bX = (X^{(1)}, \ldots, X^{(n)})$ with $N_i$ denoting the length of document $i$. 
	Define 
	\begin{equation}\label{est_Theta}
		\wh\Theta = {1\over n}\sum_{i =1}^n\left[
		{N_i \over N_i - 1}X^{(i)}X^{(i)\top}  - {1\over N_i-1} \textrm{diag}(X^{(i)})\right]
	\end{equation}
	and
	\begin{equation}\label{def_R_hat}
		\wh R = D_X^{-1} 	\wh\Theta  D_X^{-1}
	\end{equation}
	with $D_X = n^{-1}\diag(\bX\1_n)$.

	\subsection{Estimation of the index set of the anchor words, its partition and the number of topics}
	
	We write the set of anchor words as $	I = \cup_{k\in [K]}I_k$ and its partition $\I = \{I_1,\ldots, I_K\}$ where
	\[
	I_k = \{j\in [p]: A_{jk} > 0, \ A_{\ell k} = 0, \ \forall \ \ell \ne j\}.
	\] 
	Algorithm \ref{alg_I} estimates the index set $I$, its partition $\mathcal{I}$ and the number of topics $K$ from the input matrix $\wh R$. The choice $C_1 = 1.1$ is recommended and is empirically verified to be robust in \cite{bing2020fast}. A data-driven choice of $\delta_{j\ell}$ is specified in \cite{bing2020fast} as 
	\begin{equation}\label{delta3}
		\wh \delta_{j\ell} =  {n^2 \over\| \bX_{j\cdot}\|_1 \|  \bX_{\ell \cdot}\|_1 }
		\!\left\{\wh \eta_{j\ell}  +2 \wh\Theta_{j\ell}   \sqrt{\log M \over n }
		\!\left[\! \frac{n}{\|\bX_{j\cdot}\|_1} \!\left( \frac{1}{n} \sum_{i=1}^n \frac{\bX_{ji} }{ N_i} \right)^{{1\over 2}}\!\!\!+  \!\frac{n}{\|\bX_{\ell\cdot}\|_1}  \!\left( \frac{1}{n} \sum_{i=1}^n \frac{\bX_{\ell i} }{ N_i} \right)^{{1\over 2}}
		\right]\right\}
	\end{equation}
	with $M = n \vee p \vee \max_i N_i$ and 
	\begin{align}\label{def_eta3}
		\wh \eta_{j\ell} =  &~3\sqrt{6}\left( \left\| \bX_{j\cdot}\right\|_\infty^{1\over 2} +\left\| \bX_{\ell\cdot}\right\|_\infty^{1\over 2}\right) \sqrt{ \log M\over n}\left(\frac{1}{n} \sum_{i=1}^n \frac{ \bX_{ji} \bX_{\ell i} }{N_i}\right)^{1\over 2}+
		\\\nonumber &+  {2 \log M \over n}\left( \| \bX_{j\cdot}\|_\infty + \| \bX_{\ell\cdot}\|_\infty \right) \frac1n \sum_{i=1}^n {1\over N_i}
		+ 31 \sqrt{(\log M)^4 \over n}\left({1\over n}\sum_{i =1}^n{\bX_{ji} + \bX_{\ell i} \over N_i^3} \right)^{\rs \frac12} 
	\end{align}

	\begin{algorithm}[ht]
		\caption{Estimate the partition of the anchor words $\I$ by $\wh \I$}\label{alg_I}
		\begin{algorithmic}[1]
			\Require matrix $\wh R\in\RR^{p\times p}$, $C_1$ and $Q\in\RR^{p\times p}$ such that $Q[j,\ell] := C_1\wh \delta_{j\ell}$ 
			\Procedure{FindAnchorWords}{$\wh R$, $Q$}
			\State initialize $\wh \I = \emptyset$
			\For{$i\in [p]$} 
			\State $ \wh a_i = \argmax_{1\le j\le p}\wh R_{ij}$
			\State set $\wh I^{(i)} =  \{\ell\in [p]: \wh R_{i\wh a_i}-\wh R_{il} \le Q[i,\wh a_i]+ Q[i,\ell]\}$ and $\textsc{Anchor}(i) = \textsc{True}$
			\For {$j \in \wh I^{(i)}$}
			\State $\wh a_j = \argmax_{1\le k\le p}\wh R_{jk}$
			\If {$\Bigl|\wh R_{ij}-\wh R_{j\wh a_j}\Bigr| > Q[i,j] + Q[j, \wh a_j]$}   
			\State $\textsc{Anchor}(i) =\textsc{False}$
			\State \textbf{break}
			\EndIf	
			\EndFor
			\If {$\textsc{Anchor}(i) $}
			\State $\wh \I = \textsc{Merge}(\wh I^{(i)}$, $\wh \I$)
			\EndIf
			\EndFor
			\State\Return $\wh \I = \{ \wh I_1, \wh I_2, \ldots, \wh I_{\wh K}\}$ 
			\EndProcedure
			\Statex
			
			\Procedure{Merge}{$\wh I^{(i)}$, $\wh\I$}
			\For {$G \in \wh \I$}
			\If {$G \cap \wh I^{(i)}\ne \emptyset$} 
			\State replace $G$ in $\wh \I$ by $G\cap \wh I^{(i)}$
			\State\Return $\wh \I$
			\EndIf
			\EndFor
			\State {$\wh I^{(i)} \in \wh \I$}
			\State\Return $\wh \I$
			\EndProcedure
		\end{algorithmic}
	\end{algorithm}

	\subsection{Estimation of the word-topic matrix $A$ with a given partition of anchor words}

	Given the estimated partition of anchor words 
	$\wh \I = \{\wh I_1, \ldots, \wh I_{\wh K}\}$ and its index set $\wh I = \cup_{k\in [\wh K]}\wh I_k$, 
	Algorithm \ref{alg_1} below estimates the matrix $A$.

	\cite{bing2020optimal} recommends to set $\lambda = 0$ whenever $\wh M$ is invertible and otherwise choose $\lambda$  large enough
	such that $\wh M + \lambda \bI_K$ is invertible. Specifically, \cite{bing2020optimal} recommends to choose $\lambda$ as 
	\begin{equation*}
		\lambda(t^*) = 0.01 \cdot  t^*  \cdot K\left({K\log (n\vee p) \over [\min_{i\in \wh I}(D_{X})_{ii}]n}\cdot  \frac{1}{n}\sum_{i=1}^n {1\over N_i}\right)^{1/2}.\\
	\end{equation*}
	where
	\[
	t^* = \arg\min\left\{t\in \{0,1,2,\ldots\}:\, \wh M + \lambda(t)\bI_K\text{ is invertible}\right\}.
	\]

	\begin{algorithm}[H]
		\caption{Sparse Topic Model solver (STM)
		}\label{alg_1}
		\begin{algorithmic}[1]
			\Require frequency data matrix $\bX\in\RR^{p\times n}$ with document lengths $N_1, \ldots, N_n$;  the partition of anchor words $\{I_1,\ldots, \wh I_{\wh K}\}$ and its index set $\wh I = \cup_{k\in [\wh K]}\wh I_k$, the tuning parameter $\lambda\ge 0$
			\Procedure{}{}
			\State compute $D_X = n^{-1}\diag(\bX\1_n)$, $\wh \Theta$ from (\ref{est_Theta}) and $\wh R$ from (\ref{def_R_hat})
			\State compute $\wh B_{\wh I\cdot}$ by  
			$\wh B_{i\cdot} = \be_k$ for each $i\in \wh I_k$ and  $k\in [\wh K]$
			\State compute $\wh M = \wh B_{\wh I\cdot}^+
			\wh R_{\wh I\wh I} \wh B_{\wh I\cdot}^{+\T}$ and $\wh H = \wh B_{\wh I\cdot}^+\wh R_{\wh I\wh I^c}$ with 
			$\wh B_{\wh I\cdot}^+= (\wh B_{\wh I\cdot}^\top \wh B_{\wh I\cdot})^{-1} \wh B_{\wh I\cdot}^\top$
			and $\wh I^c = [p]\setminus \wh I$
			\State solve $\wh B_{\wh I^c\cdot}$ from 
			\begin{alignat*}{2}
				\wh B_{j\cdot } &= 0, &&\quad \text{if }(D_X)_{jj}\le  {7\log(n\vee p) \over n} \left({1\over n} \sum_{i=1}^n{1 \over N_i}\right),\\
				\wh B_{j\cdot } &=\argmin_{\beta \ge 0,\ \|\beta\|_1 = 1}\beta^\top  (\wh M  + \lambda \bI_K)\beta - 2\beta^\top  \wh h^{(j)}, &&\quad \text{otherwise,}
			\end{alignat*}
			\indent for each $j\in \wh I^c$, with $\wh h^{(j)}$ being the corresponding column of $\wh H$.
			\State compute $\wh A$ by normalizing $D_X\wh B$ to unit column sums
			\State \Return $\wh A$
			\EndProcedure
		\end{algorithmic}
	\end{algorithm}

	\section{Proofs}\label{app_proofs}
	
	To avoid a proliferation of superscripts, we adopt simplified notation for $r^{(\ell)} = A\alpha^{(\ell)}$ with $\ell \in \{i,j\}$ and its related quantities throughout the proofs.
	In place of $r^{(i)} = A \alpha^{(i)}$ and $r^{(j)} = A \alpha^{(j)}$, we consider instead
	\[
	r = A\alpha, \quad s = A\beta 
	\]
	with their corresponding mixing measures 
	\[
	\balpha = \sum_{k=1}^K \alpha_k \delta_{A_k},\quad  \bbeta = \sum_{k=1}^K \beta_k \delta_{A_k}.
	\]
	Similarly, we write $\wh\balpha,\wh\bbeta$ and $\wt\balpha, \wt \bbeta$ for $\wh\balpha^{(i)},\wh\balpha^{(j)}$ and $\wt\balpha^{(i)}, \wt\balpha^{(j)}$, respectively.

	\subsection{Proofs for \cref{setup}}
	
	\subsubsection{Proof of \cref{embedding}}\label{app_proof_embedding}
	
	\begin{proof}
		We will prove  
		\begin{align}\label{eq_SWD}
			\swd(r,s) & = \inf_{\alpha, \beta \in \Delta_K: r = A \alpha, s = A \beta} \inf_{\gamma \in \Gamma(\balpha, \bbeta)} \sum_{\ell, k \in [K]} \gamma_{k \ell} d(A_k, A_\ell) \nonumber \\
			& = \inf_{\alpha, \beta \in \Delta_K: r = A \alpha, s = A \beta} W(\balpha, \bbeta; d)
		\end{align}
		under \cref{extreme_points}. In particular, \cref{embedding} follows immediately under \cref{ass_id}. 
		
		To prove (\ref{eq_SWD}), let us denote the right side of (\ref{eq_SWD}) by $D(r, s)$ and write $W(\balpha, \bbeta) := W(\balpha, \bbeta; d)$ for simplicity.
		We first show that $D$ satisfies \textbf{R1} and \textbf{R2}.
		To show convexity, fix $r, r', s, s' \in \cS$, and let $\alpha, \alpha', \beta, \beta' \in \Delta_K$ satisfy
		\begin{equation}\label{mixture_requirements}
			\begin{split}
				r &= A \alpha \\
				r' & = A \alpha' \\
				s &= A \beta \\
				s' &= A \beta'\,.
			\end{split}
		\end{equation}
		Then for any $\lambda \in [0, 1]$, it is clear that $\lambda r + (1-\lambda) r' = A (\lambda \alpha + (1-\lambda) \alpha')$, and similarly for $s$.
		Therefore
		\begin{align*}
			D(\lambda r + (1-\lambda) r', \lambda s + (1-\lambda)s') &\leq W(\lambda \balpha + (1-\lambda) \balpha', \lambda \bbeta + (1-\lambda) \bbeta')\\
			&\leq \lambda W(\balpha, \bbeta) + (1-\lambda) W(\balpha', \bbeta')\,,
		\end{align*}
		where the second inequality uses the convexity of the Wasserstein distance~\cite[Theorem 4.8]{Vil08}.
		
		Taking infima on both sides over $\alpha, \alpha', \beta, \beta' \in \Delta_K$ satisfying \cref{mixture_requirements} shows that $D(\lambda r + (1-\lambda) r', \lambda s + (1-\lambda)s') \leq \lambda D(r, s) + (1-\lambda) D(r', s')$, which establishes that $D'$ satisfies \textbf{R1}.
		To show that it satisfies \textbf{R2}, we note that \cref{extreme_points} implies that if $A_k = A \alpha$, then we must have $\balpha = \delta_{A_k}$.
		Therefore
		\begin{equation}
			D(A_k, A_\ell) = \inf_{\alpha, \beta \in \Delta_K: A_k = A \alpha, A_\ell = A \beta} W(\balpha, \bbeta) = W(\delta_{A_k}, \delta_{A_\ell})\,,
		\end{equation}
		Recalling that the Wasserstein distance is defined as an infimum over couplings between the two marginal measures and using the fact that the only coupling between $\delta_{A_k}$ and $\delta_{A_\ell}$ is $\delta_{A_k} \times \delta_{A_\ell}$, we obtain
		\begin{equation*}
			W(\delta_{A_k}, \delta_{A_\ell}) = \int d(A, A') \dd(\delta_{A_k} \times \delta_{A_\ell})(A, A') = d(A_k, A_\ell)\,,
		\end{equation*}
		showing that $D$ also satisfies \textbf{R2}.
		As a consequence, by  Definition (\ref{wa_def}), we obtain in particular that $\swd(r,s) \geq D(r, s)$ for all $r, s \in \cS$.
		
		To show equality, it therefore suffices to show that $\swd(r,s) \leq D(r, s)$ for all $r, s \in \cS$.
		To do so, fix $r, s \in \cS$ and let $\alpha, \beta$ satisfy $r = A \alpha$ and $s = A \beta$.
		Let $\gamma \in \Gamma(\balpha, \bbeta)$ be an arbitrary coupling between $\alpha$ and $\beta$, which we identify with an element of $\Delta_{K \times K}$.
		The definition of $\Gamma(\balpha, \bbeta)$ implies that, $\sum_{\ell=1}^K \gamma_{k \ell} = \alpha_k$ for all $k \in [K]$, and since $r = A \alpha$, we obtain $r = \sum_{\ell, k \in [K]} \gamma_{k \ell} A_k$.
		Similarly, $s = \sum_{\ell, k \in [K]} \gamma_{k \ell} A_\ell$.
		The convexity of $\swd$ (\textbf{R1}) therefore implies
		\begin{equation*}
			\swd(r,s) = \swd\left(\sum_{\ell, k \in [K]} \gamma_{k \ell} A_k, \sum_{\ell, k \in [K]} \gamma_{k \ell} A_\ell\right) \leq \sum_{\ell, k \in [K]} \gamma_{k \ell} \swd(A_k, A_\ell)\,.
		\end{equation*}
		Since $\swd$ agrees with $d$ on the original mixture components (\textbf{R2}), we further obtain
		\begin{equation*}
			\swd(r,s) \leq \sum_{\ell, k \in [K]} \gamma_{k \ell}\ d(A_k, A_\ell)\,. 
		\end{equation*}
		Finally, we may take infima over all $\alpha, \beta$ satisfying $r = A \alpha$ and $s = A \beta$ and $\gamma \in \Gamma(\balpha, \bbeta)$ to obtain
		\begin{align*}
			\swd(r,s) & \leq \inf_{\alpha, \beta \in \Delta_K: r = A \alpha, s = A \beta} \inf_{\gamma \in \Gamma(\balpha, \bbeta)} \sum_{\ell, k \in [K]} \gamma_{k \ell} d(A_k, A_\ell) \\
			& = \inf_{\alpha, \beta \in \Delta_K: r = A \alpha, s = A \beta} W(\balpha, \bbeta)\\
			&= D(r, s)\,.
		\end{align*}
		Therefore $\swd(r,s) \leq D(r, s)$, establishing that $\swd(r,s) = D(r, s)$, as claimed.
	\end{proof}

	\subsubsection{Proof of \cref{metric}}\label{app_proof_metric}
	\begin{proof}
		Under \cref{ass_id}, the map $\iota: \alpha \mapsto r = A \alpha$ is a bijection between $\Delta_K$ and $\cS$.
		Since $(\cX, d)$ is a (finite) metric space, $(\Delta_K, W)$ is a complete metric space~\cite[Theorem 6.18]{Vil08}, and \cref{embedding} establishes that the map from $(\Delta_K, W)$ to $(\cS, \swd)$ induced by $\iota$ is an isomorphism of metric spaces.
		In particular, $(\cS, \swd)$ is a complete metric space, as desired.
		%
		%
	\end{proof}

	\subsubsection{A counterexample to \cref{metric} under \cref{extreme_points}}\label{triangle_counterexample}
	Let $A_1, \dots, A_4$ be probability measures on the set $\{1, 2, 3, 4\}$, represented as elements of $\Delta_4$ as
	\begin{align*}
		A_1 & = (\frac 12, \frac 12, 0, 0) \\
		A_2 & = (0, 0, \frac 12, \frac 12) \\
		A_3 & = (\frac 12, 0, \frac 12, 0) \\
		A_4 & = (0, \frac 12, 0, \frac 12).
	\end{align*}
	We equip $\cX = \{A_1, A_2, A_3, A_4 \}$ with any metric $d$ satisfying the following relations:
	\begin{align*}
		d(A_1, A_2),  d(A_3, A_4) & < 1 \\
		d(A_1, A_3) & = 1\,.
	\end{align*}
	These distributions satisfy Assumption~\ref{extreme_points} but not Assumption~\ref{ass_id}, since $A_1 \in \mathrm{span}(A_2, A_3, A_4)$.
	Let $r = A_1$, $t = A_3$, and $s = \frac 12 (A_1 + A_2) = \frac 12 (A_3 + A_4).$
	Then $\swd(r, t) = d(A_1, A_3) = 1$, but $\swd(r, s) + \swd(s, t) = \frac 12 d(A_1, A_2) + \frac 12 d(A_3, A_4) < 1$.
	Therefore, \cref{extreme_points} alone is not strong enough to guarantee that $\swd$ as defined in \eqref{d_equivalence} is a metric.

	\subsection{Proofs for \cref{sec:discrete_swd}}
	\subsubsection{Proof of the upper bound in \cref{basic_ineq}}\label{app_proof_general_SWD_rate}
	\begin{proof}
		Start with
		\[
		\EE |W(\wh\balpha, \wh\bbeta; d)-W(\balpha, \bbeta; d)|  ~ \le
		~ \EE  W(\wh\balpha , \balpha; d) + \EE  W(\wh \bbeta, \bbeta; d), 
		\]
		which follows by  two applications of  the triangle inequality. Let $\wt \balpha  = \sum_{k=1}^K \wh \alpha_k  \delta_{A_k}$. The triangle inequality yields
		\[  
		W(\wh\balpha , \balpha; d)  \le W(\wh \balpha, \wt \balpha; d) + W(\wt \balpha, \balpha; d).
		\]
		The result follows by noting that 
		\begin{align*}
			& \EE W(\wh \balpha, \wt \balpha; d)  = \EE( \sum_{k=1}^K \wh \alpha_k d(\wh A_k, A_k)) \le \epsilon_A,\\
			& \EE W(\wt \balpha, \balpha; d) \le {1\over 2}\diam(\cA)\EE\|\wh \alpha - \alpha\|_1
		\end{align*}
		and using the same arguments for $ W(\wh \bbeta, \bbeta; d)$.
	\end{proof}

	\subsubsection{Proof of \cref{thm_lower_bound_SWD}}\label{app_proof_thm_lower_bound_SWD}
	\begin{proof}
		We first prove a reduction scheme. 
		Fix $1< \tau \le K$. Denote the parameter space of $W(\balpha,\bbeta;d)$ as 
		\[
		\Theta_W = \{
		W(\balpha,\bbeta;d): \alpha,\beta \in \Theta_\alpha(\tau), A\in \Theta_A
		\}.
		\]
		Let $\EE_W$ denote the expectation with respect to the law corresponding to any $W\in \Theta_W$.
		We first notice that for any subset $\bar \Theta_{W}\subseteq \Theta_{W}$, 
		\begin{align}\label{reduction}\nonumber
			\inf_{\wh W} \sup_{W \in \Theta_W} \EE_W |\wh W - W|& \ge  \inf_{\wh W} \sup_{W \in \bar \Theta_W} \EE_W |\wh W - W|\\
			&\ge {1\over 2}\inf_{\wh W\in \bar\Theta_W} \sup_{W \in \bar\Theta_W} \EE_W |\wh W - W|.
		\end{align}
		To see why the second inequality holds, for any $\bar \Theta_{W}\subseteq \Theta_{W}$, pick  any $\bar W \in \bar \Theta_{W}$. By the triangle inequality,
		\begin{align*}
			\sup_{W\in\bar\Theta_{W}}\EE_W |\bar W - W| &\le   \sup_{W\in\bar\Theta_{W}} \EE_W |\wh W - \bar W| +   \sup_{W\in\bar\Theta_{W}} \EE_W |\wh W - W|.
		\end{align*}
		Hence
		\begin{align*}
			\inf_{\bar W\in \bar\Theta_{W}}\sup_{W\in\bar\Theta_{W}}  \EE_W |\bar W - W|  &\le  \sup_{W\in\bar\Theta_{W}}  \inf_{\bar W\in \bar\Theta_{W}}\EE_W |\wh W - \bar W| +   \sup_{W\in\bar\Theta_{W}}\EE_W|\wh W - W|\\
			&\le  2\sup_{W\in\bar\Theta_{W}} \EE_W |\wh W -  W|.
		\end{align*}
		Taking the infimum over $\wh W $ yields the desired inequality. 
		
		We then use \eqref{reduction} to prove each term in the lower bound separately by choosing different subsets $\bar\Theta_W$. For simplicity, we write $\EE = \EE_W$.
		To prove the first term, let us fix  $A\in \Theta_A$, some set $S\subseteq [K]$ with $|S| = \tau$ and choose 
		\[
		\bar\Theta_W = \{
		W(\balpha,\bbeta;d):  \alpha_S, \beta_S\in \Delta_\tau, \alpha_{S^c} = \beta_{S^c}=0
		\}.
		\]
		It then follows that 
		\begin{align*}
			\inf_{\wh W} \sup_{W\in \Theta_W} \EE |\wh W  - W |&\ge {1\over 2}\inf_{\wh W} \sup_{\alpha_S,\beta_S\in \Delta_\tau} \EE |\wh W  - W(\balpha_S,\bbeta_S;d)|.
		\end{align*}
		Here we slightly abuse the notation by writing $\balpha_S = \sum_{k\in S} \alpha_k \delta_{A_k}$ and $\bbeta_S = \sum_{k\in S} \beta_k \delta_{A_k}$.
		Note that \cref{thm:lower_bound} in \cref{app_general_lower_bound} applies to an observation model where we have direct access to i.i.d.\ samples from the measures $\balpha_S$ and $\bbeta_S$ on known $\{A_k\}_{k\in S}$. There exists a \emph{transition} (in the sense of Le Cam~\cite{LeC86}) from this model to the observation model where only samples from $r$ and $s$ are observed; therefore, a lower bound from \cref{thm:lower_bound} implies a lower bound for the above display. By noting that, for any $\alpha_S \ne \beta_S$,
		\begin{align*}
			{2W(\balpha_S,\bbeta_S;d) \over  \|\alpha_S - \beta_S\|_1} &\le \max_{k,k'\in S}d(A_k, A_{k'})\le {C_d \over 2}\max_{k,k'\in S}\|A_k-A_{k'}\|_1 \le C_d,\\
			{2W(\balpha_S,\bbeta_S;d) \over  \|\alpha_S - \beta_S\|_1}&\ge  \min_{k\ne k'\in S}d(A_k, A_{k'})  \ge {c_d \over 2} \min_{k\ne k'\in S} \|A_k-A_{k'}\|_1 \ge  c_d ~ \ok_\tau,
		\end{align*}
		where the last inequality uses the definition in \eqref{def_kappa_A},
		an application of  \cref{thm:lower_bound} with $K = \tau$, $c_* =  \ok_\tau c_d $ and $C_* =  C_d$  yields the first term in the lower bound.
		
		To prove the second term, fix $\alpha = \be_1$ and $\beta = \be_2$, where $\{\be_1,\ldots,\be_K\}$ is the canonical basis of $\RR^K$. 
		For any $k\in [K]$, we write $I_k\subset [p]$ as the index set of anchor words in the  $k$th topic under \cref{ass_anchor}. The set of non-anchor words is defined as $I^c = [p]\setminus I$ with $I = \cup_k I_k$. We have $|I^c| = p-|I| = p-m$.  Define 
		\[
		\bar\Theta_W= \{
		W(\balpha,\bbeta;d): A\in \bar\Theta_A
		\}
		\]
		where
		\[
		\bar\Theta_A := \{A\in \Theta_A: |I_1| = |I_2|, A_{i1} = A_{j2} = \gamma, \text{for all }i\in I_1, j\in I_2\}
		\]
		for some $0< \gamma \le 1/|I_1|$ to be specified later. Using \eqref{reduction} we find that
		\begin{align}\label{disp_lower_A}\nonumber
			\inf_{\wh W} \sup_{W\in \Theta_W} \EE |\wh W  - W |&\ge {1\over 2}\inf_{\wh W\in \bar\Theta_W} \sup_{W\in \bar\Theta_W} \EE |\wh W  - W(\balpha,\bbeta;d)|\\
			&=  {1\over 2}\inf_{\wh A\in \bar\Theta_A} \sup_{A\in \bar\Theta_A} \EE \left|d(\wh A_1,\wh A_2) - d(A_1, A_2)\right|.
		\end{align}
		By similar arguments as before, it suffices to prove a lower bound for a stronger observational model where one has samples of $\text{Multinomial}_p(M, A_1)$ and $\text{Multinomial}_{p}(M, A_2)$ with $M \asymp nN / K$. Indeed, under the topic model assumptions, let $\{g_1, \ldots, g_K\}$ be a partition of $[n]$ with $| |g_k| - |g_{k'}| | \le 1$ for any $k\ne k'\in [K]$. If we treat the topic proportions $\alpha^{(1)}, \ldots, \alpha^{(n)}$ as known and choose them such that $\alpha^{(i)} = \be_k$ for any $i\in g_k$ and $k\in [K]$, then for the purpose of estimating $d(A_1,A_2)$, this model is equivalent to the observation model where one has samples  
		$$\wh Y^{(k)} \sim \text{Multinomial}_p(|g_k| N, A_k),\quad  \text{for } k\in \{1,2\}.$$ 
		Here $N|g_k| \asymp nN/K$. 
		Then we aim to invoke \cref{thm:lower_bound}  to bound from below \eqref{disp_lower_A}. To this end, first notice that 
		\[
		c_d \|A_1-A_2\|_1  \le 2 d(A_1, A_2) \le C_d \|A_1-A_2\|_1.
		\]
		By inspecting the proof of \cref{thm:lower_bound}, and by choosing 
		\[
		\gamma = {1\over |I_1| + |I^c|} = {1\over |I_1| + p - m} \asymp {1\over p}
		\]
		under $m\le cp$, 
		we can verify that \cref{thm:lower_bound} with $c_* = c_d  $ and $C_* = C_d$ yields
		\begin{align*}
			\inf_{\wh d} \sup_{A\in \bar\Theta_A} \EE \left|\wh d - d(A_1, A_2)\right|
			& ~ \gtrsim ~  {c_d^2  \over C_d}\sqrt{pK \over  nN \log^2(p)},
		\end{align*}
		completing the proof. 
		Indeed, in the proof of \cref{thm:lower_bound} and under the notations therein, we can specify the following choice of parameters:
		\begin{align*}
			\rho = \begin{bmatrix}
				0  \\ \1_{|I_2|} \\ 0 \\ \1_{|I^c|}
			\end{bmatrix} \gamma,\qquad \alpha = \begin{bmatrix}
				\1_{|I_1|} \\ 0 \\0 \\ \1_{|I^c|}
			\end{bmatrix} \gamma
		\end{align*}
		with the two priors of $P$ as 
		\begin{align*}
			\mu_0 &= \text{Law}\left(
			\alpha + \begin{bmatrix}
				0 \\ X_1 \\ X_2 \\ \vdots \\ X_{|I^c|}
			\end{bmatrix} {\epsilon \gamma}
			\right),\qquad 
			\mu_1 = \text{Law}\left(
			\alpha + \begin{bmatrix}
				0 \\ Y_1 \\ Y_2 \\ \vdots \\ Y_{|I^c|}
			\end{bmatrix} {\epsilon \gamma}
			\right).
		\end{align*}
		Here $X_i$ and $Y_i$ for $i\in \{1,\ldots, |I^c|\}$ are two sets of random variables constructed as in \cref{moment_matching}. Following the same arguments
		in the proof of \cref{thm:lower_bound}  with $\gamma \asymp 1/p$ and $|I^c| \asymp p$  yields the claimed result. 
	\end{proof}

	\subsection{Proofs for \cref{sec_inference}}
	\subsubsection{Proof of \cref{LT_SWD}}\label{app_proof_LT_SWD}
	\begin{proof}
		We first recall  the Kantorovich-Rubinstein dual formulation 
		\begin{align}\label{whynot}  
			W(\balpha, \bbeta; d) =   \sup_{f \in \cF}  f^\top(\alpha - \beta).
		\end{align}
		with $\cF$ defined in \eqref{def_F}. By  \eqref{whynot} and \eqref{SWD_bar}, we have the decomposition of $$\widetilde W -  W(\balpha, \bbeta; d)  = \rI + \rII, $$
		where 
		\begin{align*}
			\rI  &= \sup_{f\in \widehat \cF} f^\top  (\widetilde  \alpha- \wt \beta) - \sup_{f\in \cF} f^\top  (\wt \alpha- \wt \beta),\\  
			\rII &=  \sup_{f\in\cF} f^\top ( \wt \alpha- \wt \beta) - \sup_{f\in \cF} f^\top  ( \alpha- \beta).
		\end{align*}
		We control these terms separately. In Lemmas \ref{lem_W_dist_Haus_dist} \& \ref{lem_Haus_dist} of \cref{app_lemma_LT_SWD},  we show that 
		\[ 
		\rI \leq d_{H}(\cF, \wh \cF) ~ \| \wt \alpha - \wt \beta\|_1 \leq   \epsilon_A ~ \| \wt \alpha - \wt \beta\|_1 \overset{\eqref{cond_distr}}{=} \cO_\PP(\epsilon_A), 
		\]
		where, for two subsets $A$ and $B$ of $\RR^{K}$, the Hausdorff distance is defined as 
		\begin{equation*}
			d_H(A, B) = \max\left\{\sup_{a \in A} \inf_{b \in B} \|a - b\|_\infty, 
			~ \sup_{b \in B} \inf_{a \in A} \|a - b\|_\infty\right\}\,.
		\end{equation*} 
		Therefore for any estimator $\wh A$ for which $\sqrt{N}\epsilon_A \to 0$, we have $\sqrt{N} \rI  = o_\PP(1)$, and the asymptotic limit of the target equals   the asymptotic limit of $\rII$, scaled by $\sqrt{N}$.  Under (\ref{cond_distr}), by recognizing (as proved in \cref{prop_DHD}) that the function $h: \RR^K \to \RR$
		defined as $h(x) = \sup_{f\in \cF} f^\T x$ is Hadamard-directionally differentiable at $u = \alpha - \beta$ with derivative $h'_u: \RR^K\to \RR$ defined as 
		\[
		h'_u(x) = \sup_{f\in \cF: f^\T u = h(u)}f^\T x. 
		\]
		\cref{SWD_limit_distr} follows by an application the functional $\delta$-method in \cref{thm_delta_method}. The proof is complete.
	\end{proof}
	
	The following theorem is a variant of the $\delta$-method that is suitable for Hadamard directionally differentiable functions.

	\begin{theorem}[Theorem 1, \cite{Romish2004}]\label{thm_delta_method}
		Let $f$ be a function defined on a subset $F$ of $\RR^d$ with values in $\RR$, such that: (1) $f$ is Hadamard directionally differentiable at $u \in F$ with derivative $f'_u: F \to \RR$, and (2) there is a sequence of $\RR^d$-valued random variables $X_n$ and a sequence of non-negative numbers $\rho_n \to \infty$ such that $\rho_n(X_n - u) \overset{d}{\to} X$ for some random variable $X$ taking values in $F$.
		Then, $\rho_n(f(X_n) - f(u)) \overset{d}{\to} f'_u (X)$.
	\end{theorem}
	
	\subsubsection{Hadamard directional derivative}

	To obtain distributional limits, we employ the notion of \emph{directional Hadamard differentiability}.
	The differentiability of our functions of interest follows from well known general results~\cite[see, e.g.,][Section 4.3]{Bonnans2000}; we give a self-contained proof for completeness..

	\begin{proposition}\label{prop_DHD}
		Define a function $f: \RR^d \to \RR$ by
		\begin{equation}
			f(x) = \sup_{c \in \cC} c^\top x\,,
		\end{equation}
		for a convex, compact set $\cC \subseteq \RR^d$.
		Then for any $u \in \RR^d$ and sequences $t_n \searrow 0$ and $h_n \to h \in \RR^d$,
		\begin{equation}\label{eq:differentiable}
			\lim_{n \to \infty} \frac{f(u + t_n h_n) - f(u)}{t_n} = \sup_{c \in \cC_u} c^\top h\,,
		\end{equation}
		where $\cC_u \defeq \{c \in \cC: c^\top u = f(u)\}$.
		In particular, $f$ is  Hadamard-directionally differentiable.
	\end{proposition}
	\begin{proof}
		Denote the right side of \cref{eq:differentiable} by $g_u(h)$.
		For any $c \in \cC_u$, the definition of $f$ implies
		\begin{equation}
			\frac{f(u + t_n h_n) - f(u)}{t_n} \geq \frac{c^\top(u + t_n h_n) - c^\top u}{t_n} = c^\top h_n\,.
		\end{equation}
		Taking limits of both sides, we obtain
		\begin{equation}
			\liminf_{n \to \infty} \frac{f(u + t_n h_n) - f(u)}{t_n} \geq c^\top h\,,
		\end{equation}
		and since $c \in \cC_u$ was arbitrary, we conclude that
		\begin{equation}
			\liminf_{n \to \infty} \frac{f(u + t_n h_n) - f(u)}{t_n} \geq g_u(h)\,.
		\end{equation}
		
		In the other direction, it suffices to show that any cluster point of $$\frac{f(u + t_n h_n) - f(u)}{t_n}$$ is at most $g_u(h)$.
		For each $n$, pick $c_n \in \cC_{u + t_n h_n}$, which exists by compactness of $\cC$.
		By passing to a subsequence, we may assume that $c_n \to \bar c \in \cC$, and since
		\begin{equation}
			\bar c^\top u = \lim_{n \to \infty} c_n^\top (u + t_n h_n) = \lim_{n \to \infty} f(u + t_n h_n) = f(u)\,,
		\end{equation}
		we obtain that $\bar c  \in \cC_u$.
		On this subsequence, we therefore have
		\begin{equation}
			\limsup_{n \to \infty} \frac{f(u + t_n h_n) - f(u)}{t_n} \leq \lim_{n \to \infty} \frac{c_n^\top (u + t_n h_n) - c_n^\top u}{t_n} = \bar c^\top h \leq g_u(h)\,,
		\end{equation}
		proving the claim.
	\end{proof}

	\subsubsection{Lemmata used in the proof of \cref{LT_SWD}}\label{app_lemma_LT_SWD}

	\begin{lemma}\label{lem_W_dist_Haus_dist}
		For any $\alpha, \beta \in \RR^K$ and any subsets $A, B \subseteq \RR^K$,
		\[
		\left|
		\sup_{f \in A}f^\T(\alpha - \beta) -   \sup_{f \in B}f^\T(\alpha - \beta)
		\right| \le d_H(A, B) \|\alpha - \beta\|_1.
		\]
	\end{lemma}
	\begin{proof}
		Since the result trivially holds if either $A$ or $B$ is empty, we consider the case that $A$ and $B$ are non-empty sets. 
		Pick any $\alpha,\beta\in\RR^K$. 
		Fix $\ep > 0$, and let $f^* \in A$ be such that $f^{*\T}(\alpha - \beta) \geq \sup_{f \in A}f^\T(\alpha - \beta)- \ep$. By definition of $d_H(A, B)$, there exists $\wh f\in B$ such that $\|\wh f - f^*\|_\infty \le d_H(A, B) + \ep$. Then 
		\begin{align*}
			\sup_{f \in A}f^\T(\alpha - \beta) -  \sup_{f \in B}f^\T(\alpha - \beta)  &\leq f^{*\T}(\alpha - \beta) -  \sup_{f \in B}f^\T(\alpha - \beta) + \ep\\
			&\le (f^* - \wh f)^\T(\alpha - \beta) + \ep \\
			&\le d_H(A, B) \|\alpha - \beta\|_1 + \ep(1+\|\alpha - \beta\|_1).
		\end{align*}
		Since $\ep > 0$ was arbitrary, we obtain
		\begin{equation*}
			\sup_{f \in A}f^\T(\alpha - \beta) -  \sup_{f \in B}f^\T(\alpha - \beta) \leq d_H(A, B) \|\alpha - \beta\|_1.
		\end{equation*}
		Analogous arguments also yield 
		\[
		\sup_{f \in B}f^\T(\alpha - \beta) - \sup_{f \in A}f^\T(\alpha - \beta)  \le d_H(A, B) \|\alpha - \beta\|_1
		\]
		completing the proof. 
	\end{proof}

	\begin{lemma}\label{lem_Haus_dist}
		Let $\wh \cF$ and $\cF$ be defined in \eqref{def_F_hat} and \eqref{def_F}, respectively. Then
		\begin{align*}
			d_H(\wh \cF, \cF) &\le \max_{k, k'\in [K]}  \left|d(A_k, A_{k'}) - d(\wh A_k, \wh A_{k'})\right|.
		\end{align*}
	\end{lemma} 
	\begin{proof}
		For two subsets $A$ and $B$ of $\RR^{K}$, we recall the Hausdorff distance
		\begin{equation*}
			d_H(A, B) = \max\left\{\sup_{a \in A} \inf_{b \in B} \|a - b\|_\infty, 
			~ \sup_{b \in B} \inf_{a \in A} \|a - b\|_\infty\right\}\,.
		\end{equation*}
		For notational simplicity, we write
		\begin{equation}\label{def_E_kk'}
			E_{kk'} = \left|
			d(A_k, A_{k'}) - \bar d(k, k')
			\right|, 
		\end{equation}
		with $\bar d(k, k') = d(\wh A_k, \wh A_{k'})$, for any $k,k'\in [K]$.
		We first prove 
		\begin{equation}\label{target}
			\sup_{f\in \cF}\inf_{\wh f\in \wh \cF}\|\wh f - f\|_\infty \le 
			\max_{k, k'\in [K]}  \left|d(A_k, A_{k'}) - \bar d(k, k')\right| = \max_{k,k'\in [K]}E_{kk'}.
		\end{equation}
		For any $f\in \cF$, we know that $f_1 = 0$ and $f_k - f_\ell \le d(A_k, A_\ell)$ for all $k,\ell \in [K]$. Define 
		\begin{equation}\label{def_f_hat}
			\wh f_k = \min_{\ell \in [K]} \left\{
			\bar d(k, \ell) + f_\ell + E_{\ell 1} 
			\right\},\qquad \forall ~  k\in [K].
		\end{equation}
		It suffices to show  $\wh f\in \wh \cF$ and 
		\begin{equation}\label{eq_diff_f_sup}
			\|\wh f - f\|_\infty \le  \max_{k,k'\in [K]}E_{kk'}.
		\end{equation}
		Clearly, for any $k,k' \in [K]$, the definition in \cref{def_f_hat} yields 
		\begin{align*}
			\wh f_k - \wh f_{k'} & = \min_{\ell \in [K]} \left\{
			\bar d(k, \ell) + f_\ell + E_{\ell 1} 
			\right\} - \min_{\ell' \in [K]} \left\{
			\bar d(k', \ell') + f_{\ell'} + E_{\ell' 1}
			\right\}\\
			& = \min_{\ell \in [K]} \left\{
			\bar d(k, \ell) + f_\ell +E_{\ell 1}
			\right\} - \Bigl\{ \bar d(k', \ell^*) + f_{\ell^*} + E_{\ell^* 1} \Bigr\} \quad \quad \text{for some $\ell^* \in [K]$}\\
			& \le  \bar d(k, \ell^*) -  \bar d(k, \ell^*)\\
			&\le \bar d(k, k').
		\end{align*}
		Also notice that 
		\begin{align*}
			\wh f_1 &= \min \left\{
			\min_{\ell \in [K]\setminus\{1\}} \left\{
			\bar d(1, \ell) + f_\ell + E_{\ell 1} 
			\right\}, ~  \bar d(1, 1) + f_1 + E_{1 1} 
			\right\}\\
			& = \min \left\{
			\min_{\ell \in [K]\setminus\{1\}} \left\{
			\bar d(1, \ell) + f_\ell + E_{\ell 1} 
			\right\}, ~  0
			\right\},
		\end{align*}
		by using $f_1 = 0$ and $E_{11} = 0$.
		Since, for any $\ell \in [K]\setminus\{1\}$, by the definition of $\mathcal{F}$, 
		\[
		\bar d(1, \ell) + f_\ell + E_{\ell 1} \ge \bar d(1, \ell) - d(A_1, A_\ell) + E_{\ell 1}  \ge 0,
		\]
		we conclude $\wh f_1 = 0$, hence $\wh f \in \wh \cF$.  
		
		Finally, to show \cref{eq_diff_f_sup}, we have $f_1 = \wh f_1 = 0$ and for any $k \in [K]\setminus \{1\}$, 
		\begin{align*}
			\wh f_k - f_k & =  \min_{\ell \in [K]} \left\{
			\bar d(k, \ell) + f_\ell + E_{\ell 1}  - f_k
			\right\}  \le E_{k 1}
		\end{align*}
		and, conversely,
		\begin{align*}
			\wh f_k - f_k 
			& = \min_{\ell \in [K]} 
			\{\bar d(k, \ell) + f_\ell - f_k + E_{\ell 1}\}\\
			&\ge \min_{\ell \in [K]} 
			\{\bar d(k, \ell) - d(A_k, A_\ell) + E_{\ell 1}\}\\
			&\ge \min_{\ell \in [K]} \{
			-E_{k\ell} + E_{\ell 1}\}.
		\end{align*}
		The last two displays imply 
		\begin{align*}
			\|\wh f-f\|_\infty &\le 
			\max_{k\in [K]} \max\left\{
			E_{k1}, \max_{\ell \in [K]} \{ E_{k\ell} - E_{\ell 1}\}
			\right\}\\
			&\l= \max_{k,\ell \in [K]}\{
			E_{k\ell} - E_{\ell 1}
			\}\\
			&\le \max_{k,\ell} E_{k\ell},
		\end{align*}
		completing the proof of \cref{target}.

		By analogous arguments, we also have 
		\[
		\max_{\wh f\in \wh \cF}\min_{f\in  \cF}\|\wh f - f\|_\infty \le  \max_{k,k'\in [K]}E_{kk'},
		\]
		completing the proof. 
	\end{proof}

	\subsubsection{ Asymptotic normality proofs}  \label{app_proof_thm_asn}
	{
		Recall the notation in \cref{sec_background_topic_models}. We first state our most general result, Theorem \ref{thm_asn_general},  for any estimator $\wh A$ that satisfies \eqref{ass_A_simplex} -- \eqref{ass_A_error_row}. 
		

		\begin{theorem}\label{thm_asn_general}
			Assume $|\oJ \setminus \uJ|=\cO(1)$ and 
			\begin{equation}\label{conds_1}
				{\log^2(L) \over \uk_\tau^4 \wedge \ok_K^2}{(1\vee \xi)^3\over \Tm^3} {1\over N} \to 0,\qquad \text{as }n,N\to \i.
			\end{equation}
			Let $\wh A$ be any estimator satisfying \eqref{ass_A_simplex} -- \eqref{ass_A_error_row} with 
			\begin{equation}\label{conds_3}
				{2 (1\vee \xi) \over  \Tm} \epsilon_{\i,\i}  \le 1
			\end{equation}
			and 
			\begin{equation}\label{conds_2}
				{ \sqrt{N} \over \ok_K}{(1\vee \xi)\over  \Tm} \epsilon_{1,\i}\to 0, \qquad \text{as }n,N\to \i.
			\end{equation}
			Then as $n,N\to \i$, we have 
			\[
			\sqrt{N}  ~ \Sigma^{+{1\over 2}}\left(\wt \alpha - \alpha\right) \overset{d}{\to} \cN_K\left(0, \begin{bmatrix}
				\bI_{K-1} & 0 \\ 0 & 0
			\end{bmatrix}\right).
			\]
		\end{theorem}
		
		\cref{thm_asn} below is a particular case of the general \cref{thm_asn_general} corresponding to estimators $\wh A$ satisfying  \eqref{rate_A} and \eqref{rate_A_sup}. Its proof is an immediate consequence of the proof below, when one uses    \eqref{rate_A} and \eqref{rate_A_sup} in conjunction with \eqref{conds_1_simp} and \eqref{conds_2_simp}. 
		
		\begin{theorem}\label{thm_asn}
			Let $\wh A$ be any estimator such that \eqref{rate_A} and \eqref{rate_A_sup} hold. For any $\alpha_* \in \Theta_\alpha(\tau)$, assume $|\oJ \setminus \uJ|=\cO(1)$,
			\begin{equation}\label{conds_1_simp}
				\lim_{n,N\to \i}   {\log^2(L)\over \uk_\tau^4 \wedge \ok_K^2}{(1\vee \xi)^3\over \Tm^3 }   {1\over N}  = 0
			\end{equation}
			and 
			\begin{equation}\label{conds_2_simp}
				\lim_{n,N\to \i} {\log(L)\over \ok_K^2}{(1\vee \xi)^2\over  \Tm^2} {p\over n}=0.
			\end{equation}
			Then we have 
			\begin{equation}\label{eq_conv_distr}
				\lim_{n,N\to \i} \sqrt{N}  ~ \Sigma^{+{1\over 2}}\left(\wt \alpha - \alpha_*\right) \overset{d}{\to} \cN_K\left(0, \begin{bmatrix}
					\bI_{K-1} & 0 \\ 0 & 0
				\end{bmatrix}\right).
			\end{equation}
		\end{theorem}

		\begin{remark}[ Discussion of conditions of  \cref{thm_asn} and Theorem \ref{cor_asn_fixed} ]
			Quantities such as $\uk_\tau^{-1},\ok_K^{-1}$, $\xi$ and $\Tm^{-1}$ turn out to be important  quantities even in the finite sample analysis of   $\|\wh\alpha - \alpha_*\|_1$, as discussed in detail in   \cite{bing2021likelihood}. The first two quantities are bounded from above in many cases, 
			see for instance \cite[Remarks 1 \& 2]{bing2021likelihood} for a detailed discussion. \\
			
			The assumptions of  \cref{thm_asn} can be greatly simplified when the quantities $\uk_\tau^{-1},\ok_K^{-1}$, $\xi$ are bounded. Further, assume that the order of magnitude of the mixture weights does not depend on $n$ or $N$. Then, the smallest non-zero mixture weight will be bounded below by a constant, and $\Tm^{-1} = \cO(1)$. 
			
			Under these simplifying assumptions,  condition \eqref{conds_1_simp} reduces to $\log^2(p\vee n) = o(N)$, a mild  restriction on $N$, while condition \eqref{conds_2_simp} requires the  sizes of  $p$ and $n$ to satisfy $p\log(L) = o(n)$. The latter typically holds in the topic model context,  as the number of documents ($n$) is typically (much) larger  than the dictionary size ($p$).  The corresponding result is stated as 
			Theorem \ref{cor_asn_fixed}, in \cref{sec_ASN} of the main document. 
			
		\end{remark}
		
	}

	\begin{proof} We give below the proof of Theorem \ref{thm_asn_general}. 
		Recall that $X$ is the vector of empirical frequencies of $r = A\alpha$ with $J = \supp(X)$ and $\wh J = \{j\in[p]: \wh A_{j\cdot}^\T \wh\alpha > 0\}$.  
		The K.K.T. condition of (\ref{MLE}) states that
		\begin{equation}\label{eq_KKT}
			\1_K = \sum_{j\in J}{X_j \over \wh A_{j\cdot}^\T \wh \alpha}\wh A_{j\cdot} + {\lambda \over N}  = \sum_{j\in \wh J}{X_j \over \wh A_{j\cdot}^\T \wh \alpha}\wh A_{j\cdot} + {\lambda \over N}
		\end{equation}
		with $\lambda_k \ge 0$, $\lambda_k \wh \alpha_k = 0$ for all $k\in [K]$.
		The last step in (\ref{eq_KKT}) uses the fact that $J \subseteq \wh J$ with probability one. Recall from \eqref{def_V_td} that 
		\[
		\wt V   =  \sum_{j\in \wh J}{X_j \over \wh A_{j\cdot}^\T \wh\alpha ~ \wh A_j^\T \alpha}\wh A_{j\cdot}\wh  A_{j\cdot}^\T,\qquad  \Psi (\alpha)  =  \sum_{j \in \wh J} {X_j -  \wh A_{j\cdot}^\T  \alpha \over  \wh A_{j\cdot}^\T \alpha} \wh A_{j\cdot}.
		\] 
		Further recall from \eqref{KKTuse} that 
		\[
		\Psi(\wh \alpha) = - {\lambda \over N}  + \sum_{j\in \wh J^c}\wh A_{j\cdot}.
		\]
		Adding and subtracting terms in (\ref{eq_KKT}) yields
		\begin{align*}
			\wt V (\wh \alpha - \alpha) = \sum_{j\in \wh J}{X_j \over \wh A_{j\cdot}^\T \wh \alpha}{\wh A_{j\cdot}^\T (\wh \alpha - \alpha) \over \wh A_{j\cdot}^\T \alpha}\wh A_{j\cdot}  &=    \Psi (\alpha)+{\lambda \over N} + \sum_{j\in \wh J}\wh A_{j\cdot} - \1_K\\
			&=   \Psi (\alpha)  +  {\lambda \over N}  - \sum_{j\in \wh J^c}\wh A_{j\cdot}\\
			& = \Psi(\alpha) - \Psi(\wh \alpha)
		\end{align*}
		where the second line uses the fact that $\sum_{j=1}^p \wh  A_{j\cdot} = \1_K$.
		By \eqref{def_alpha_td}, we obtain 
		\[
		\wt \alpha - \alpha = \left(\bI_K - \wh V^{+}\wt V\right)(\wh \alpha - \alpha) + \wh V^+ \Psi(\alpha).
		\]
		By using the equality
		\begin{align*}
			\Psi(\alpha)  & =  \sum_{j \in \wh J} {(A_{j \cdot} - \wh A_{j\cdot})^\T \alpha \over \wh A_{j\cdot}^\T \alpha}\wh A_{j\cdot}+ \sum_{j\in \wh J}{X_j - r_j \over r_j}A_{j\cdot}  + \sum_{j\in \wh J}(X_j - r_j)\left(
			{\wh A_{j\cdot} \over \wh A_{j\cdot}^\T \alpha} - {A_{j\cdot} \over A_{j\cdot}^\T \alpha} 
			\right)
		\end{align*}
		we conclude 
		\begin{align*}
			\wt \alpha - \alpha &=  \rI + \rII + \rIII + \rIV 
		\end{align*} 
		where 
		\begin{align*}
			&\rI  = \wh V^{+}\sum_{j\in \wh J}{X_j - r_j \over r_j}A_{j\cdot}\\ 
			&\rII =  \left(\bI_K - \wh V^{+}\wt V\right)(\wh \alpha - \alpha)\\
			&\rIII =  \wh V^{+} \sum_{j \in \wh J} {(A_{j \cdot} - \wh A_{j\cdot})^\T \alpha \over \wh A_{j\cdot}^\T \alpha}\wh A_{j\cdot}\\ 
			&\rIV = \wh V^{+} \sum_{j\in \wh J}(X_j - r_j)\left(
			{\wh A_{j\cdot} \over \wh A_{j\cdot}^\T \alpha} - {A_{j\cdot} \over A_{j\cdot}^\T \alpha} 
			\right).
		\end{align*}
		For notational convenience, write 
		\begin{equation}\label{def_H}
			H  = A^\T D_{1/r} A,\qquad F = 	\bI_K - H^{1/2} \alpha\alpha^\T H^{1/2}
		\end{equation}
		such that 
		\[
		\Sigma = H^{-1} - \alpha\alpha^\T =  H^{-1/2} F H^{-1/2}.
		\]
		Here $D_{1/r}$ is a diagonal matrix with its $j$th entry equal to $1/r_{j}$ if $r_j>0$ and $0$ otherwise for any $j\in [p]$.
		From Lemma \ref{lem_H}, condition (\ref{conds_1}) implies $\kappa(A_{\oJ},K) \ge \kappa(A_{\uJ},K)>0$, hence $\rank(H) = K$.  Meanwhile, the $K\times K$ matrix $F$ has rank $(K-1)$ which can be seen from $\alpha^\T F \alpha = 0$. Furthermore, the eigen-decomposition of $F = \sum_{k=1}^K \sigma_k u_ku_k^\T$ satisfies 
		\begin{equation}\label{eigen_F}
			\sigma_1 = \cdots = \sigma_{K-1} = 1,\quad \sigma_K = 0.
		\end{equation}
		Let $\cS(F) \subset \RR^K$ be the space spanned by $u_1,\ldots, u_{K-1}$ and $\cK(F)$ be its complement space.   It remains to show that,   
		\begin{alignat}{2}\label{task_1}
			&\sqrt{N} u^\T H^{1/2} \rI   \overset{d}{\to} \cN(0,1), &&\textrm{for any $u\in \cS(F) \cap \bS^K$};\\\label{task_2}
			&\sqrt{N}u^\T H^{1/2}   \rI    = o_\PP(1), 
			&&\textrm{for any $u\in \cK(F) \cap \bS^K$};\\\label{task_3}
			&\sqrt{N} u^\T H^{1/2} (\rII + \rIII + \rIV) = o_\PP(1), \qquad &&
			\textrm{for any $u\in  \bS^K$}. 
		\end{alignat}
		
		These statements are proved in the following three subsections.

		\subsubsection{Proof of \cref{task_1}}
		We work on the event  (\ref{ass_A_error_row}) intersecting with  
		\begin{equation}\label{def_event_alpha}
			\cE_{\alpha} := \left\{
			4 	\rho \|\wh \alpha - \alpha\|_1 \le 1
			\right\}
		\end{equation}
		where, for future reference, we define
		\begin{equation}\label{bd_rho}
			\rho :=  {1\vee \xi \over \Tm} \ge \max_{j\in \oJ} {\|A_{j\cdot}\|_\i \over r_j}.
		\end{equation}
		The inequality is shown in  display (8) of \cite{bing2021likelihood}.
		From Lemma \ref{lem_prod_AT} and Lemma \ref{lem_Vhat} in \cref{app_tech_lemma_thm_asn}, we have  $\wh J = \oJ$ and $\wh V$ is invertible.
		Also note that $\lim_{n,N\to\infty}\PP(\cE_{\alpha}) = 1$ is ensured by Theorem \ref{thm_mle_alpha}  under conditions (\ref{conds_1}) -- (\ref{conds_2}). Indeed, we see that (\ref{cond_M1M2}) and (\ref{event_Ahat_col_prime}) hold under  (\ref{conds_1}) -- (\ref{conds_2}) such that Theorem \ref{thm_mle_alpha} ensures  
		\begin{align*}
			\PP\left\{\|\wh \alpha - \alpha\|_1 \lesssim  
			{1\over\ok_\tau} \sqrt{K\log(p) \over N} + {1 \over \ok_\tau^2}\epsilon_{1,\i} \right\} = 1-\cO(p^{-1}).
		\end{align*}
		It then follows from (\ref{conds_1}) -- (\ref{conds_2}) together with \eqref{ineq_kappas} that 
		$$
		\rho \|\wh \alpha - \alpha\|_1 =  \cO_{\PP}\left(
		{(1\vee \xi)\over \ok_\tau \Tm} \sqrt{K\log(L) \over N} + {(1\vee \xi ) \over \ok_\tau^2 \Tm }\epsilon_{1,\i}
		\right) = o_\PP(1).
		$$

		We now prove (\ref{task_1}). Pick any $u\in \cS(F) \cap \bS^K$.
		Since 
		\begin{equation}\label{decomp_task_1}
			u^\T  H^{1/2}\rI = u^\T  H^{1/2}  H^{-1}\sum_{j\in \wh J}{X_j - r_j \over r_j}A_{j\cdot} + u^\T H^{1/2}\left(
			\wh V^{-1} - H^{-1}\right)\sum_{j\in \wh J}{X_j - r_j \over r_j}A_{j\cdot}
		\end{equation}
		and $\wh J = \oJ$, 
		it remains to prove 
		\begin{align}\label{task_11}
			&\sqrt N  u^\T  H^{-1/2}  \sum_{j\in \oJ}{X_j - r_j \over r_j}A_{j\cdot} \overset{d}{\to} \cN(0,1),\\\label{task_12}
			&\sqrt N   u^\T H^{1/2}  \left(
			\wh V^{-1} - H^{-1}\right)\sum_{j\in \oJ}{X_j - r_j \over r_j}A_{j\cdot} = o_\PP(1).
		\end{align}
		For (\ref{task_11}), by recognizing that the left-hand-side of (\ref{task_11}) can be written as 
		\[
		{1\over \sqrt N}\sum_{t=1}^N u^\T  H^{-1/2} A^\T D_{1/r} (Z_t - r):=  {1\over \sqrt N}\sum_{t=1}^N Y_t.
		\]
		with $Z_t$, for $t\in [N]$, being i.i.d. samples of $\textrm{Multinomial}_p(1; r)$.  We will apply the Lyapunov central limit theorem. It is easy to see that $\EE[Y_t]  = 0$ and 
		\begin{align}\label{eq_F}\nonumber
			\EE[Y_t^2] &= u^\T H^{-1/2} A^\T D_{1/r}\left(
			D_{r} - rr^\T 
			\right)D_{1/r}\ A H^{-1/2} u\\\nonumber
			&= u^\T H^{-1/2}  \left(A^\T D_{1/r} A^\T  -A^\T   D_{1/r}  r r^\T   D_{1/r}  A\right) H^{-1/2}  u\\\nonumber
			& = u^\T \left(\bI_K - H^{-1/2}A^\T   D_{1/r}  r r^\T   D_{1/r}  A H^{-1/2}\right) u \\\nonumber
			&=u^\T \left(\bI_K - H^{1/2} \alpha   \alpha ^\T  H^{1/2}\right) u && (\star)\\
			&= u^\T F u = 1.
		\end{align}
		Here the $(\star)$ step uses the fact $H^{1/2}\alpha = H^{-1/2}A^\T D_{1/r}r$ implied by
		\[
		H \alpha = A^\T D_{1/r} A \alpha = A^\T  D_{1/r} r.
		\]
		It then remains to verify 
		\[
		\lim_{N\to\infty}{\EE[|Y_t|^3]\over \sqrt{N}} = 0.
		\]
		This follows by noting that 
		\begin{align*}
			\EE[|Y_t|^3] &\le 2\|u^\T H^{-1/2} A^\T D_{1/r}\|_\infty   && \textrm{by }\EE[Y_t^2]=1, \|Z_t-r\|_1 \le 2\\
			&= \max_{j\in \oJ} {|A_{j\cdot}^\T H^{-1/2} u|\over r_j}\\
			& \le  \max_{j\in \oJ} {\|A_{j\cdot}\|_\i \over r_j} \|H^{-1/2}u\|_1\\
			&\le \rho{ \|u\|_2 \over \kappa(A_{\oJ}, K)} && \textrm{by Lemma \ref{lem_H}}\\
			&\le {2 (1\vee \xi) \over  \ok_K\Tm} && \textrm{by (\ref{bd_rho}) and \eqref{def_kappas}}
		\end{align*}
		and by using   (\ref{conds_1}).  
		
		To prove (\ref{task_12}), starting with 
		\begin{align*}
			H^{1/2} \left(
			\wh V^{-1} - H^{-1}
			\right) & =  H^{1/2}\wh V^{-1} \left( H - 
			\wh V
			\right)H^{-1}, 
		\end{align*}
		we obtain 
		\begin{align*}
			&u^\T H^{1/2}  \left(
			\wh V^{-1} - H^{-1}\right)\sum_{j\in \oJ}{X_j - r_j \over r_j}A_{j\cdot}\\
			& \le \left\|u^\T H^{1/2}\wh V^{-1} H^{1/2}\right\|_2 \left\|H^{-1/2}(H-\wh V)H^{-1/2}\right\|_\op \left\|H^{-1/2} \sum_{j\in \oJ}
			{X_j - r_j \over r_j}A_{j\cdot}\right\|_2\\
			&\le \left\| H^{1/2}\wh V^{-1} H^{1/2}\right\|_\op \left\|H^{-1/2}(H-\wh V)H^{-1/2}\right\|_\op \left\|H^{-1/2} \sum_{j\in \oJ}
			{X_j - r_j \over r_j}A_{j\cdot}\right\|_2.
		\end{align*}
		By invoking Lemma \ref{lem_Vhat} and Lemma \ref{lem_oracle_error_whitening}, we have 
		\begin{align*}
			&\sqrt{N} u^\T H^{1/2}  \left(
			\wh V^{-1} - H^{-1}\right)\sum_{j\in \oJ}{X_j - r_j \over r_j}A_{j\cdot}\\
			& = \cO_{\PP}\left(
			\left(\sqrt{K} +  {\rho \over \kappa(A_{\oJ},K)} {K\over \sqrt N}\right) \left(
			{\rho^2 \over \kappa^2(A_{\oJ}, K)} \epsilon_{1,\i} +\rho \|\wh\alpha - \alpha\|_1 
			\right)
			\right).
		\end{align*}
		Recall that $K = \cO(1)$, $\kappa(A_{\oJ}, K) \ge \ok_K$ and $\rho \|\wh \alpha - \alpha\|_1 = o_\PP(1)$. Conditions (\ref{conds_1}) and (\ref{conds_2}) imply  (\ref{task_12}).

		\subsubsection{Proof of \cref{task_2}}
		In view of (\ref{decomp_task_1}) and (\ref{task_12}) together with $\wh J = \oJ$, it remains to prove, for any $u\in \cK(F) \cap \bS^K$,
		\[
		\sqrt N  u^\T  H^{-1/2}  \sum_{j\in \oJ}{X_j - r_j \over r_j}A_{j\cdot} = o_\PP(1).
		\]
		This follows from Lemma \ref{lem_oracle_error_whitening_null}, (\ref{bd_rho}) and   (\ref{conds_1}).  
		
		\subsubsection{Proof of \cref{task_3}}
		Pick any $u\in \bS^K$.  We  bound from above each terms in the left hand side of  (\ref{task_3}) separately on the event $\{\wh J = \oJ\}$. 
		
		First note that 
		\begin{align*}
			u^\T H^{1/2} \rII &=   u^\T H^{1/2} \left(\bI_K - \wh V^{-1}\wt V\right)(\wh \alpha - \alpha)\\
			&=  u^\T H^{1/2} \wh V^{-1} \left(\wh V - \wt V\right)H^{-1/2}H^{1/2}(\wh \alpha - \alpha)\\
			&\le \left\|H^{1/2} \wh V^{-1}H^{1/2} \right\|_{\op} \left\|
			H^{-1/2} \left(\wt V - \wh V\right)H^{-1/2}
			\right\|_{\op} \left\|H^{1/2}(\wh \alpha - \alpha)\right\|_2. 
		\end{align*} 
		The proof of Theorem 9 in \cite{bing2021likelihood} (see, the last display in Appendix G.2) shows that 
		\begin{align*}
			\left\|H^{1/2}(\wh \alpha - \alpha)\right\|_2 &= \cO_{\PP}
			\left(
			\sqrt{K\log(K)\over N} + {1 \over \kappa(A_{\oJ}, \tau)} \|\wh A_{\oJ} - A_{\oJ}\|_{1,\i} + {1 \over \kappa(A_{\oJ}, \tau)} {\log (p)\over N}
			\right)\\
			&=  \cO_{\PP}
			\left(
			{1\over \sqrt N} + {1 \over\ok_\tau}\epsilon_{1,\i}  
			\right).
		\end{align*}
		under condition (\ref{conds_1}). To further invoke Lemmas  \ref{lem_Vhat} \& \ref{lem_Vtilde}, note that conditions (\ref{cond_Vhat_strong}) and (\ref{cond_Htilde}) assumed therein hold under conditions (\ref{conds_1}) -- (\ref{conds_2}). In particular, 
		\[
		B \log (p) = {\rho \left(1 + \xi\sqrt{K-s} \right) \log(p) \over \kappa^{2}(A_{\oJ},K)\Tm} = \cO\left(
		{1\vee \xi^2 \over \ok_K^2}{\log (p) \over \Tm^2}
		\right) = o(N).
		\]
		Thus, invoking Lemmas  \ref{lem_Vhat} \& \ref{lem_Vtilde} together with conditions (\ref{conds_1}) -- (\ref{conds_2}) and (\ref{bd_rho}) yields
		\begin{align*}
			\sqrt{N} u^\T H^{1/2} \rII &= \cO_{\PP}\left(
			{\rho^2 \over \kappa^2(A_{\oJ}, K)}  \epsilon_{1,\i}   + \rho \|\wh\alpha - \alpha\|_1+ \sqrt{B \log(p)\over N} 
			\right)   = o_\PP(1).
		\end{align*}  
		
		Second, 
		\begin{align*}
			u^\T H^{1/2} \rIII &=u^\T H^{1/2} \wh V^{-1} \sum_{j \in \oJ} {(A_{j \cdot} - \wh A_{j\cdot})^\T \alpha \over \wh A_{j\cdot}^\T \alpha}\wh A_{j\cdot}\\
			&\le \left\|\wh V^{-1/2} \wh V^{-1/2}H^{1/2}u\right\|_1 \left\|\max_{k\in [K]} \sum_{j \in \oJ} {|A_{j k} - \wh A_{jk}| \|\alpha\|_1 \over \wh A_{j\cdot}^\T \alpha}\wh A_{j\cdot}\right\|_\i\\
			&\le \left\|\wh V^{-1/2} \wh V^{-1/2}H^{1/2}u\right\|_1 \max_{j\in \oJ}{\|\wh A_{j\cdot}\|_\i \over  \wh A_{j\cdot}^\T \alpha} \|\wh A_{\oJ} - A_{\oJ}\|_{1,\i}.
		\end{align*}
		By Lemma \ref{lem_prod_AT} and Lemma \ref{lem_Vhat}, we conclude 
		\begin{align*}
			\sqrt{N} u^\T H^{1/2} \rIII  &= \cO_{\PP}
			\left(
			{\rho \sqrt{N} \over \kappa(A_{\oJ},K)}	\|\wh V^{-1/2}H^{1/2}\|_\op   \|\wh A_{\oJ} - A_{\oJ}\|_{1,\i}
			\right) \\
			&  =  \cO_{\PP}
			\left(
			{(1\vee \xi) \sqrt{N} \over \ok_K\Tm}	\epsilon_{1,\i}
			\right) &&\textrm{by (\ref{bd_rho}) and \eqref{ineq_kappas}}\\
			& = o_\PP(1) &&\textrm{by (\ref{conds_2})}.
		\end{align*}
		
		Finally, by similar arguments and Lemma \ref{lem_Ahat}, we find 
		\begin{align*}
			\sqrt{N} u^\T H^{1/2} \rIV & \le  \sqrt{N} \left\| \wh V^{-1}H^{1/2}u\right\|_1 \left\|\sum_{j\in  \oJ}(X_j - r_j)\left(
			{\wh A_{j\cdot} \over \wh A_{j\cdot}^\T \alpha} - {A_{j\cdot} \over A_{j\cdot}^\T \alpha} 
			\right)\right\|_\i\\
			& = \cO_{\PP}\left(
			{\rho \sqrt{N} \over \kappa(A_{\oJ},K)}\left(
			\|\wh A_{\oJ} - A_{\oJ}\|_{1,\i} + 
			\sum_{j\in \oJ\setminus \uJ} {\|\wh A_{j\cdot}-A_{j\cdot}\|_\i \over r_j}{\log(p)\over N}\right)
			\right)\\
			& = \cO_{\PP}\left(
			{(1\vee \xi)  \sqrt{N} \over \ok_K\Tm}
			\epsilon_{1,\i} +  {(1\vee \xi)   \over \ok_K\Tm} {\log(p)\over \sqrt N}
			\right)\\
			& = o_\PP(1).
		\end{align*} 
		
		This completes the proof of Theorem \ref{thm_asn}.
	\end{proof}

	\subsubsection{Technical lemmas used in the proof of Theorem \ref{thm_asn}}\label{app_tech_lemma_thm_asn}

	Recall that 
	\[
	\oJ = \{j\in[p]: r_j > 0\}, \qquad \wh J = \{j\in[p]: \wh A_{j\cdot}^\T \wh \alpha > 0\}.
	\]
	\cref{lem_prod_AT} below provides upper and lower bounds of $\wh A_{j\cdot}^\T \alpha$ and $\wh A_{j\cdot}^\T \wh\alpha$. 
	
	\begin{lemma}\label{lem_prod_AT}
		On the event (\ref{ass_A_error_row}), we have 
		\[
		{1\over 2} r_j \le  \wh A_{j\cdot}^\T  \alpha \le {3\over 2} r_j,\qquad \forall j\in \bar J,
		\]
		and 
		\[
		\max_{j\in \oJ} {\|\wh A_{j\cdot}\|_\i \over r_j} \le {3\over 2}\rho,\qquad  \max_{j\in \oJ} {\|\wh A_{j\cdot}\|_\i \over \wh A_{j\cdot}^\T \alpha} \le 3\rho.
		\]
		On the event (\ref{ass_A_error_row})  intersecting (\ref{def_event_alpha}),
		we further have 
		\[
		{1\over 4} r_j \le  \wh A_{j\cdot}^\T \wh \alpha \le 2 r_j,\qquad \forall j\in \bar J,
		\]
		hence $$\bar J = \wh J.$$
	\end{lemma}
	
	\begin{proof}
		For the first result, for any $j\in  \oJ$, we have 
		\begin{align*}
			{3\over 2}r_j \ge \wh A_{j\cdot}^\T  \alpha + 	 \|\wh A_{j\cdot} - A_{j\cdot}\|_\i  \| \alpha\|_1	\ge \wh A_{j\cdot}^\T  \alpha  &\ge A_{j\cdot}^\T \alpha -  \|\wh A_{j\cdot} - A_{j\cdot}\|_\i  \| \alpha\|_1\ge {1\over 2}r_j.
		\end{align*}
		Furthermore,  by using $r_j = A_{j\cdot}^\T \alpha \le \|A_{j\cdot}\|_\i$,
		\[
		{\|\wh A_{j\cdot}\|_\i \over r_j} \le {\|A_{j\cdot}\|_\i \over r_j} + {{\|\wh A_{j\cdot} - A_{j\cdot}\|_\i \over r_j}} \le  {3\over 2}{\|A_{j\cdot}\|_\i \over r_j}.
		\]
		Similarly,  
		\begin{align*}
			{\|\wh A_{j\cdot}\|_\i \over \wh A_{j\cdot}^\T \alpha} &\le{\|\wh A_{j\cdot}\|_\i  \over  A_{j\cdot}^\T \alpha- |(\wh A_{j\cdot}-A_{j\cdot})^\T \alpha|}  \le 
			{\|\wh A_{j\cdot}\|_\i   \over  r_j- \|\wh A_{j\cdot}-A_{j\cdot}\|_\i} \le {3\|A_{j\cdot}\|_\i \over r_j}. 
		\end{align*} 
		Taking the union over $j\in \oJ$ together with the definition of $\rho$ proves the first claim. 
		
		To prove the second result, for any $j\in  \oJ$, we have 
		\begin{align*}
			\wh A_{j\cdot}^\T \wh \alpha  &\ge A_{j\cdot}^\T \wh\alpha -  \|\wh A_{j\cdot} - A_{j\cdot}\|_\i  \|\wh \alpha\|_1\\
			&\ge r_j - \|A_{j\cdot}\|_\i \|\wh \alpha - \alpha\|_1 -  \|\wh A_{j\cdot} - A_{j\cdot}\|_\i\\
			&\ge   r_j / 4.
		\end{align*}
		Similarly, 
		\begin{align*}
			\wh A_{j\cdot}^\T \wh \alpha  &\le A_{j\cdot}^\T \wh\alpha +  \|\wh A_{j\cdot} - A_{j\cdot}\|_\i  \|\wh \alpha\|_1\\
			&\le r_j +\|A_{j\cdot}\|_\i \|\wh \alpha - \alpha\|_1 +  \|\wh A_{j\cdot} - A_{j\cdot}\|_\i\\
			&\le 2 ~  r_j.
		\end{align*}
		The second claim follows immediately from the first result. 
	\end{proof}

	\medskip 
	
	The following lemma studies the properties related with the matrix $\wh V$ in (\ref{def_V_hat}).

	\begin{lemma}\label{lem_Vhat}
		Suppose 
		\begin{equation}\label{cond_Vhat}
			2\|\wh A_{\oJ}-A_{\oJ}\|_{1,\i} \le \kappa(A_{\oJ}, K).
		\end{equation}
		On the event (\ref{ass_A_error_row})  intersecting (\ref{def_event_alpha}),  $\wh V$  satisfies the following statements.
		\begin{enumerate}
			\item[(a)] For any $u\in \RR^K$, 
			\[
			u^\T \wh V u  \ge   {1\over 4}\kappa^2(A_{\oJ}, K) \|u\|_1^2.
			\]
			\item[(b)] 
			\begin{align*}
				\left\| H^{-1/2}\left(\wh V - H\right) H^{-1/2} \right\|_\op \le  & ~  {14\rho^2 \over \kappa^2(A_{\oJ}, K)}  \|\wh A_{\oJ} - A_{\oJ}\|_{1,\i}   + \rho \|\wh\alpha - \alpha\|_1.
			\end{align*} 
			\item[(c)]  If additionally 
			\begin{equation}\label{cond_Vhat_strong}
				{\rho^2 \over \kappa^2(A_{\oJ}, K)}  \|\wh A_{\oJ} - A_{\oJ}\|_{1,\i}  \le {1\over 56 }		
			\end{equation}
			holds, 
			then
			$$
			{1\over 2}\le \lambda_K(H^{-1/2}\wh V H^{1/2}) \le \lambda_1(H^{-1/2}\wh V H^{1/2}) \le 2.
			$$
		\end{enumerate}
	\end{lemma}
	\begin{proof}
		We work on  the event (\ref{ass_A_error_row})  intersecting (\ref{def_event_alpha}).
		From Lemma \ref{lem_prod_AT}, we know 
		$\wh J = \oJ$.  
		
		To prove part (a), using the fact that 
		$$\sum_{j \in \wh J} \wh A_{j\cdot}^\T \wh\alpha = \sum_{j=1}^p \wh A_{j\cdot}^\T \wh\alpha =  \1_K^\T \wh\alpha = 1, $$
		we have, for any $u\in \RR^K$,
		\begin{align*}
			u^\T \wh V u &= \sum_{j \in \wh J} {1 \over \wh A_{j\cdot}^\T \wh \alpha} (\wh A_{j\cdot}^\T u)^2 \left(\sum_{j \in \wh J} \wh A_{j\cdot}^\T \wh\alpha \right)\\
			&\ge \left(\sum_{j \in \wh J} |\wh A_{j\cdot}^\T u|\right)^2	&\textrm{by Cauchy Schwarz}\\
			&\ge \left(\| A_{\oJ} u\|_1 - \|(\wh A_{\oJ} -A_{\oJ})u\|_1 \right)^2	 &\textrm{by }\wh J = \oJ\\
			&\ge \|u\|_1^2\left(\kappa(A_{\oJ}, K)- \|\wh A_{\oJ} -A_{\oJ}\|_{1,\i} \right)^2.
		\end{align*}
		Invoking condition (\ref{cond_Vhat}) concludes the result. 
		
		To prove part (b), by definition,
		\begin{align*}
			\wh V - H  &= \sum_{j \in \oJ}
			{1\over \wh A_{j\cdot}^\T \wh \alpha}
			\left(\wh A_{j\cdot}\wh A_{j\cdot}^\T - A_{j\cdot} A_{j\cdot}^\T \right) 
			+\sum_{j \in \oJ}
			{A_{j\cdot}^\T  \alpha - \wh A_{j\cdot}^\T \wh \alpha \over \wh A_{j\cdot}^\T \wh \alpha}
			{A_{j\cdot} A_{j\cdot}^\T  \over A_{j\cdot}^\T  \alpha}.
		\end{align*}
		Therefore, for any $u\in \bS^{K}$, 
		\[
		u^\T H^{-1/2}(	\wh V - H) H^{-1/2}u \le \rI +\rII + \rIII + \rIV
		\]
		where 
		\begin{align*}
			\rI &=  \sum_{j \in \oJ}
			{1\over \wh A_{j\cdot}^\T \wh \alpha}
			u^\T H^{-1/2} \wh A_{j\cdot}(\wh A_{j\cdot} - A_{j\cdot})^\T H^{-1/2}u\\
			\rII &= \sum_{j \in \oJ}
			{1\over \wh A_{j\cdot}^\T \wh \alpha}
			u^\T H^{-1/2} A_{j\cdot}(\wh A_{j\cdot} - A_{j\cdot})^\T H^{-1/2}u\\
			\rIII &=   \sum_{j \in \oJ}
			{(A_{j\cdot}- \wh A_{j\cdot})^\T \wh \alpha \over \wh A_{j\cdot}^\T \wh \alpha}
			u^\T H^{-1/2} 	{A_{j\cdot} A_{j\cdot}^\T  \over A_{j\cdot}^\T  \alpha} H^{-1/2} u\\
			\rIV &= \sum_{j \in \oJ}
			{A_{j\cdot}^\T  (\alpha -  \wh \alpha) \over \wh A_{j\cdot}^\T \wh \alpha}
			u^\T H^{-1/2}	{A_{j\cdot} A_{j\cdot}^\T  \over A_{j\cdot}^\T  \alpha} H^{-1/2} u.
		\end{align*}
		By using Lemmas \ref{lem_prod_AT} \& \ref{lem_H}, we obtain 
		\begin{align*}
			\rI &\le \left \|H^{-1/2}u \right\|_1^2 ~ \max_{j \in  \oJ}{ \|\wh A_{j\cdot}\|_\i \over \wh A_{j\cdot}^\T \wh\alpha} \max_{k\in [K]}\sum_{j \in \oJ} |\wh A_{jk} - A_{jk}| \le {6\rho \over \kappa^2(A_{\oJ}, K)}\|\wh A_{\oJ} - A_{\oJ}\|_{1,\i}\\
			\rII &\le \left \|H^{-1/2}u \right\|_1^2 ~ \max_{j \in  \oJ}{ \|A_{j\cdot}\|_\i \over \wh A_{j\cdot}^\T \wh\alpha} \max_{k\in [K]}\sum_{j \in \oJ} |\wh A_{jk} - A_{jk}| \le {4\rho \over \kappa^2(A_{\oJ}, K)}\|\wh A_{\oJ} - A_{\oJ}\|_{1,\i}\\
			\rIII&\le \left \|H^{-1/2}u \right\|_1^2 ~ \max_{j \in  \oJ}{ \|A_{j\cdot}\|_\i \over \wh A_{j\cdot}^\T \wh\alpha}{ \|A_{j\cdot}\|_\i \over  A_{j\cdot}^\T \alpha} \max_{k\in [K]}\sum_{j \in \oJ} |\wh A_{jk} - A_{jk}| \le {4\rho^2 \over \kappa^2(A_{\oJ}, K)}\|\wh A_{\oJ} - A_{\oJ}\|_{1,\i}.
		\end{align*}
		By also noting that 
		\begin{align*}
			\rIV&\le   \|\wh \alpha - \alpha\|_1  \max_{j \in  \oJ}{ \|A_{j\cdot}\|_\i \over \wh A_{j\cdot}^\T \wh\alpha} \sum_{j \in \oJ} u^\T H^{-1/2}	{A_{j\cdot} A_{j\cdot}^\T  \over A_{j\cdot}^\T  \alpha} H^{-1/2} u \le 4\rho   \|\wh \alpha - \alpha\|_1
		\end{align*}
		and using the fact that $\rho \ge 1$,
		we have proved part (b).
		
		Finally, part (c) follows from part (b), (\ref{cond_Vhat_strong}) and  Weyl's inequality. Indeed,
		\begin{align*}
			1 - \left\| H^{-1/2}\left(\wh V -  H\right) H^{-1/2} \right\|_\op &\le  \lambda_K(H^{-1/2}\wh V H^{-1/2})\\
			& \le \lambda_1(H^{-1/2}\wh V H^{-1/2})\\
			&\le 
			1 + 
			\left\| H^{-1/2}\left(\wh V -  H\right) H^{-1/2} \right\|_\op. 
		\end{align*}
	\end{proof}
	
	\medskip

	The following lemmas controls the operator norm bounds for the normalized estimation errors of $H$ for various estimators. 
	Let 
	\begin{equation}\label{def_H_td}
		\wt H =  \sum_{j \in  \oJ} {X_j \over r_j^2}  \wh A_{j\cdot}\wh A_{j\cdot}^\T.
	\end{equation}
	Recall that 
	\[
	B = {\rho \left(1 + \xi\sqrt{K-s} \right) \over \kappa^{2}(A_{\oJ},K)\Tm}
	\]
	
	\begin{lemma}\label{lem_Htilde}
		Suppose $|\oJ \setminus \uJ| \le C$ for some constant $C>0$. Then, on the event (\ref{ass_A_error_row}),  with probability $1-2p^{-1}$, we have 
		\begin{align*}
			\left\| H^{-1/2}\left(\wt H -  H\right) H^{-1/2} \right\|_\op \lesssim ~  &   {\rho \over \kappa^2(A_{\oJ}, K)} \|\wh A_{\oJ} - A_{\oJ}\|_{1,\i}    + \sqrt{B \log(p)\over N} + {B \log(p)\over  N}. 
		\end{align*}
		Moreover,  if there exists a sufficiently large constant $C'>0$ such that 
		\begin{equation}\label{cond_Htilde}
			{\rho \over \kappa^2(A_{\oJ}, K)}\|\wh A_{\oJ} - A_{\oJ}\|_{1,\i}    \le C',\qquad {B \log(p) \over  N}\le C',
		\end{equation}
		then, on the event (\ref{ass_A_error_row}), with probability $1-2p^{-1}$, 
		\[
		{1\over 2} \le \lambda_K(H^{-1/2}\wt H H^{-1/2}) \le \lambda_1(H^{-1/2}\wt H H^{-1/2}) \le 2.
		\]
	\end{lemma}
	\begin{proof}
		Define 
		\[
		\wh H =  \sum_{j\in \oJ}{X_j  \over r_j^2}A_{j\cdot}A_{j\cdot}^\T.
		\]
		By 
		\[
		\left\| H^{-1/2}\left(\wt H -  H\right) H^{-1/2} \right\|_\op \le 	\left\| H^{-1/2}\left(\wt H -\wh  H\right) H^{-1/2} \right\|_\op + 	\left\| H^{-1/2}\left(\wh H -  H\right) H^{-1/2} \right\|_\op
		\]
		and Lemma \ref{lem_H_deviation}, it remains to bound from above the first term on the right hand side.  To this end, pick any $u \in \bS^K$. Using the definitions of $\wt H$ and $\wh H$ gives 
		\begin{align*}
			&u^\T H^{-1/2}\left(\wt H - \wh H\right) H^{-1/2} u\\ 
			&=\sum_{j \in \oJ}
			{X_j \over r_j^2} u^\T H^{-1/2} \left(
			\wh A_{j\cdot} \wh A_{j\cdot}^\T -  A_{j\cdot}  A_{j\cdot}^\T
			\right)
			H^{-1/2} u\\
			&=  \sum_{j \in \oJ}
			{X_j \over r_j^2} \left[ u^\T H^{-1/2}  \wh A_{j\cdot} (\wh A_{j\cdot} -  A_{j\cdot} )^\T
			H^{-1/2} u  + u^\T H^{-1/2}  A_{j\cdot} (\wh A_{j\cdot} -  A_{j\cdot} )^\T
			H^{-1/2} u \right].
		\end{align*} 
		The first term is bounded from above by
		\begin{align*}
			&\|H^{-1/2}u\|_1^2 \max_{k\in [K]} \sum_{j \in \oJ} {X_j \over r_j} |\wh A_{jk} - A_{jk}|{\|\wh A_{j\cdot}\|_\i\over r_j}\\
			&\le \|H^{-1/2}u\|_1^2 \max_{k\in [K]} \sum_{j \in \oJ} {X_j \over r_j} |\wh A_{jk} - A_{jk}| 
			\max_{j\in \oJ} {\|\wh A_{j\cdot}\|_\i\over r_j}\\
			&\le  {3 \rho\over 2\kappa^{2}(A_{\oJ},K)} \max_{k\in [K]} \sum_{j \in \oJ} {X_j \over r_j} |\wh A_{jk} - A_{jk}|
		\end{align*}
		where we used Lemmas \ref{lem_prod_AT} \& \ref{lem_H} in the last step.  Similarly, we have 
		\begin{align*}
			\sum_{j \in \oJ}
			{X_j \over r_j^2} u^\T H^{-1/2}   A_{j\cdot} (\wh A_{j\cdot} -  A_{j\cdot} )^\T
			H^{-1/2} u
			&\le {\rho\over \kappa^{2}(A_{\oJ},K)} \max_{k\in [K]} \sum_{j \in \oJ} {X_j \over r_j} |\wh A_{jk} - A_{jk}|
		\end{align*}
		By the arguments in the proof of \cite[Theorem 8]{bing2021likelihood}, on the event $\cE$, we have 
		\begin{align}\label{bd_X_r_diff_A}\nonumber
			\max_{k\in [K]} \sum_{j \in \oJ} {X_j \over r_j} |\wh A_{jk} - A_{jk}|
			&\le  \max_{k\in [K]} \sum_{j \in \oJ}  |\wh A_{jk} - A_{jk}|\left(
			1 + {7\log (p) \over 3r_j N}
			\right)\\\nonumber
			&\le 
			4\|\wh A_{\oJ} - A_{\oJ}\|_{1,\i} + {14\over 3}\sum_{j\in \oJ\setminus \uJ} {\|\wh A_{j\cdot} - A_{j\cdot}\|_\i \over r_j} {\log(p)\over N}\\
			&\le 4\|\wh A_{\oJ} - A_{\oJ}\|_{1,\i} + {7C'\over 3} {\log(p)\over N}
		\end{align}
		where in the last step we uses $|\oJ \setminus \uJ|\le C'$ and (\ref{ass_A_error_row}). 
		The first result then follows by combining the last three bounds and Lemma \ref{lem_H_deviation} with $t = 4\log(p)$.
		
		Under (\ref{cond_Htilde}), the second result follows immediately by using Weyl's inequality. 
	\end{proof}

	Recall that $\wh V$ and $\wt V$ are defined in (\ref{def_V_hat})  and (\ref{def_V_td}), respectively. The following lemma bounds their difference normalized by $H^{-1/2}$ in operator norm. 
	
	\begin{lemma}\label{lem_Vtilde}
		Suppose $|\oJ \setminus \uJ| \le C$ for some constant $C>0$. Assume (\ref{cond_Vhat_strong}) and (\ref{cond_Htilde}). On the event (\ref{ass_A_error_row}) intersecting with (\ref{def_event_alpha}), with probability $1-2p^{-1}$,
		we have 
		\begin{align*}
			\left\|
			H^{-1/2}\left(\wh V - \wt V\right) H^{-1/2} 
			\right\|_\op &\lesssim  ~  \rho \|\wh\alpha - \alpha\|_1 +    {\rho^2 \over \kappa^2(A_{\oJ}, K)}  \|\wh A_{\oJ} - A_{\oJ}\|_{1,\i}  + \sqrt{B \log(p)\over N}.
		\end{align*}
	\end{lemma}
	\begin{proof}
		By 
		\[
		\wh V - \wt V = \wh V - H + H - \wt H + \wt H - \wt V,
		\]
		we have
		\begin{align*}
			\left\|H^{-1/2}\left(\wh V - \wt V\right) H^{-1/2}\right\|_\op &\le  \left\|H^{-1/2}\left(\wh V - H \right) H^{-1/2}\right\|_\op  + \left\|H^{-1/2}\left(\wt H -  H\right) H^{-1/2}\right\|_\op \\
			&\quad + \left\|H^{-1/2}\left(\wt V - \wt H\right) H^{-1/2}\right\|_\op.
		\end{align*}
		In view of Lemmas  \ref{lem_Vhat} \& \ref{lem_Htilde},
		it suffices to bound from above 
		$\|H^{-1/2}\left(\wt V - \wt H\right) H^{-1/2}\|_\op$. Recalling from (\ref{def_V_td}) and (\ref{def_H_td}), it follows that 
		\begin{align*}
			&\left\|H^{-1/2}\left(\wt V - \wt H\right) H^{-1/2}\right\|_\op\\
			&=  \left\| H^{-1/2}\sum_{j \in  \oJ} \left({X_j \over \wh A_{j\cdot}^\T \wh \alpha \wh A_{j\cdot}^\T \alpha} -  		
			{X_j \over r_j^2}
			\right) \wh A_{j\cdot} \wh A_{j\cdot}^\T 
			H^{-1/2} \right\|_\op \\
			&=  \left\| H^{-1/2}\sum_{j \in  \oJ} {X_j \over (\wh A_{j\cdot}^\T \wh \alpha) (\wh A_{j\cdot}^\T \alpha ) r_j^2} \left( 
			\wh A_{j\cdot}^\T \wh \alpha  \wh A_{j\cdot}^\T  \alpha  -  A_{j\cdot}^\T  \alpha  A_{j\cdot}^\T  \alpha 
			\right) \wh A_{j\cdot} \wh A_{j\cdot}^\T 
			H^{-1/2}  \right\|_\op\\
			&\le \left\| H^{-1/2}\sum_{j \in  \oJ} {X_j \over (\wh A_{j\cdot}^\T \wh \alpha)   r_j^2}  
			\wh A_{j\cdot}^\T (\wh \alpha-\alpha)  
			\wh A_{j\cdot} \wh A_{j\cdot}^\T 
			H^{-1/2}  \right\|_\op\\
			&+\quad  \left\| H^{-1/2}\sum_{j \in  \oJ} {X_j \over (\wh A_{j\cdot}^\T \wh \alpha) (\wh A_{j\cdot}^\T \alpha ) r_j^2} \left[
			\left(\wh A_{j\cdot}^\T  \alpha  \right)^2 -  \left(A_{j\cdot}^\T  \alpha\right)^2
			\right] \wh A_{j\cdot} \wh A_{j\cdot}^\T 
			H^{-1/2}  \right\|_\op.
		\end{align*}
		For the first term, we find
		\begin{align*}
			&\sup_{u\in \bS^K} u^\T  H^{-1/2}\sum_{j \in  \oJ} {X_j \over (\wh A_{j\cdot}^\T \wh \alpha)   r_j^2}  
			\|\wh A_{j\cdot}\|_\i \|\wh \alpha-\alpha\|_1  
			\wh A_{j\cdot} \wh A_{j\cdot}^\T 
			H^{-1/2}  u\\
			&\le  16  \|\wh \alpha-\alpha\|_1  \max_{j\in \oJ} {\|\wh A_{j\cdot}\|_\i \over r_j} \sup_{u\in \bS^K} u^\T  H^{-1/2}\sum_{j \in  \oJ} {X_j \over   r_j^2}  
			\wh A_{j\cdot} \wh A_{j\cdot}^\T 
			H^{-1/2}  u\\
			&\le 24 \rho  \|\wh \alpha-\alpha\|_1  \left\|
			H^{-1/2}\wt H H^{-1/2}
			\right\|_\op && \text{by Lemma \ref{lem_prod_AT}}\\
			&\le 48  \rho  \|\wh \alpha-\alpha\|_1 && \text{by Lemma \ref{lem_Htilde}}.
		\end{align*}
		Regarding the second term, using similar arguments bounds it from above by 
		\begin{align*}
			&\left\| H^{-1/2}\sum_{j \in  \oJ} {X_j \over (\wh A_{j\cdot}^\T \wh \alpha) (\wh A_{j\cdot}^\T \alpha ) r_j^2} \left(\wh A_{j\cdot}^\T \alpha + A_{j\cdot}^\T \alpha\right)\left[ (\wh A_{j\cdot}- A_{j\cdot})^\T \alpha \right]
			\wh A_{j\cdot} \wh A_{j\cdot}^\T 
			H^{-1/2}  \right\|_\op\\
			&\le 12  \max_{k \in [K]}\sup_{u\in \bS^K} \sum_{j \in  \oJ} {X_j \over   r_j^2} {|\wh A_{jk}-A_{jk}| \over r_j}
			u^\T  H^{-1/2} \wh A_{j\cdot} \wh A_{j\cdot}^\T 
			H^{-1/2}  u\\
			&\le  12  \max_{j \in  \oJ} {\|\wh A_{j\cdot}\|_\i^2 \over r_j^2}   \max_{k \in [K]}\sup_{u\in \bS^K} \|H^{-1/2}u\|_1^2 \sum_{j \in  \oJ} {X_j \over   r_j} |\wh A_{jk}-A_{jk}|\\
			&\le {27 \rho^2 \over \kappa^2(A_{\oJ}, K)} \max_{k\in [K]}\sum_{j \in  \oJ} {X_j \over   r_j} |\wh A_{jk}-A_{jk}|
		\end{align*}
		In the last step  we used Lemmas \ref{lem_prod_AT} \& \ref{lem_H}.  Invoking (\ref{bd_X_r_diff_A}) and collecting terms obtains the desired result. 
	\end{proof}

	\subsubsection{Auxiliary lemmas used in the proof of \cref{thm_asn}}
	For completeness, we collect several lemmas that are used in the proof of \cref{thm_asn} and established in \cite{bing2021likelihood}. 
	Recall that 
	\[
	H = \sum_{j\in \oJ}{1\over r_j}A_{j\cdot}A_{j\cdot}^\T ,\qquad \wh H =  \sum_{j\in \oJ}{X_j  \over r_j^2}A_{j\cdot}A_{j\cdot}^\T.
	\]
	
	\begin{lemma}\label{lem_H}
		For any $u\in \RR^K$, 
		\[
		u^\T H u \ge \kappa^2(A_{\oJ}, K) \|u\|_1^2.
		\]
	\end{lemma}
	\begin{proof}
		This is proved in the proof of Theorem 2 \citep{bing2021likelihood}, see, display (F.8).
	\end{proof}

	\begin{lemma}[Lemma I.4 in \cite{bing2021likelihood}]\label{lem_H_deviation}
		For any $t\ge 0$, one has 
		\begin{equation*}
			\PP\left\{
			\left\|H^{-1/2}(\wh H - H)H^{-1/2}\right\|_{\rm op} \le  \sqrt{2B t\over N} + {B t \over 3 N}
			\right\}
			\ge 1 - 2K e^{-t/2},
		\end{equation*}
		with 
		$$
		B =   {\rho \left(1 + \xi\sqrt{K-s} \right) \over \kappa^{2}(A_{\oJ},K)\Tm}\le {1\vee \xi \over \kappa^{2}(A_{\oJ},K)\Tm^2}\left(1 + \xi\sqrt{K-s} \right).
		$$
		Moreover, if 
		\[
		N \ge C B \log K
		\]
		for some sufficiently large constant $C>0$, then, with probability $1-2K^{-1}$, one has 
		\[
		{1\over 2} \le \lambda_K(H^{-1/2}\wh H H^{-1/2}) \le  \lambda_1(H^{-1/2}\wh H H^{-1/2}) \le 2.
		\]
	\end{lemma}

	\begin{lemma}[Lemma I.3 in \cite{bing2021likelihood}]\label{lem_oracle_error_whitening}
		For any $t\ge 0$,  with probability $1-2e^{-t/2+ K\log 5}$, 
		\[
		\left\|
		\sum_{j\in  \oJ} {X_j - r_j \over r_j} H^{-1/2} A_{j\cdot}
		\right\|_2 \le 2\sqrt{t \over N} + {2\rho \over 3\kappa(A_{\oJ},K)} {t\over N}.
		\]
	\end{lemma}

	\begin{lemma}\label{lem_oracle_error_whitening_null}
		For any $t\ge 0$ and any $u\in \bS^K \cap \cK(F)$, with probability $1-2e^{-t/2}$, 
		\[
		\left|
		u^\T \sum_{j\in  \oJ} {X_j - r_j \over r_j} H^{-1/2} A_{j\cdot}
		\right| \le  {2\rho \over 3\kappa(A_{\oJ},K)} {t\over N}.
		\]
	\end{lemma}
	\begin{proof}
		The proof follow the same arguments of proving \cite[Lemma I.3]{bing2021likelihood} except by applying the Bernstein inequality  with the variance equal to
		\[
		u^\T H^{-1/2} A^\T D_{1/r} \left(
		D_{r} - rr^\T 
		\right)D_{1/r}  A H^{-1/2} u \overset{(\ref{eq_F})}{=} u^\T F u = 0.
		\] 
	\end{proof}
	
	\begin{lemma}\label{lem_Ahat}
		On the event that  (\ref{event_Ahat_col_prime}) holds,  we have
		\[
		\Bigl\|\sum_{j\in  \oJ}(X_j - r_j)\Bigl(
		{\wh A_{j\cdot} \over \wh A_{j\cdot}^\T \alpha} - {A_{j\cdot} \over A_{j\cdot}^\T \alpha} 
		\Bigr)\Bigr\|_\i \lesssim 
		\rho \|\wh A_{\oJ} - A_{\oJ}\|_{1,\i} + \rho
		\sum_{j\in \oJ\setminus \uJ} {\|\wh A_{j\cdot}-A_{j\cdot}\|_\i \over r_j}{\log(p)\over N}.
		\]
	\end{lemma}
	\begin{proof}
		This is proved in the proof of  \cite[Theorem 9]{bing2021likelihood}. 
	\end{proof}

	\subsubsection{Proof of \cref{thm_cdf}}\label{app_thm_cdf} 
	
	We prove \cref{thm_cdf} under the same notation as the main paper, and focus on the pair $(i,j) = (1,2)$ without loss  of generality. Recall $\cF'_{12}$ from \eqref{space_F_prime} and $ \Sigma^{(i)}$ from \eqref{def_Sigma_i}. Let
	\[
	Y:=Y_{12} = \sup_{f \in \cF'_{12}} f^\T Z_{12},
	\]
	where 
	\begin{equation}\label{asym_var}
		Z_{12} \sim \cN_K(0, \Sigma_{12}),\qquad \text{with}\quad  \Sigma_{12} = \Sigma^{(1)} + \Sigma^{(2)}.
	\end{equation}
	Similarly, with $\wh \cF'_\delta$ defined in \eqref{def_F_hat_prime_delta} and $\wh \Sigma^{(i)}$ given by \eqref{def_Sigma_hat_i}, let $\wh Y_b$ for $b\in [M]$ be i.i.d. realizations of 
	\[
	\wh Y   := \sup_{f \in \wh \cF'_\delta} f^\T \wh Z_{12}
	\]
	where 
	\begin{equation}\label{asym_var_est}
		\wh Z_{12} \sim \cN_K(0, \wh \Sigma_{12}),\qquad \text{with}\quad  \wh \Sigma_{12} = \wh\Sigma^{(1)} + \wh \Sigma^{(2)}.
	\end{equation} 
	Finally, recall that $\wh F_{N,M}(t)$ is the the empirical c.d.f. of $\{\wh Y_1,\ldots, \wh Y_M\}$ at any $t\in (0,1)$ while $F(t)$ is the c.d.f. of $Y$.

	\begin{proof}
		Denote by $\wh F_N$ the c.d.f. of $\wh Y$ conditioning on the observed data. Since 
		\[
		|\wh F_{N,M}(t) - F(t)| \le |\wh F_{N,M}(t)-  \wh F_{N}(t) | + |\wh F_{N}(t) - F(t)|,
		\]
		and the Glivenko-Cantelli theorem ensures that the first term on the right hand side converges to 0, almost surely, as $M\to\i$,  it remains to prove 
		$
		\wh F_{N}(t) - F(t) = o_\PP(1)
		$
		for any $t\in [0,1]$,
		or, equivalently, $\wh Y \overset{d}{\to} Y$. 
		
		To this end, notice that
		\begin{align*}
			\wh Y - Y &=  \underbrace{\sup_{f \in \wh \cF'_\delta} f^\T \wh Z_{12} -  \sup_{f \in \cF'_{12}} f^\T\wh Z_{12}}_{\rI} +  \underbrace{\sup_{f \in  \cF'_{12}} f^\T\wh Z_{12} -  \sup_{f \in  \cF'_{12}}  f^\T Z_{12}}_{\rII}.
		\end{align*}
		For the first term, we have 
		\begin{align}\label{eq_bd_term_I}\nonumber
			|\rI|  &\le d_H(\wh \cF'_\delta, \cF'_{12}) ~ \|\wh Z_{12}\|_1 &&\text{by \cref{lem_W_dist_Haus_dist}}\\ \nonumber
			&\le d_H(\wh \cF'_\delta, \cF'_{12}) ~ \sqrt{K}\|\wh Z_{12}\|_2\\\nonumber
			& = d_H(\wh \cF'_\delta, \cF'_{12})~  \cO_{\PP}\left(
			\sqrt{K\trace(\wh\Sigma_{12})}
			\right)\\\nonumber
			& = d_H(\wh \cF'_\delta, \cF'_{12})~  \cO_{\PP}\left(
			\sqrt{K\trace(\Sigma_{12})}
			\right) && \text{by part (b) of Lemma \ref{lem_Sigma_diff}}\\
			& = o_\PP(1) &&\text{by Lemmas \ref{lem_Sigma}   \&   \ref{lem_hausdorff_alternative} }.
		\end{align}
		Regarding $\rII$, observe that the function 
		\[
		h(u) = \sup_{f \in \cF'_{12}} f^\T u,\quad \forall~ u\in \RR^K,
		\]
		is Lipschitz, hence continuous. Indeed, for any $u,v\in \RR^K$, 
		\begin{equation}\label{h_lip}
			|h(u) - h(v)| \le \sup_{f \in \cF'_{12}} \|f\|_\i \|u-v\|_1\le  \sup_{f \in \cF} \|f\|_\i \|u-v\|_1\le \diam(\cA)\|u-v\|_1,
		\end{equation}
		and 
		\begin{equation}\label{bd_diam_cA}
			\diam(\cA) = \max_{k,k'}d(A_k,A_{k'}) \overset{\eqref{lip_d_upper}}{\le} C_d \max_{k,k'} {1\over 2}\|A_k - A_{k'}\|_1 \le  C_d.
		\end{equation} 
		By the continuous mapping theorem, it remains to prove $\wh Z_{12} \overset{d}{\to} Z_{12}$.  
		Since  
		$\wh Z_{12} \overset{d}{=} \wh\Sigma_{12}^{1/2}Z$ and $Z_{12} \overset{d}{=} \Sigma_{12}^{1/2}Z$ for $Z \sim \cN_K(0, \bI_K)$, and  
		\[
		\|\wh\Sigma_{12}^{1/2}Z - \Sigma_{12}^{1/2}Z\|_2^2 = \cO_{\PP}\left(
		\|\wh\Sigma_{12}^{1/2} - \Sigma_{12}^{1/2} \|_F^2 
		\right) = o_\PP(1)
		\]
		from  part (c) of Lemma \ref{lem_Sigma_diff},
		we conclude that 
		$\wh Z_{12} \overset{d}{\to} Z_{12}$
		hence 
		\[
		\sup_{f \in  \cF'_{12}} f^\T\wh Z_{12}   \overset{d}{\to}  \sup_{f \in  \cF'_{12}} f^\T Z_{12}. 
		\]  
		In conjunction with \eqref{eq_bd_term_I}, the proof is complete.
	\end{proof}

	\subsubsection{Lemmas used in the proof of \cref{thm_cdf}}
	
	The lemma below provides upper and lower bounds of the eigenvalues of $\Sigma^{(i)}$ defined in \eqref{def_Sigma_i}. Recall  $\ok_K$ and $\rho$ from \eqref{def_kappas} and \eqref{bd_rho}.
	
	\begin{lemma}\label{lem_Sigma}
		Provided that $\ok_K>0$, we have $\rank(\Sigma^{(i)}) = K-1$ and 
		\[
		{1\over K\rho}  \le  \lambda_{K-1}(\Sigma^{(i)}) \le \lambda_1(\Sigma^{(i)}) \le {1\over \ok_K^2}.
		\]
	\end{lemma}
	\begin{proof}
		Let us drop the superscript. Recall  from \eqref{def_H} that
		\[
		\Sigma = H^{-1/2} F H^{-1/2}.
		\]
		By \eqref{eigen_F}, we have 
		\[
		{1\over \lambda_1(H)}\le  \lambda_{K-1}(\Sigma) \le \lambda_1(\Sigma) \le {1\over \lambda_K(H)}.
		\]
		On the one hand, \cref{lem_H} implies 
		\[
		\lambda_K(H) \ge \kappa^2(A_{\oJ},K) \ge \ok_K^2.
		\]
		On the other hand, by recalling that $r = A\alpha$,
		\begin{equation}\label{bd_H_op}
			\lambda_1(H) \le \max_{k\in [K]} \sum_{j: r_j > 0} {A_{jk} \|A_{j\cdot}\|_1 \over r_j} \overset{\eqref{bd_rho}}{\le} K \rho \max_{k\in [K]}\sum_{j: r_j > 0} A_{jk} \le  K \rho. 
		\end{equation}
		The proof is complete.
	\end{proof}

	Recall $\wh \Sigma_{12}$ from \eqref{asym_var_est} and $\Sigma_{12}$ from \eqref{asym_var}. 
	The following lemma controls $\wh \Sigma_{12} -  \Sigma_{12}$ as well as $\wh \Sigma_{12}^{1/2} -  \Sigma_{12}^{1/2}$.

	\begin{lemma}\label{lem_Sigma_diff}
		Under conditions of Theorem \ref{thm_cdf},  we have 
		\begin{itemize}
			\item[(a)] $\|	\wh\Sigma_{12} - \Sigma_{12}\|_\op = o_\PP(1)$;
			\item[(b)] $\trace(\wh\Sigma_{12}) - \trace(\Sigma_{12}) = o_\PP(1)$;
			\item[(c)] $\|\wh\Sigma_{12}^{1/2} - \Sigma_{12}^{1/2}\|_F^2 = o_\PP(1)$.
		\end{itemize}
	\end{lemma}
	\begin{proof}
		We first prove the bound in the operator norm. By definition, 
		\[
		\|	\wh\Sigma_{12} - \Sigma_{12}\|_\op \le 	\|	\wh Q^{(1)} - Q_1^{(1)}\|_\op + 	\|	\wh Q^{(2)} - Q^{(2)}\|_\op.
		\]
		Let us focus on $\|\wh Q^{(1)} - Q_1^{(1)}\|_\op$ and drop the superscripts. 
		Starting with
		\begin{align*}
			\|\wh Q  - Q \|_\op &\le \|\wh V^{-1} - H^{-1}\|_\op + \| \wh \alpha\wh \alpha^{\T}-  \alpha  \alpha^{\T}\|_\op
		\end{align*}
		by recalling $\wh V$ and $H$ from \eqref{def_V_hat} and \eqref{def_H},
		the second term is bounded from above by 
		\begin{align*}
			\| \wh \alpha \wh \alpha^{\T}-  \alpha \alpha^{\T}\|_\op &= \sup_{u\in \bS^{K}} u^\T (\wh \alpha + \alpha ) (\wh \alpha^  - \alpha)^\T u\\
			& \le  \|\wh \alpha  + \alpha \|_2 \|\wh \alpha  - \alpha \|_2\\
			&\le   2 \|\wh \alpha  - \alpha\|_1\\
			&  = o_\PP(1).
		\end{align*}
		Furthermore, by invoking Lemma \ref{lem_Vhat}, \eqref{def_kappas} and \eqref{bd_H_op}, we have
		\begin{align*}
			\|\wh V^{-1} - H^{-1}\|_\op &= \|\wh V^{-1}(H-\wh V) H^{-1}\|_\op \\
			&\le \|H^{-1/2}\|_\op \| H^{1/2}\wh V^{-1}H^{1/2}\|_\op \|H^{-1/2}(H - \wh V) H^{-1/2}\|_\op   \|H^{-1/2}\|_\op\\
			& = \|H^{-1}\|_\op \cO_{\PP}\left( 
			{\rho^2 \over \ok_K^2}  \epsilon_{1,\i}  + \rho  \|\wh\alpha - \alpha\|_1 
			\right)\\
			& = o_\PP(1).
		\end{align*}
		Combining the last two displays and using the same arguments for $\|\wh Q^{(2)} - Q^{(2)}\|_\op$ finish the proof of part (a).
		
		The second result in part (b) follows immediately as 
		$\trace(\wh \Sigma_{12} - \Sigma_{12}) \le  K\|\wh \Sigma_{12}  -  \Sigma_{12}\|_\op$.
		
		To prove part (c),  write the singular value decomposition of $\wh \Sigma_{12}$ as
		\[
		\wh \Sigma_{12} = \wh U  \wh \Lambda \wh U ^\T
		\]
		where $\wh U=[\wh u_1,\ldots, \wh u_K]$ and $\wh\Lambda = \diag(\wh \lambda_1, \ldots, \wh \lambda_{\wh r}, 0, \ldots, 0)$ with $\wh \lambda_1 \ge \cdots \ge \wh \lambda_{\wh r}$ and  $\wh r = \rank(\wh \Sigma_{12})$. Similarly, write 
		\[
		\Sigma_{12} = U \Lambda U^\T
		\]
		with $ U=[ u_1,\ldots,  u_K]$, $\Lambda = \diag(\lambda_1, \ldots, \lambda_r,0,\ldots,0)$ and $r = \rank(\Sigma_{12})$.  We will verify in the end of the proof that 
		\begin{equation}\label{lb_lambda_r}
			\lambda_r \ge c,
		\end{equation}
		for some constant $c>0$. Since Weyl's inequality and  part (a) imply that
		\begin{equation}\label{singular_values}
			\max_{1\le j\le  r \vee \wh r}\left| \wh \lambda_j - \lambda_j \right| \le 	\|	\wh\Sigma_{12} - \Sigma_{12}\|_\op = o_\PP(1),
		\end{equation} 
		by further using \eqref{lb_lambda_r}, we can deduce that 
		$$
		\lim_{n,N\to \i}\PP(\wh r \ge  r)  =  1.
		$$
		From now on, let us work on the event $\{\wh r \ge r\}$. Notice that \eqref{lb_lambda_r} and \eqref{singular_values} also imply
		\begin{align}\label{sqrt_singular_values_1}
			&\max_{r<j\le \wh r} \wh \lambda_j^{1/2} \le \|\wh\Sigma_{12} - \Sigma_{12}\|_\op,\\\label{sqrt_singular_values_2}
			&\max_{1\le j\le r} \left| \wh \lambda_j^{1/2} - \lambda_j^{1/2} \right| = \max_{1\le j\le r}   {|\wh \lambda_j - \lambda_j| \over \wh \lambda_j^{1/2} + \lambda_j^{1/2}} \le  {\|\wh\Sigma_{12} - \Sigma_{12}\|_\op \over \sqrt c}.
		\end{align} 
		Furthermore, by the variant of Davis-Kahan theorem \cite{Yu15}, there exists some orthogonal matrix $Q\in \RR^{r\times r}$ such that
		\begin{equation}\label{singular_vectors}
			\|\wh U_r - U_r Q \|_F \le  {2^{3/2}\sqrt{r}\|\wh \Sigma_{12} - \Sigma_{12}\|_\op \over \lambda_r} \le {2^{3/2}\sqrt{r}\|\wh \Sigma_{12} - \Sigma_{12}\|_\op \over \sqrt c}.
		\end{equation}
		Here $\wh U_r$ ($U_r$) contains the first $r$ columns of $\wh U$ ($U$). 
		By writing $\wh \Lambda_r = \diag(\wh\lambda_1,\ldots,\wh \lambda_r)$ and $\Lambda_r = \diag(\lambda_1,\ldots,\lambda_r)$ and using the identity
		\begin{align*}
			\wh\Sigma_{12}^{1/2} - \Sigma_{12}^{1/2} &=  ~ (\wh U_r - U_rQ)\wh \Lambda_r^{1/2}\wh U_r^\T + U_rQ\wh \Lambda_r^{1/2}(\wh U_r- U_rQ)^\T\\
			&\quad  + U_rQ(\wh\Lambda_r^{1/2} - Q^\T \Lambda_r^{1/2} Q) Q^\T U_r^\T + \sum_{r<j\le \wh r} \wh\lambda_j^{1/2} \wh u_j \wh u_j^\T,
		\end{align*}
		we find 
		\begin{align*}
			&\|\wh\Sigma_{12}^{1/2} - \Sigma_{12}^{1/2} \|_F\\
			&\le 2 \|\wh U_r-U_rQ\|_F ~  \wh \lambda_1^{1/2}  + \|Q\wh\Lambda_r^{1/2} Q^\T - \Lambda_r^{1/2}\|_F + \sqrt{(\wh r -r)\wh\lambda_{r+1}}\\
			& = \cO\left(
			\|\wh \Sigma_{12} - \Sigma_{12}\|_\op^2 + \|\wh \Sigma_{12} - \Sigma_{12}\|_\op \lambda_1^{1/2} 
			\right) + \sqrt{r}~  \|\wh\Lambda_r^{1/2}  - Q^\T\Lambda_r^{1/2} Q\|_\op,
		\end{align*}
		where we use \eqref{sqrt_singular_values_1} -- \eqref{singular_vectors} in the last step.
		For the second term, we now prove the following by considering two cases,
		\[
		\| \wh\Lambda_r^{1/2} - Q^\T  \Lambda_r^{1/2} Q\|_\op  \lesssim   \max_{1\le j\le r} |\wh \lambda_j^{1/2}- \lambda_j^{1/2}| \overset{\eqref{singular_values}}{\le}  {\|\wh\Sigma_{12} - \Sigma_{12}\|_\op \over \sqrt c}.
		\]
		When $\Lambda_r$ has distinct values, that is, $\lambda_i - \lambda_{i+1} \ge c'$ for all $i\in \{1,\ldots, (r-1)\}$ with some constant $c'>0$,  then $Q$ is the identity matrix up to signs as $\wh u_i$ consistently estimates $u_i$, up to the sign, for all $i\in [r]$.  When $\Lambda_r$ has values of the same order,  we consider the case $\lambda_1 \asymp \cdots \asymp \lambda_s$ for some $s < r$ and 
		$\lambda_i - \lambda_{i+1} \ge c'$ for all $s \le i\le r-1$, and similar arguments can be used to prove the general case. Then since $\{u_{s+1}, \ldots, u_r\}$ can be consistently estimated, we must have, up to signs,
		$$
		Q  = \begin{bmatrix}
			Q_1 & 0 \\
			0 & \bI_{r-s}
		\end{bmatrix}
		$$ 
		with $Q_1^\T Q_1 = \bI_s$. Then by writing $\bar \lambda^{1/2} = s^{-1}\sum_{i=1}^s \lambda_i^{1/2}$, we have
		\begin{align*}
			\| \wh\Lambda_r^{1/2}  - Q^\T \Lambda_r^{1/2}Q\|_\op  \le  
			\max\left\{
			\max_{s < i \le r} |\wh \lambda_j^{1/2}- \lambda_j^{1/2}|,~ 
			\right\}
		\end{align*}
		and 
		\begin{align*}
			\| \wh\Lambda_s^{1/2}  - Q_1^\T \Lambda_s^{1/2}Q_1\|_\op &\le 
			\|\wh\Lambda_s^{1/2}   - \bar \lambda^{1/2} \bI_s   \|_\op +   \| \bar \lambda^{1/2} \bI_s - \Lambda_s^{1/2} \|_\op\\
			&\le \|\wh\Lambda_s^{1/2}   - \Lambda_s^{1/2}\|_\op + 2\| \bar \lambda^{1/2} \bI_s - \Lambda_s^{1/2} \|_\op\\
			& = \max_{1 \le i\le s} |\wh \lambda_j^{1/2}- \lambda_j^{1/2}| + o(1).
		\end{align*}
		Collecting terms concludes the claim in part (c). 
		
		Finally, we verify \eqref{lb_lambda_r} to complete the proof. By \cref{lem_Sigma}, we know that  $r \ge K-1$. Furthermore, when $r = K-1$, we must have $\alpha^{(1)} = \alpha^{(2)}$ whence 
		$
		\lambda_{K-1} \ge 2\lambda_{K-1}(\Sigma^{(i)})
		$ 
		is bounded from below, by using \cref{lem_Sigma} again. If $r =K$, then $\alpha^{(1)} \ne \alpha^{(2)}$, hence the null space of $\Sigma^{(i)}$ does not overlap with that of $\Sigma^{(j)}$. This together with \cref{lem_Sigma} implies \eqref{lb_lambda_r}.  
	\end{proof}
	
	\subsubsection{Consistency of $\wh \cF'_\delta$ in the Hausdorff distance}
	
	The following lemma proves the consistency of $\wh \cF'_\delta$ in the Hausdorff distance. For some sequences $\epsilon_A, \epsilon_{\alpha, A}>0$, define the event 
	\[
	\cE(\epsilon_A, \epsilon_{\alpha, A}):= \left\{\max_{k\in[K]} d(\wh  A_k, A_k) \le \epsilon_A\right\} \bigcap\left\{ \sum_{i=1}^2  \|\wh \alpha^{(i)} - \alpha^{(i)}\|_1 \le 2\epsilon_{\alpha, A} \right\}.
	\]

	\begin{lemma}\label{lem_hausdorff_alternative}
		On the event $\cE(\epsilon_A, \epsilon_{\alpha, A})$ for some  sequences $\epsilon_A,\epsilon_{\alpha, A}\to 0$,  by choosing  
		\[
		\delta \ge 6\epsilon_A +  2\diam(\cA) \epsilon_{\alpha, A},
		\]
		and $\delta \to 0$, 
		we have 
		$$
		\lim_{\delta\to 0}d_H(\wh\cF'_\delta, \cF'_{12}) = 0.
		$$ 
		In particular, under conditions of \cref{thm_cdf}, we have 
		$$
		d_H(\wh\cF'_\delta, \cF'_{12}) = o_\PP(1),\quad \text{as }n,N\to \i.
		$$
	\end{lemma}
	\begin{proof}
		For simplicity, we drop the subscripts and write 
		\[
		W = W(\balpha^{(1)}, \balpha^{(2)};d),\qquad \wh W = W(\wh \balpha^{(1)}, \wh \balpha^{(2)};d).
		\]
		By the proof of \cref{basic_ineq} together with \eqref{bd_diam_cA}, we have 
		\begin{equation}\label{bd_W_diff}
			|\wh W - W| \le 2\epsilon_A + \diam(\cA)\epsilon_{\alpha, A}.
		\end{equation}
		Recall that
		\begin{alignat*}{2}
			&\cF'  = \cF \cap G, \qquad &&  G := \left\{f\in \RR^K: f^\T (\alpha^{(1)} - \alpha^{(2)})  -  W = 0 \right\};\\
			&\wh\cF'_\delta = \wh\cF \cap \wh G_\delta, &&  \wh G_\delta :=\left\{f\in \RR^K: \left | f^\T (\wh \alpha^{(1)} - \wh \alpha^{(2)})  - \wh W\right| \le \delta\right\}.
		\end{alignat*} 
		By definition of the Hausdorff distance, we need to show that, as $\delta \to 0$, 
		\[
		\sup_{f \in \wh \cF'_\delta} f^\T v  \to  \sup_{f\in \cF'}f^\T v , \quad \text{uniformly over }  \cB_1 := \{v\in \RR^K: \|v\|_1 \le 1\}. 
		\]
		Since the functions $h_\delta: \cB_1 \to \RR$, defined as $h_\delta(v) = \sup_{f \in \wh \cF'_\delta} f^\T v$ for the sequence of $\delta$, are equicontinuous by the fact that $\wh \cF$ is compact (see, \eqref{h_lip}), it suffices to prove   
		$$\lim_{\delta \to 0} \sup_{f \in \wh \cF'_\delta} f^\T v  = \sup_{f\in \cF'}f^\T v,\qquad \text{for each  }v\in \cB_1.$$  To this end, fix any $v \in \cB_1$. 
		We first bound from above 
		\[
		g(v) :=  \sup_{f\in \cF'}f^\T v- \sup_{f \in \wh \cF'_\delta} f^\T v.
		\]
		Pick any $f\in \cF'$. By the construction of $\wh f$ as in the proof of Lemma \ref{lem_Haus_dist}, there exists $\wh f\in \wh \cF$ such that 
		\[
		\|\wh f - f\|_\i \le \max_{k, k'\in [K]}  \left|d(A_k, A_{k'}) - d(\wh A_k, \wh A_{k'})\right| \le 2\epsilon_A.
		\]
		Moreover,  by adding and subtracting terms, we find
		\begin{align}\label{bd_delta}\nonumber
			\left|
			\wh  f^\T (\wh \alpha^{(1)} - \wh \alpha^{(2)})  - \wh W
			\right| &\le 
			\left|
			(\wh f-f)^\T (\wh \alpha^{(1)} - \wh \alpha^{(2)})
			\right| + 	\left|
			f^\T (\wh \alpha^{(1)} -\alpha^{(1)}- \wh \alpha^{(2)} + \alpha^{(2)})
			\right| \\\nonumber
			&\qquad  + \left|
			f^\T (\alpha^{(1)}- \alpha^{(2)}) - \wh W
			\right| \\\nonumber
			&\le \|\wh f - f\|_\i \|\wh \alpha^{(1)} - \wh \alpha^{(2)}\|_1 + \|f\|_\i
			\|\wh \alpha^{(1)} -\alpha^{(1)}- \wh \alpha^{(2)} + \alpha^{(2)}\|_1
			\\\nonumber
			&\qquad + |\wh W - W|\\\nonumber
			&\le 6\epsilon_A +  2\diam(\cA) \epsilon_{\alpha, A}\\
			&\le \delta.
		\end{align}
		Therefore,  $\wh f \in \wh G_\delta$, hence $\wh f\in \wh\cF'_\delta$.
		For this choice of $\wh f$, we have
		\begin{align}\label{limsup_g}
			g(v) & \le  \sup_{f\in \cF'}\left( f^\T v  -\wh f~^\T v \right) \le \sup_{f\in \cF'}\|\wh f-f\|_\i \|v\|_1 \le \epsilon_A \to 0.
		\end{align}
		
		We  proceed to  bound from above 
		$-g(v)$. 
		Define 
		\begin{align*}
			&\cF_{\delta} = \{f\in \RR^K: f_1 = 0, ~ |f_k - f_\ell | \le d(A_k, A_\ell) + \delta, ~ \forall k,\ell\in [K]\},\\
			&G_\delta = \left\{f\in \RR^K: \left | f^\T (\alpha^{(1)} - \alpha^{(2)})  - W\right| \le2 \delta\right\}.
		\end{align*}
		Clearly, $\cF_{\delta}$ is compact, and both $\cF_{\delta}$ and $G_\delta$ are convex and monotonic in the sense that $\cF_{\delta}\subseteq \cF_{\delta'}$ and $G_\delta \subseteq G_{\delta'}$ for any $\delta \le \delta'$. Furthermore, on the event $\cE(\epsilon_A, \epsilon_{\alpha, A})$, we have  $\wh \cF \subseteq \cF_{\delta}$ and $\wh G_\delta \subseteq G_\delta$, implying that 
		\[
		-g(v) =   \sup_{f \in \wh \cF \cap \wh G_\delta } f^\T v - \sup_{f\in \cF'}f^\T v \le \sup_{f\in \cF_\delta \cap G_\delta}f^\T v - \sup_{f\in \cF'}f^\T v.
		\]
		Since $\cF_{\delta}\cap G_\delta$ is compact, we write $f_\delta$ for 
		\[
		f_\delta^\T v = \sup_{f\in  \cF_\delta \cap G_\delta}f^\T v.
		\]
		Observe that $f_\delta^\T v$ is a non-increasing sequence as $\delta \to 0$ and bounded from below by 0. Therefore, we can take its limit point which can be achieved by $f_0^\T v$ with $f_0 \in \cF' = \cF \cap G$. 
		%
		We thus conclude
		\begin{equation}\label{liminf_g}
			\limsup_{\delta \to 0} (-g(v)) \le \limsup_{\delta \to 0} f_\delta^\T v - \sup_{f \in \cF'}f^\T v\le 0.
		\end{equation}
		Combining \eqref{limsup_g} and \eqref{liminf_g} completes the proof. 
	\end{proof}

	\section{Minimax lower bounds of estimating any metric on discrete probability measures that is bi-Lipschitz equivalent to the Total Variation Distance}\label{app_general_lower_bound}
	
	Denote by $(\cX, d)$ an arbitrary finite metric space of cardinality $K$ with $d$ satisfying \eqref{d_sandwich} for some $0<c_*\le C_* < \i$. 
	In the  main text, we used the notation $D$ for this distance. We re-denote it here by $d$, as we found that this increased the clarity of exposition. 
	
	The following theorem provide minimax lower bounds of estimating $d(P,Q)$ for any $P, Q \in \cP(\cX)$ based on $N$ i.i.d. samples of $(P,Q)$.
	
	\begin{theorem}\label{thm:lower_bound}
		Let $(\cX, d)$ and $0< c_*\le C_* <\infty$ be as above.
		There exist positive universal constants $C, C'$ such that if $N \geq C K$, then
		\begin{equation*}
			\inf_{\wh d} \sup_{P, Q \in \cP(\cX)} \E  |\wh d - d(P,Q)| \geq C' \left({c_*^2\over C_*} \sqrt{\frac{K}{N (\log K)^2}} \vee \frac{C_*}{\sqrt N} \right)\,,
		\end{equation*}
		where the infimum is taken over all estimators constructed from $N$ i.i.d.\ samples from $P$ and $Q$, respectively.
	\end{theorem}
	
	
	

	\begin{proof} 
		In this proof, the symbols $C$ and $c$ denote universal positive constants whose value may change from line to line. For notational simplicity, we let 
		\[
		\kappa_{\cX} = {C_*\over c_*}
		\]
		Without loss of generality, we may assume by rescaling the metric that $C_* = 1$.
		The lower bound $1/\sqrt{N}$ follows directly from Le Cam's method \cite[see][Section 2.3]{Tsy09}. 
		We therefore focus on proving the lower bound $\kappa_{\cX}^{-2} \sqrt{K/(N (\log K)^2)}$, and may assume that $K/(\log K)^2 \geq C \kappa_{\cX}^4$.
		In proving the lower bound, we may also assume that $Q$ is fixed to be the uniform distribution over $\cX$, which we denote by $\rho$, and that the statistician obtains $N$ i.i.d.\ observations from the unknown distribution $P$ alone.
		We write
		\begin{equation*}
			R_N \defeq \inf_{\wh d} \sup_{P \in \Delta_K} \E_{P} |\wh d - d(P, \rho)|
		\end{equation*}
		for the corresponding minimax risk.
		
		We employ the method of ``fuzzy hypotheses" \cite{Tsy09}, and, following a standard reduction~\cite[see, e.g][]{Jiao2018,Wu2019}, we will derive a lower bound on the minimax risk by considering a modified observation model with Poisson observations.
		Concretely, in the original observation model, the empirical frequency count vector $Z \defeq N \wh P$ is a sufficient statistic, with distribution
		\begin{equation}
			Z \sim \mathrm{Multinomial}_K(N; P)\,.
		\end{equation}
		We define an alternate observation model under which $Z$ has independent entries, with $Z_k \sim \mathrm{Poisson}(N P_k)$, which we abbreviate as $Z \sim \mathrm{Poisson}(N P)$.
		We write $\E_P'$ for expectations with respect to this probability measure.
		Note that, in contrast to the multinomial model, the Poisson model is well defined for any $P \in \RR_{+}^K$.
		
		Write $\Delta_K' = \{P \in \RR_+^K: \|P\|_1 \geq 3/4\}$, and define
		\begin{equation*}
			\tilde R_N \defeq \inf_{\wh d} \sup_{P \in \Delta_K'} \E'_{P} \big|\wh d - d(P, \rho)\big|
		\end{equation*}
		
		
		The following lemma shows that $R_N$ may be controlled by $\tilde R_{2N}$.
		\begin{lemma}\label{mult_poi}
			For all $N \geq 1$,
			\begin{equation*}
				\tilde R_{2N} \leq R_N + \exp(-N/12)\,.
			\end{equation*}
		\end{lemma}
		\begin{proof}
			Fix $\delta > 0$, and for each $n \geq 1$ denote by $\wh d_n$ a near-minimax estimator for the multinomial sampling model:
			\begin{equation*}
				\sup_{P \in \Delta_K} \E_P |\wh d_n - d(P, \rho)| \leq R_n + \delta\,.
			\end{equation*}
			Since $Z$ is a sufficient statistic, we may assume without loss of generality that $\wh d_n$ is a function of the empirical counts $Z$.
			
			We define an estimator for the Poisson sampling model by setting $\wh d'(Z) = \wh d_{N'}(Z)$, where $N' \defeq \sum_{k=1}^K Z_k$.
			If $Z \sim \mathrm{Poisson}(2N P)$ for $P \in \RR_{+}^K$, then, conditioned on $N' = n'$, the random variable $Z$ has distribution $\mathrm{Multinomial}_K(n'; P/\|P\|_1)$.
			For any $P \in \Delta_K(\epsilon)$, we then have
			\begin{align*}
				\E_P' |\wh d' - d(P/\|P\|_1, \rho)| & = \sum_{n' \geq 0} \E_P'\Big[|\wh d' - d(P/\|P\|_1, \rho)\Big| N' = n'\Big] \mathbb P_P[N' = n'] \\
				& = \sum_{n' \geq 0} \E_P'\Big[|\wh d_{n'} - d(P/\|P\|_1, \rho)|\Big| N' = n'\Big] \mathbb P_P[N' = n'] \\
				& = \sum_{n' \geq 0}\E_{P/\|P\|_1}\Big[|\wh d_{n'} - d(P/\|P\|_1, \rho)|\Big] \mathbb P_P[N' = n'] \\
				& \leq \Big(\sum_{n' \geq 0} R_{n'} \mathbb P_P[N' = n']\Big) + \delta\,.
			\end{align*}
			Since $N' \sim \mathrm{Poisson}(2N \|P\|_1)$, and $R_{n'}$ is a non-increasing function of $n'$ and satisfies $R_{n'} \leq 1$, standard tail bounds for the Poisson distribution show that if $\|P\|_1 \geq 3/4$, then
			\begin{equation*}
				\E_P' |\wh d' - d(P/\|P\|_1, \rho)| \leq \mathbb P_{P}[N' < N] + \sum_{n' \geq N} R_{n'} \mathbb P_{P}[N' = n'] + \delta \leq e^{-N/12} + R_N + \delta\,.
			\end{equation*}
			Since $P \in \Delta_K'$ and $\delta > 0$ were arbitrary, taking the supremum over $P$ and infimum over all estimators $\wh d_N$ yields the claim.
		\end{proof}
		
		It therefore suffices to prove a lower bound on $\tilde R_N$.
		Fix an $\epsilon \in (0, 1/4)$, $\varkappa \geq 1$ and a positive integer $L$ to be specified.
		We employ the following proposition.
		\begin{proposition}\label{moment_matching}
			There exists a universal positive constant $c_0$ such that for any $\varkappa \geq 1$ and positive integer $L$, there exists a pair of mean-zero random variables $X$ and $Y$ on $[-1, 1]$ satisfying the following properties:
			\begin{itemize}
				\item $\E X^\ell = \E Y^\ell$ for $\ell = 1, \dots, 2L - 2$
				\item $\E |X| \geq \varkappa \E |Y| \geq c_0 L^{-1} \varkappa^{-1}$
			\end{itemize}
		\end{proposition}
		\begin{proof}
			A proof appears in \cref{app_proof_moment_matching}
		\end{proof}
		We define two priors $\mu_0$ and $\mu_1$ on $\RR_+^K$ by letting
		\begin{align*}
			\mu_1 & = \mathrm{Law}\left(\frac 1 K + \frac \epsilon K X_1, \dots, \frac 1 K + \frac \epsilon K X_K\right) \\
			\mu_0 & = \mathrm{Law}\left(\frac 1 K + \frac \epsilon K Y_1, \dots, \frac 1 K + \frac \epsilon K Y_K\right)\,,
		\end{align*}
		where $X_1, \dots, X_K$ are i.i.d.\ copies of $X$ and $Y_1, \dots, Y_K$ are i.i.d.\ copies of $Y$.
		Since $\epsilon < 1/4$ and $X_k, Y_k \geq -1$ almost surely, $\mu_1$ and $\mu_0$ are supported on $\Delta_K'$.
		
		The following lemma shows that $\mu_0$ and $\mu_1$ are well separated with respect to the values of the functional $P \mapsto d(P/\|P\|_1, \rho)$.
		\begin{lemma}
			Assume that $\varkappa \geq 7 \kappa_{\cX}$. Then there exists $r \in \RR_+$. such that
			\begin{align*}
				\mu_0(P: d(P/\|P\|_1, \rho) \leq r) \geq 1 -  \delta\\
				\mu_1(P: d(P/\|P\|_1, \rho) \geq r + 2s) \geq 1 - \delta
			\end{align*}
			where $s = \frac 1 2 \epsilon c_0 L^{-1} \varkappa^{-2}$ and $\delta = 2e^{- K c_0^2/L^2 \varkappa^4}$
		\end{lemma}
		\begin{proof}
			By \cref{d_sandwich}, we have for any two probability measures $\nu, \nu' \in \Delta_K$,
			\begin{equation}
				\kappa_\cX^{-1} \norm{\nu - \nu'}_1 \leq 2 d(\nu, \nu') \leq \norm{\nu - \nu'}_1\,.
			\end{equation}
			
			Note that
			\begin{equation}
				\int \|P - \rho\|_1 \, \mu_0(d P) = \E \sum_{k=1}^K \left| \frac 1 K + \frac \epsilon K Y_k - \frac 1 K \right| = \epsilon \E |Y|\,,
			\end{equation}
			and by Hoeffding's inequality,
			\begin{equation}
				\mu_0(P: \|P - \rho\|_1 \geq \epsilon \E |Y| + t) \leq e^{- K t^2/2 \epsilon^2}\,.
			\end{equation}
			Analogously,
			\begin{equation}
				\mu_1(P: \|P - \rho\|_1 \leq \epsilon \E |X| - t) \leq e^{- K t^2/2 \epsilon^2}\,.
			\end{equation}
			Under either distribution, Hoeffding's inequality also yields that $\big|\|P\|_1 - 1\big| \geq t$ with probability at most $e^{- K t^2/2 \epsilon^2}$.
			Take $t = \epsilon c_0 L^{-1} \varkappa^{-2}$.
			Then letting $\delta = 2e^{- K c_0^2/L^2 \varkappa^4}$, we have with $\mu_0$ probability at least $1-\delta$ that
			\begin{align}
				2d(P/\|P\|_1, \rho) & \leq \left\|\frac{P}{\|P\|_1} - \rho\right\|_1 \nonumber\\
				& \leq \left\|P - \frac{P}{\|P\|_1}\right\|_1 +  \|P - \rho\|_1 \nonumber\\
				& = \big|\|P\|_1 - 1 \big| + \|P - \rho\|_1 \nonumber \\
				& \leq \epsilon \E |Y| + 2\epsilon c_0 L^{-1} \varkappa^{-2}\,.\label{eq:mu0_est}
			\end{align}
			And, analogously, with $\mu_1$ probability at least $1 - \delta$,
			\begin{align}
				2d(P/\|P\|_1, \rho) & \geq \kappa_{\cX}^{-1}\left\|\frac{P}{\|P\|_1} - \rho\right\|_1 \nonumber \\
				& \geq \kappa_{\cX}^{-1}\|P - \rho\|_1 - \left\|P - \frac{P}{\|P\|_1}\right\|_1 \nonumber\\
				& \geq \epsilon \kappa_{\cX}^{-1}\E |X| - 2\epsilon c_0 L^{-1} \varkappa^{-1} \nonumber \\
				& \geq \epsilon \varkappa \kappa_{\cX}^{-1} \E |Y| - 2\epsilon c_0 L^{-1} \varkappa^{-2}\,, \label{eq:mu1_est}
			\end{align}
			where we have used that $\kappa_{\cX} \geq 1$ and $\E |X| \geq \varkappa \E |Y|$ by construction.
			Therefore, as long as $\varkappa \geq 7 \kappa_{\cX}$, we may take $r = \frac 12 \epsilon \E |Y| + \epsilon c_0 L^{-1} \varkappa^{-2}$, in which case
			\begin{align*}
				r + 2 s & = \frac 12 \epsilon \E |Y| + 2\epsilon c_0 L^{-1} \varkappa^{-2} \\
				& \leq \frac 12 \epsilon \varkappa \kappa_{\cX}^{-1} \E |Y| - 3 \epsilon \E |Y| + 2\epsilon c_0 L^{-1} \varkappa^{-2} \\
				& \leq \frac 12 \epsilon \varkappa \kappa_{\cX}^{-1} \E |Y| - \epsilon c_0 L^{-1} \varkappa^{-2}\,,
			\end{align*}
			where the last inequality uses that $\E |Y| \geq c_0 L^{-1} \varkappa^{-2}$.
			With this choice of $r$ and $s$, we have by~\cref{eq:mu0_est} and~\cref{eq:mu1_est} that
			\begin{align*}
				\mu_0(P: d(P/\|P\|_1, \rho) \leq r) \geq 1 -  \delta\\
				\mu_1(P: d(P/\|P\|_1, \rho) \geq r + 2s) \geq 1 - \delta
			\end{align*}
			as claimed.
		\end{proof}
		
		Following \cite{Tsy09}, we then define ``posterior" measures
		\begin{equation*}
			\mathbb P_j = \int \mathbb P_P \mu_j(d P), \quad \quad j = 0, 1\,,
		\end{equation*}
		where $\mathbb P_P = \mathrm{Poisson}(N P)$ is the distribution of the Poisson observations with parameter $P$. 
		We next bound the statistical distance between $\mathbb P_0$ and $\mathbb P_1$.
		\begin{lemma}
			If $\epsilon^2 \leq (L+1)K/(4 e^2 N)$, then
			\begin{equation*}
				\tv{\mathbb P_0}{\mathbb P_1} \leq K2^{-L}\,.
			\end{equation*}
		\end{lemma}
		\begin{proof}
			By construction of the priors $\mu_0$ and $\mu_1$, we may write
			\begin{equation*}
				\mathbb P_j = (\mathbb Q_j)^{\otimes K}, \quad j = 0, 1\,,
			\end{equation*}
			where $Q_0$ and $Q_1$ are the laws of random variables $U_0$ and $U_1$ defined by
			\begin{align*}
				U_0 \mid \lambda & \sim \mathrm{Poisson}(\lambda_0)\,, \quad \quad
				\lambda_0 \overset{d}{=} \frac N K + \frac{N \epsilon}{K} Y \\
				U_1 \mid \lambda & \sim \mathrm{Poisson}(\lambda_1)\,, \quad \quad
				\lambda_1 \overset{d}{=} \frac N K + \frac{N \epsilon}{K} X\,.
			\end{align*}
			By~\cite[Lemma 32]{Jiao2018}, if $L + 1 \geq (2 e N \epsilon/K)^2/(N/K) = 4 e^2 \epsilon^2 N/K$, then
			\begin{equation*}
				\tv{\mathbb Q_0}{\mathbb Q_1} \leq 2 \left(\frac{e N \epsilon}{K\sqrt{N(L+1)/K}}\right)^{L+1} \leq 2^{-L}\,.
			\end{equation*}
			Therefore, under this same condition,
			\begin{equation*}
				\tv{\mathbb P_0}{\mathbb P_1} = \tv{\mathbb Q_0^{\otimes K}}{\mathbb Q_1^{\otimes K}}\leq K\tv{\mathbb Q_0}{\mathbb Q_1} \leq K 2^{-L}\,.
			\end{equation*}
		\end{proof}
		
		Combining the above two lemmas with \cite[Theorem 2.15(i)]{Tsy09}, we obtain that as long as $\varkappa \geq 7 \kappa_{\cX}$ and $\epsilon^2 \leq (L+1)K/(4 e^2 N)$, we have
		\begin{equation*}
			\inf_{\wh d} \sup_{P \in \Delta_K'} \E_{P}' \1\left\{|\wh d - d(P/\|P\|_1, \rho)| \geq\frac 1 2 \epsilon c_0 L^{-1} \varkappa^{-2}\right\} \geq \frac 12 (1 - K 2^{-L}) - 2e^{- K c_0^2/L^2 \varkappa^4}\,.
		\end{equation*}
		Let $\varkappa = 7 \kappa_{\cX}$, $L = \lceil \log_2 K \rceil + 1$, and set $\epsilon = c \sqrt{K/N}$ for a sufficiently small positive constant $c$.
		By assumption, $K/(\log K)^2 \geq C \varkappa^4$, and by choosing $C$ to be arbitrarily small we may make $2e^{- K c_0^2/L^2 \varkappa^4}$ arbitrarily small, so that the right side of the above inequality is strictly positive.
		We therefore obtain
		\begin{equation*}
			\tilde R_N \geq C \varkappa^{-2} \sqrt{\frac{K}{N (\log K)^2}}\,,
		\end{equation*}
		and applying \cref{mult_poi} yields
		\begin{equation*}
			R_N \geq C \varkappa^{-2} \sqrt{\frac{K}{N (\log K)^2}} - \exp(-N/12)\,.
		\end{equation*}
		Since we have assumed that the first term is at least $N^{-1/2}$, the second term is negligible as long as $N$ is larger than a universal constant.
		This proves the claim.
	\end{proof}
	
	%

	\subsection{Proof of \cref{moment_matching}}\label{app_proof_moment_matching}
	\begin{proof}
		First, we note that it suffices to construct a pair of random variables $X'$ and $Y'$ on $[0, (c_1 L \varkappa)^2]$ satisfying $\E (X')^{k} = \E (Y')^{k}$ for $k = 1, \dots, L - 1$ and $\E \sqrt{X'} \geq \varkappa \E \sqrt{Y'} \geq c_2$ for some positive constants $c_1$ and $c_2$.
		Indeed, letting $X = (c_1 L \varkappa)^{-1}\ep \sqrt{X'}$ and $Y = (c_1 L \varkappa)^{-1} \ep \sqrt{Y'}$ where $\ep$ is a Rademacher random variable independent of $X'$ and $Y'$ yields a pair $(X, Y)$ with the desired properties.
		We therefore focus on constructing such a $X'$ and $Y'$.
		
		By \cite[Lemma 7]{Wu2019} combined with \cite[Section 2.11.1]{Timan1994}, there exists a universal positive constant $c_1$ such we may construct two probability distributions $\mu_+$ and $\mu_-$ on $[8 \varkappa^2, (c_1\varkappa L)^2]$ satisfying
		\begin{align}
			\int x^k \dd \mu_+(x) & = \int x^k \dd \mu_-(x)  \quad k = 1, \dots, L \\
			\int x^{-1} \dd \mu_+(x) & = \int x^{-1} \dd \mu_-(x) + \frac {1}{16\varkappa^2} \,.
		\end{align}
		We define a new pair of measures $\nu_+$ and $\nu_-$ by
		%
		\begin{align}
			\nu_+(\rd x) & = \frac{1}{Z}\left(\frac{1}{x(x - 1)} \mu_+(\rd x) + \alpha_0 \delta_0(\rd x)\right) \\
			\nu_-(\rd x) & = \frac{1}{Z}\left(\frac{1}{x(x - 1)} \mu_-(\rd x) + \alpha_1 \delta_1(\rd x)\right)\,,
		\end{align}
		where we let
		\begin{align}
			Z & = \int \frac{1}{x-1} \dd\mu_+(x) - \int \frac{1}{x} \dd\mu_-(x) \\
			\alpha_0 & = \int \frac{1}{x} \dd\mu_+(x) - \int \frac{1}{x} \dd\mu_-(x) = \frac{1}{32 \varkappa^2}\\
			\alpha_1 & = \int \frac{1}{x-1} \dd\mu_+(x) - \int \frac{1}{x-1} \dd\mu_-(x)\,.
		\end{align}
		Since the support of $\mu_+$ and $\mu_-$ lies in $[8 \varkappa^2, (c_1 \varkappa L)^2]$, these quantities are well defined, and \cref{const_bounds} shows that they are all positive.
		Finally, the definition of $Z$, $\alpha_0$, and $\alpha_1$ guarantees that $\nu_+$ and $\nu_-$ have total mass one.
		Therefore, $\nu_+$ and $\nu_-$ are probability measures on $[0, (c_1 \varkappa L)^2]$.
		
		We now claim that
		\begin{equation}
			\int x^{k} \nu_+(\rd x) = \int x^{k} \nu_-(\rd x) \quad k = 1, \dots, L\,.
		\end{equation}
		By definition of $\nu_+$ and $\nu_-$, this claim is equivalent to
		\begin{equation}
			\int \frac{x^{\ell}}{x-1} \mu_+(\rd x) = \int \frac{x^{\ell}}{x - 1} \mu_-(\rd x) + \alpha_1 \quad k = 0, \dots, L-1\,,
		\end{equation}
		or, by using the definition of $\alpha_1$,
		\begin{equation}
			\int \frac{x^{\ell} - 1}{x-1} \mu_+(\rd x) = \int \frac{x^{\ell} - 1}{x-1} \mu_-(\rd x) \quad k = 0, \dots, L-1\,.
		\end{equation}
		But this equality holds due to the fact that $\frac{x^{\ell} - 1}{x-1}$ is a degree-$(\ell-1)$ polynomial in $x$, and the first $L$ moments of $\mu_+$ and $\mu_-$ agree.
		
		Finally, since $x \geq 1$ on the support of $\nu_-$, we have
		\begin{equation}
			\int \sqrt{x} \dd \nu_-(x) \geq 1\,.
		\end{equation}
		We also have
		\begin{equation}
			\int \sqrt{x} \dd \nu_+(x)  = \frac{1}{Z} \int \frac{1}{\sqrt{x}(x - 1)} \dd \mu_+(x)\,,
		\end{equation}
		and by \cref{sqrt_bounds}, this quantity is between $1/(8 \varkappa)$ and $1/\varkappa$.

		Letting $Y' \sim \nu_+$ and $X' \sim \nu_-$, we have verified that $X', Y' \in [0, (c_1 L \varkappa)^2]$ almost surely and $\E (X')^k = \E (Y')^k$ for $k = 1, \dots, L$.
		Moreover, $\E \sqrt{X'} \geq 1 \geq \varkappa \E \sqrt{Y'} \geq 1/8$, and this establishes the claim.
	\end{proof}
	\subsection{Technical lemmata}
	\begin{lemma}\label{const_bounds}
		If $\varkappa \geq 1$, then the quantities $Z, \alpha_0, \alpha_1$ are all positive, and $Z \in [\varkappa^{-2}/16, \varkappa^{-2}/8]$.
	\end{lemma}
	\begin{proof}
		First, $\alpha_0 = \frac{1}{16 \kappa^2}$ by definition, and $Z \geq \alpha_0$ since $(x-1)^{-1} \geq x^{-1}$ on the support of $\mu_+$.
		Moreover, for all $x \geq 8\varkappa^2 \geq 8$,
		\begin{equation}
			\left|\frac{1}{x - 1} - \frac{1}{x}\right| = \frac{1}{x(x-1)} \leq \varkappa^{-2}/50\,.
		\end{equation}
		Therefore $\alpha_1 \geq \alpha_0 - \varkappa^{-2}/50 > 0$ and $Z \leq \alpha_0 + \varkappa^{-2}/50 \leq \varkappa^{-2}/8$, as claimed.
	\end{proof}
	
	\begin{lemma}\label{sqrt_bounds}
		If $\varkappa \geq 1$, then
		\begin{equation}
			\frac{1}{Z} \int \frac{1}{\sqrt{x}(x - 1)} \dd \mu_+(x) \in \left[\frac{1}{8 \varkappa}, \frac{1}{ \varkappa}\right]
		\end{equation}
	\end{lemma}
	\begin{proof}
		For $x \geq 8 \varkappa^2 \geq 8$,
		\begin{equation}
			\frac 1 Z \left|\frac{1}{\sqrt{x}(x - 1)} - \frac{1}{x^{3/2}}\right| = \frac{1}{Zx^{3/2}(x-1)} \leq \frac{16 \varkappa^2}{7(8 \varkappa^2)^{3/2}} =  \frac{1}{7 \sqrt 2 \varkappa}\,.
		\end{equation}
		Since
		\begin{equation}
			\frac{1}{Z} \int \frac{1}{x^{3/2}} \dd \mu_+(x) \leq \frac{1}{Z} (8 \varkappa^2)^{-3/2} \leq \frac{1}{\sqrt 2 \varkappa}\,,
		\end{equation}
		we obtain that 
		\begin{equation}
			\frac{1}{Z} \int \frac{1}{\sqrt{x}(x - 1)} \dd \mu_+(x) \leq \frac{1}{7 \sqrt 2 \varkappa} + \frac{1}{\sqrt 2 \varkappa} = \varkappa^{-1}\,.
		\end{equation}
		
		The lower bound bound follows from an application of Jensen's inequality:
		\begin{equation}
			\frac{1}{Z} \int \frac{1}{x^{3/2}} \dd \mu_+(x) \geq \frac{1}{Z}\left(\int x^{-1} \dd \mu_+(x)\right)^{3/2} \geq \frac{1}{Z} (16 \varkappa^2)^{-3/2} \geq\frac{1}{8 \varkappa}\,,
		\end{equation}
		where we have used the fact that
		\begin{equation}
			\int x^{-1} \dd \mu_+(x) = \int x^{-1} \dd \mu_+(x) + \frac{1}{16 \varkappa^2} \geq (16 \varkappa^2)^{-1}\,.
		\end{equation}
	\end{proof}

	\section{Lower bounds for estimation of mixing measures in topic models}\label{sec_upper_bound_mm_est}
	In this section, we substantiate \cref{rem:mm_lb} by exhibiting a lower bound without logarithmic factors for the estimation of a mixing measure $\balpha$ in Wasserstein distance.
	Combined with the upper bounds of \cite{bing2021likelihood}, this implies that $\wh \balpha$ is nearly minimax optimal for this problem.
	While a similar lower bound could be deduced directly from \cref{thm_lower_bound_SWD}, the proof of this result is substantially simpler and slightly sharper.
	We prove the following.

	\begin{theorem}\label{thm_lower_bound_mm}
		Grant topic model assumptions and assume $1< \tau \le cN$ and $pK \le c'(nN)$ for some universal constants $c,c'>0$. Then, for any $d$ satisfying \eqref{lip_d} with $c_d>0$, we have 
		\[
		\inf_{\wh\balpha} \sup_{\alpha \in \Theta_\alpha(\tau), A\in \Theta_A} \EE [W(\wh\balpha, \balpha; d)] ~ \gtrsim ~  c_d  ~ \ok_\tau\left(
		\sqrt{\tau  \over N} +  {1\over \ok_\tau}  \sqrt{pK \over nN}\right).
		\]
		Here the infimum is taken over all estimators. 
	\end{theorem}
	
	
	Our proof of \cref{thm_lower_bound_mm} builds on a reduction scheme that reduces the problem of proving minimax lower bounds for $W(\wh\balpha, \balpha; d)$ to that for  columns of $\wh A - A$ in $\ell_1$-norm and for $\|\wh\alpha - \alpha\|_1$ in which we can invoke  \eqref{disp_lower_bound_A} as well as the results in \cite{bing2021likelihood}.
	
	\begin{proof}[Proof of \cref{thm_lower_bound_mm}]
		Fix $1<\tau \le K$. Define the space of $\balpha$ as 
		\[
		\Theta_{\balpha} := \left\{
		\sum_{k=1}^K \alpha_k \delta_{A_k}: \alpha \in \Theta_\alpha(\tau), A\in\Theta_A
		\right\}.
		\]
		Similar to the proof of \cref{thm_lower_bound_SWD}, for any subset $\bar \Theta_{\balpha} \subseteq  \Theta_{\balpha}$, we have the following reduction as the first step of the proof,
		\begin{align*}
			\inf_{\wh\balpha} \sup_{\alpha \in \Delta_K, A\in \Theta_A} \EE_{\balpha}\left[W(\wh\balpha, \balpha; d)\right] 
			&\ge {1\over 2}\inf_{\wh\balpha\in \bar\Theta_{\balpha} } \sup_{\balpha \in \bar\Theta_{\balpha} } \EE_{\balpha}\left[W(\wh\balpha, \balpha; d)\right].
		\end{align*}

		We proceed to prove the two terms in the lower bound separately. 
		To prove the first term, let us fix one $A\in \Theta_A$ as well as $S = \{1,2,\ldots, \tau\}$ and choose
		\[
		\bar \Theta_{\balpha} = \left\{
		\sum_{k=1}^K \alpha_k \delta_{A_k}: \supp(\alpha) = S
		\right\}.
		\]
		It then follows that 
		\begin{align*}
			\inf_{\wh\balpha} \sup_{\alpha \in \Theta_\alpha(\tau), A\in \Theta_A} \EE\left[W(\wh\balpha, \balpha; d)\right] & \ge {1\over 2}\inf_{\wh\balpha\in \bar\Theta_{\balpha} } \sup_{\balpha \in \bar\Theta_{\balpha} } \EE\left[W(\wh\balpha, \balpha; d)\right]\\
			&\ge {1\over 4}\min_{k\ne k'\in S}d(A_k, A_{k'}) ~ \inf_{\substack{\wh\alpha\in \Theta_\alpha(\tau)\\\supp(\wh\alpha) = S}} \sup_{\substack{\alpha \in \Theta_\alpha(\tau)\\\supp(\alpha) = S}} \EE\left[ \|\wh\alpha - \alpha\|_1 \right]  \\
			&\ge {c_d\over 2}\ok_\tau ~  \inf_{\wh \alpha} \sup_{\substack{\alpha \in \Theta_\alpha(\tau)\\\supp(\alpha) = S}}  \EE\left[ \|\wh\alpha - \alpha\|_1 \right] .
		\end{align*}
		where in the last line we used the fact that 
		\begin{equation}\label{lb_diam_A}
			\min_{k\ne k'\in S}d(A_k, A_{k'}) \ge  c_d \min_{k\ne k'\in S} \|A_k-A_{k'}\|_1 \ge 2c_d ~ \ok_\tau
		\end{equation}
		from \eqref{lip_d} and the definition in \eqref{def_kappa_A}. Since the lower bounds in \cite[Theorem 7]{bing2021likelihood} also hold for the parameter space $\{\alpha \in \Theta_\alpha(\tau): \supp(\alpha) = S\}$, the first term is proved.
		
		To prove the second term, fix $\alpha = \sum_{k=1}^s \be_k / s$ and choose
		\[
		\bar \Theta_{\balpha} = \left\{
		{1\over s}\sum_{k=1}^s \delta_{A_1}: A \in \Theta_A
		\right\}.
		\]
		We find that
		\begin{align*}
			\inf_{\wh\balpha} \sup_{\alpha \in \Delta_K, A\in \Theta_A} \EE\left[W(\wh\balpha, \balpha; d)\right] & \ge {1\over 2}\inf_{\wh\balpha\in \bar\Theta_{\balpha} } \sup_{\balpha \in \bar\Theta_{\balpha} } \EE\left[W(\wh\balpha, \balpha; d)\right]\\
			&\ge {1\over 2} \inf_{\wh A\in \Theta_A} \sup_{A\in \Theta_A} \EE\left[ {1\over \tau}\sum_{k=1}^s d(\wh A_k, A_k) \right]  \\
			&\ge {c_d\over 2} \inf_{\wh A} \sup_{A\in \Theta_A} \EE\left[{1\over \tau}\sum_{k=1}^s \|\wh A_k-A_k\|_1 \right]. 
		\end{align*}
		By the proof of \cite[Theorem 6]{bing2020fast} together with  $m + K(p-m)\ge (1-c)pK$ under $m\le cp$ for $c<1$, we obtain the second lower bound, hence completing the proof.
	\end{proof}

	\section{Limiting distribution of our W-distance estimator for two samples with different sample sizes}\label{app_proof_thm_limit_distr_nm}
	
	We generalize the results in \cref{thm_limit_distr} to cases where $X^{(i)}$ and $X^{(j)}$ are the empirical frequencies based on $N_i$ and $N_j$, respectively, i.i.d. samples of $r^{(i)}$ and $r^{(j)}$.  Write 
	$N_{\min} = \min\{N_i, N_j\}$ and $N_{\max} = \max\{N_i, N_j\}$. 
	Let 
	$Z_{ij}' \sim \cN_K(0, Q^{(ij)'})$ where 
	\[
	Q^{(ij)'} = 
	\lim_{N_{\min}  \to \infty}
	{N_j \over N_i + N_j}\Sigma^{(i)}  +  {N_i \over N_i + N_j}\Sigma^{(j)}.
	\]
	Note that we still assume $N \asymp N_\ell$ for all $\ell \in [n]$.

	\begin{theorem}\label{thm_limit_distr_diff_sizes}
		Grant conditions in \cref{thm_limit_distr}.
		For any $d$ satisfying \eqref{lip_d_upper} with $C_d = \cO(1)$,  we have the following convergence in distribution, as $n, N_{\min} \to \i$,
		\begin{equation*}
			\sqrt{N_iN_j \over N_i + N_j}\left(
			\wt W - W(\balpha^{(i)}, \balpha^{(j)}; d)
			\right) \overset{d}{\to} \sup_{f\in \cF'_{ij}}f^\top Z_{rs}'
		\end{equation*}
		with  $\cF'_{ij}$ defined in \eqref{space_F_prime}.
	\end{theorem}
	\begin{proof}
		From \cref{thm_asn}, by recognizing that
		\[
		\sqrt{N_iN_j \over N_i + N_j}\left((\wt \alpha^{(i)}- \wt \alpha^{(j)}) - (\alpha^{(i)} - \alpha^{(j)})\right) \overset{d}{\to}  Z_{rs}',\quad \textrm{as }n,N_{\min}\to \infty,
		\]
		the proof follows from the same arguments in the proof of \cref{thm_limit_distr}.
	\end{proof}

	\section{Simulation results}\label{app_sim}
	
	In this section we provide numerical supporting evidence to our theoretical findings. 
	In \cref{sec_sim_CI} we evaluate the coverage and lengths of fully data-driven confidence intervals for the Wasserstein distance between mixture distributions, abbreviated in what follows as the W-distance.  Speed of convergence in distribution of our distance estimator $\wt W$ with its dependence on $K$, $p$ and $N$ is given in    \cref{sec_sim_LT}.  \\

	Before presenting the simulation results regarding the distance estimators, we begin by illustrating that the proposed estimator of the mixture weights, the debiased MLE estimator  \eqref{def_alpha_td},    has the desired behavior.  We verify its asymptotic normality 
	in Appendix  \ref{sec_sim_ASN}.
	For completeness, we also  compare its behavior to that of a weighted least squares estimator,  in \ref{app_sim_LS}. The impact of using different estimators for inference on the W-distance is discussed in  \cref{sec_sim_MLE_LS}. \\

	We use the following data generating mechanism throughout. The mixture weights are generated uniformly from $\Delta_K$ (equivalent to the symmetric Dirichlet distribution with parameter equal to 1) in the dense setting. In the sparse setting, for a given cardinality $\tau$, we randomly select the  support  as well as non-zero entries from Unif$(0,1)$ and normalize in the end to the unit sum. 
	For the mixture components in $A$, we first generate its entries as i.i.d. samples from Unif$(0,1)$ and then normalize each column to the unit sum. Samples of $r = A\alpha$ are generated according to the multinomial sampling scheme in \eqref{multinomial}. We take the distance $d$  in $W(\balpha^{(i)}, \balpha^{(j)}; d)$ to be the total variation distance, that is, $d(A_k, A_\ell) =  \|A_k - A_\ell\|_1/2$ for each $k,\ell \in [K]$. To simplify presentation, we consider  $A$ known throughout the simulations.

	\subsection{Asymptotic normality of the proposed estimator of the mixture weights}\label{sec_sim_ASN}
	
	We verify the asymptotic normality of our proposed debiased MLE (MLE\_debias) given by \eqref{def_alpha_td} in both the sparse setting, and the dense setting, of the mixture weights, and compare its speed of convergence to a normal limit with its counterpart restricted on $\Delta_K$:   the maximum likelihood estimator (MLE) in \eqref{MLE}. 
	
	Since $A$ is known, it suffices to consider one mixture distribution $r = A\alpha$,  where $\alpha \in \Delta_K$ is generated according to either the sparse setting with $\tau = | \supp(\alpha)| =  3$ or the dense setting. 
	We fix $K = 5$ and $p = 1000$ and vary $N \in \{50, 100, 300, 500\}$. For each setting, we obtain estimates of $\alpha$ for each estimator based on 500 repeatedly generated multinomial samples. In the sparse setting, the first panel of Figure \ref{fig_ASN_T} depicts the QQ-plots (quantiles of the estimates after centering and standardizing versus quantiles of the standard normal distribution) of all two estimators for $\alpha_3 = 0$.  It is clear that MLE\_debias converges to  normal limits whereas MLE does not,  corroborating our discussion in \cref{sec_ASN}. Moreover, in the dense setting, the second panel of Figure \ref{fig_ASN_T} shows that,  for $\alpha_3\approx 0.18$, although the MLE converges to normal limits eventually, its speed of convergence is outperformed by the MLE\_debias. We conclude from both settings that the MLE\_debias always converges to a normal limit faster than the MLE regardless of the sparsity pattern of the mixture weights. 
	
	\begin{figure}[ht]
		\centering
		\begin{subfigure}[b]{\textwidth}
			\includegraphics[width=\textwidth]{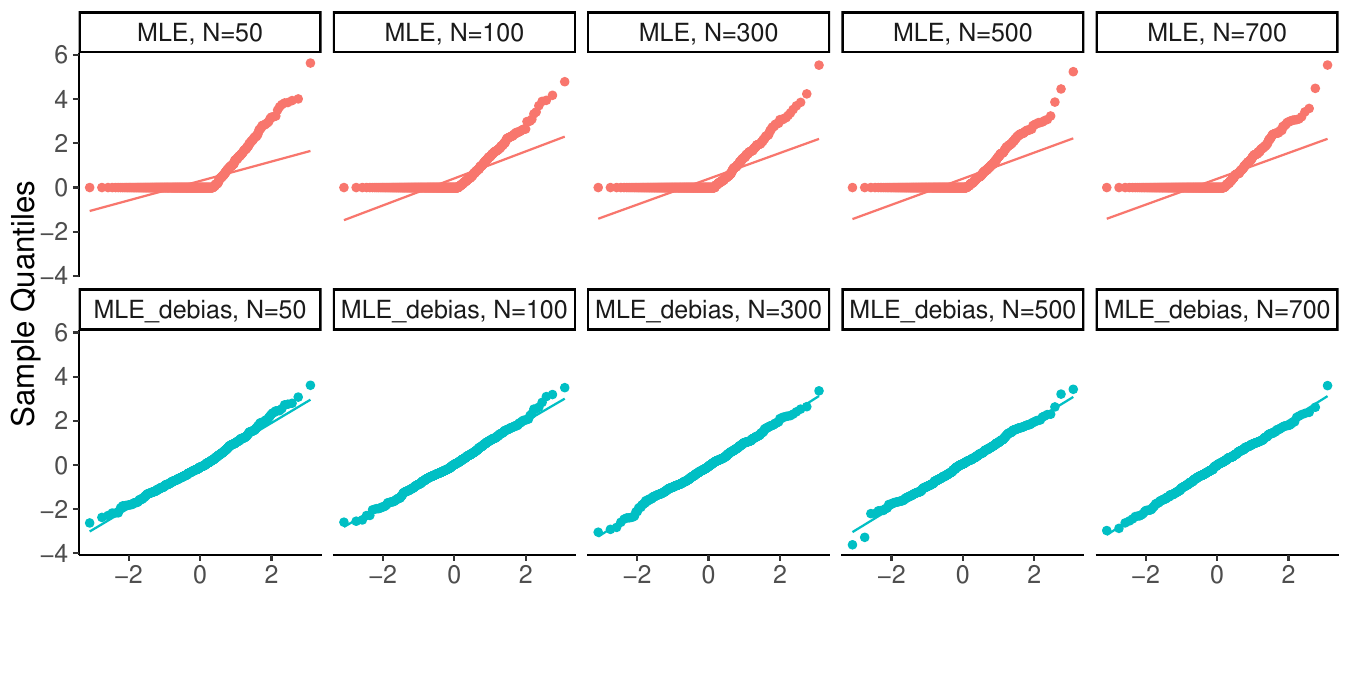}
			\vspace{-1.2cm}
		\end{subfigure}
		\hfill
		\begin{subfigure}[b]{\textwidth}
			\includegraphics[width=\textwidth]{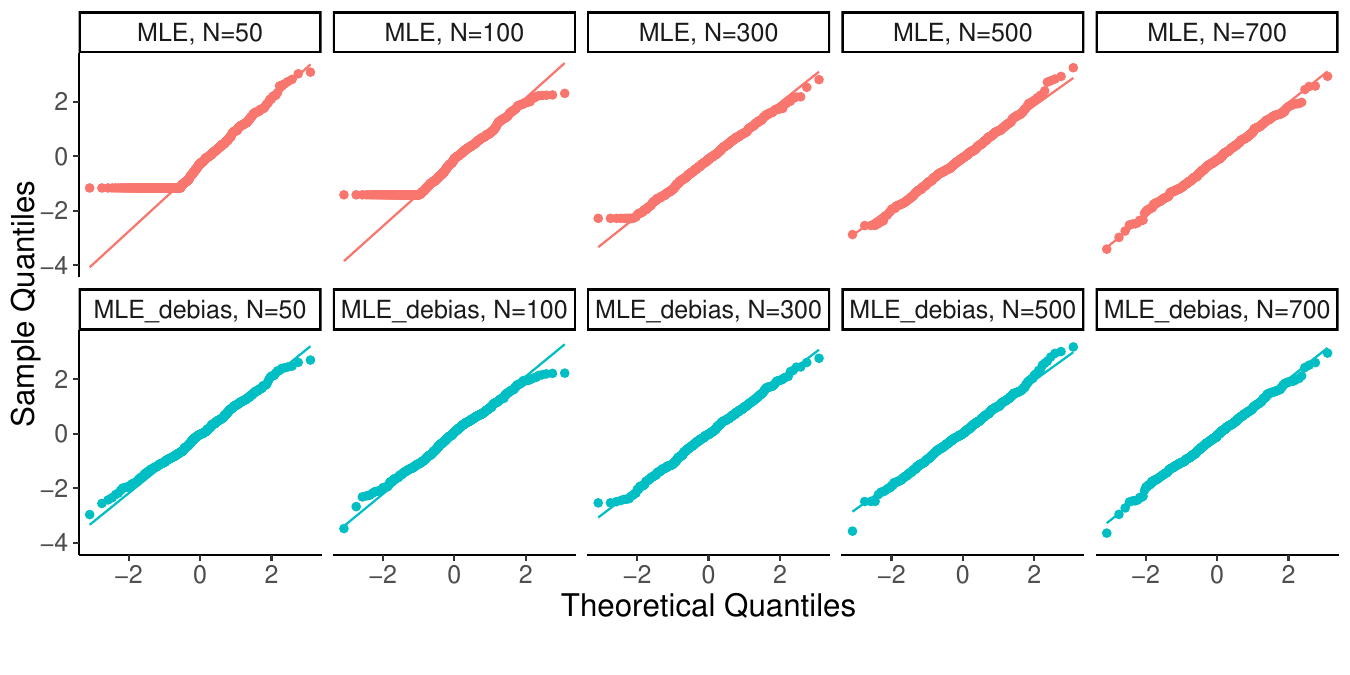}
		\end{subfigure}
		\vspace{-1cm}
		\caption{\small QQ-plots for  the MLE and the MLE\_debias in the sparse (upper) and dense (lower) settings}
		\label{fig_ASN_T}
	\end{figure}
	
	In \cref{app_sim_LS}, we also compare the weighted least squares estimator (WLS) in \cref{rem_LS} with its counterpart restricted to $\Delta_K$, and find that WLS converges to normal limits faster than its restricted counterpart.

	\subsection{Asymptotic normality of the weighted least squares estimator of the mixture weights}\label{app_sim_LS}
	
	We verify the speed of convergence to normal limits of the weighted least squares estimator (WLS) in \cref{rem_LS}, and compare with its counterpart restricted on $\Delta_K$ (WRLS) given by \cref{def_WRLS}. We follow the same settings used in \cref{sec_sim_ASN}. As shown in \cref{fig_ASN_T_LS}, WLS converges to normal limits at  much faster speed than WRLS in both dense and sparse settings.

	\begin{figure}[ht]
		\centering
		\begin{subfigure}[b]{\textwidth}
			\includegraphics[width=\textwidth]{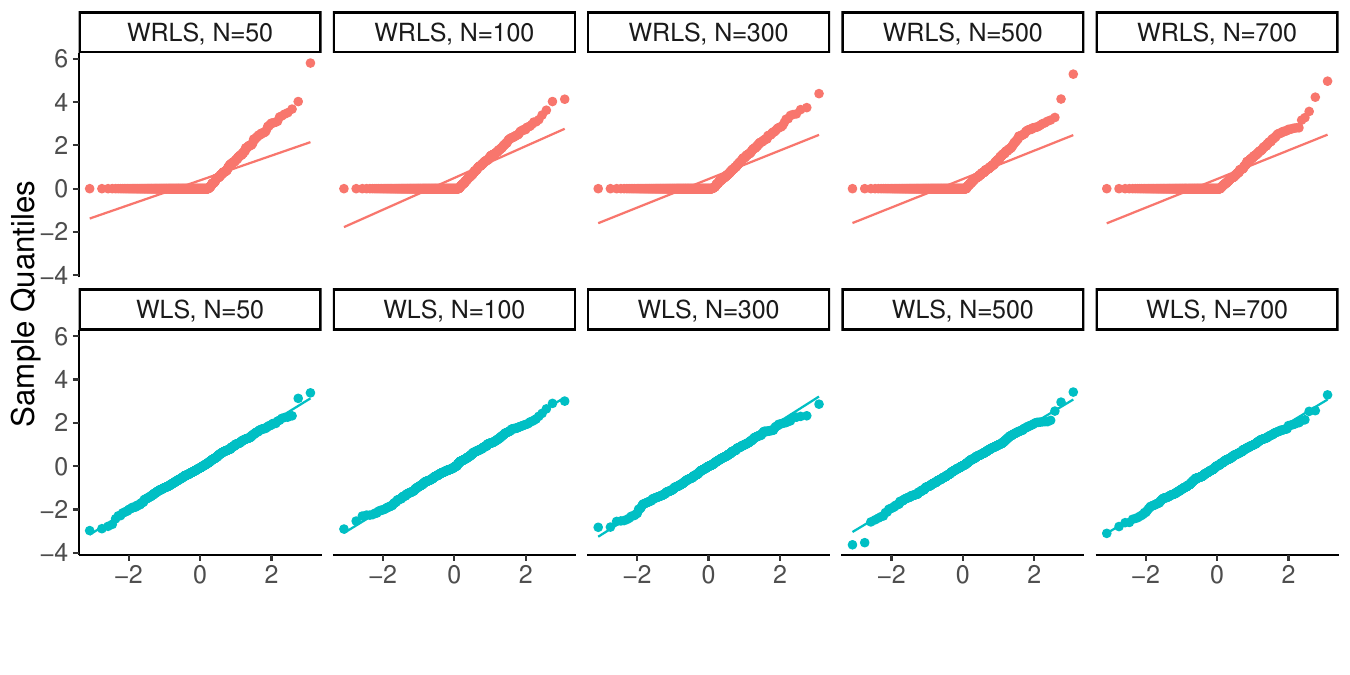}
			\vspace{-1.2cm}
		\end{subfigure}
		\hfill
		\begin{subfigure}[b]{\textwidth}
			\includegraphics[width=\textwidth]{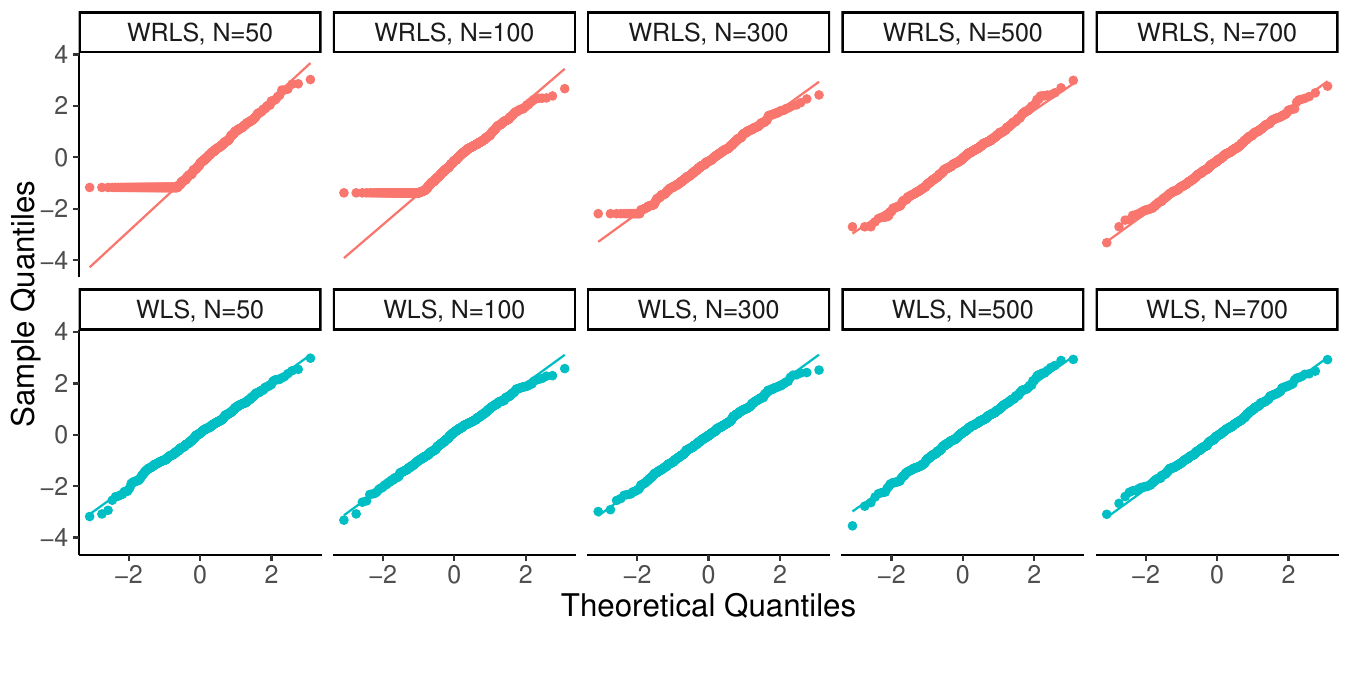}
		\end{subfigure}
		\vspace{-1cm}
		\caption{\small QQ-plots for  the WRLS and the WLS in the sparse (upper) and dense (lower) settings}
		\label{fig_ASN_T_LS}
	\end{figure}

	\subsection{ Impact of weight estimation on distance estimation }\label{sec_sim_MLE_LS}
	In this section, we construct confidence intervals for $W(\balpha^{(i)}, \balpha^{(j)})$, based on two 
	possible distance estimators: 
	(i) $\wt W$, the plug-in distance estimator, with weights estimated by the debiased MLE (\ref{def_alpha_td}) and (\ref{def_V_hat}) above; (ii) $\wt W_{LS}$, the plug-in distance estimator with weights estimated by the weighted least squares estimator given in 
	\cref{rem_LS}. 
	
	We  evaluate the coverage and length of 95\% confidence intervals (CIs) based on their respective limiting distributions, given  in \eqref{limit_distr_MLE} and \eqref{limit_distr_LS}. We illustrate the performance when $\alpha^{(i)} = \alpha^{(j)}$.

	We set $K = 5$, $p = 500$ and vary $N \in \{50, 100, 300, 500\}$. For each $N$, we generate 10 mixtures $r^{(\ell)} = A\alpha^{(\ell)}$ for $\ell\in\{1,\ldots, 10\}$, with each $\alpha^{(\ell)}$ generated according to the dense setting.  Then for each $r^{(\ell)}$, we repeatedly draw 200 pairs of $X^{(i)}$ and $X^{(j)}$, sampled from $\text{Multinomial}_p(N; r^{(\ell)})$ independently, and compute the coverage and averaged length of the 95\% CIs for each procedure. For a fair comparison, we assume the two limiting distributions in \eqref{limit_distr_MLE} and \eqref{limit_distr_LS} known, and approximate their quantiles by their empirical counterparts based on 10,000 realizations. By further averaging the coverages and averaged lengths of 95\% CIs over $\ell \in \{1,\ldots, 10\}$, we see in \cref{tab_MLE_LS} that although both procedures have nominal coverage, CIs based on the debiased MLE are uniformly narrower than those based on the WLS, in line with our discussion in \cref{rem_LS}. The same experiment is repeated for the sparse setting with $\tau = 3$ from which we draw the same conclusion except that the advantage of using the debiased MLE becomes more visible. Furthermore, the CIs based on the debiased MLE get slightly narrower in the sparse setting whereas those based on the WLS remain the same. 
	
	\begin{table}[H]
		\centering
		\caption{The averaged coverage and length of 95\% CIs.}
	\label{tab_MLE_LS}
	\renewcommand{\arraystretch}{1.25}{
		\resizebox{\textwidth}{!}{
			\begin{tabular}{c l ccccc|ccccc}\hline 
				&  & & \multicolumn{4}{c|}{Coverage} & \multicolumn{4}{c}{The averaged length of 95\% CIs} \\
				&  &   $N=25$ & $N=50 $ & $N=100$  & $N=300$  &  $N=500$ & $N=25$ & $N=50 $ & $N=100$  & $N=300$  &  $N=500$\\\hline 
				\multirow{2}{*}{Dense} &   $\wt W$ & 0.950 &  0.948 &  0.961 &  0.942 &  0.950 &  0.390 & 0.276  & 0.195 &  0.113 &  0.087 \\ 
				&   $\wt W_{LS}$ & 0.933 & 0.948 &  0.955 &  0.945 &  0.948 & 0.403 & 0.285 &  0.202 &  0.116 &  0.090\\
				\hline
				\multirow{2}{*}{Sparse} &   $\wt W$ & 0.961 &  0.952 &  0.964 &   0.944  &  0.952  & 0.372 &   0.263 &  0.186 &  0.107  & 0.083  \\ 
				&   $\wt W_{LS}$ & 0.945  &  0.955   & 0.966  & 0.953  &  0.959  & 0.402  & 0.285 & 0.201  & 0.116  & 0.090\\
				\hline
			\end{tabular} 
	}}
\end{table}

\subsection{Confidence intervals of the W-distance}\label{sec_sim_CI}
In this section, we evaluate the procedures introduced  in \cref{sec_est_LT} for constructing confidence intervals for $W(\balpha^{(i)}, \balpha^{(j)}; d)$ for both $\alpha^{(i)} = \alpha^{(j)}$ and $\alpha^{(i)} \ne \alpha^{(j)}$. Specifically, we evaluate both the coverage and the length of 95\% confidence intervals (CIs) of the true W-distance based on: the $m$-out-of-$N$ bootstrap ($m$-$N$-BS), the derivative-based bootstrap (Deriv-BS) and the plug-in procedure (Plug-in), detailed in  \cref{sec_est_LT_plugin}. For the $m$-out-of-$N$ bootstrap, we follow \cite{sommerfeld2018inference} by setting $m = N^\gamma$ for some $\gamma \in (0, 1)$. Suggested by our simulation in \cref{sec_sim_BS}, we choose $\gamma = 0.3$. For the two procedures based on bootstrap, we set the number of repeated bootstrap samples as 1000. For fair comparison, we also set $M=1000$ for the plug-in procedure (cf. \cref{samples}).

Throughout this section, we set $K = 5$, $p=500$ and vary  $N\in \{100, 500, 1000, 3000\}$.  The simulations of Section~\ref{sec_sim_ASN} above showed that our de-biased estimator of the mixture weights should be preferred even when the MLE is available (the dense setting).  The construction of the data driven  CIs for the distance only depends on the mixture weight estimators and their asymptotic normality, and not the nature of the true mixture weights.
In the experiments presented below, the mixture weights are generated, without loss of generality,   in the dense setting. We observed the same behavior in the sparse setting. 

%
%
%
\subsubsection{Confidence intervals of the W-distance for $\alpha^{(i)} = \alpha^{(j)}$}

We generate the mixture $r  = A\alpha$, and for each choice of $N$, we generate $200$ pairs of $X^{(i)}$ and $X^{(j)}$, sampled from Multinomial$_p(N, r)$ independently. We record in \cref{tab_CIs_null} the coverage and averaged length of 95\% CIs of the W-distance over these 200 repetitions
for each procedure. We can see that Plug-in and Deriv-BS have comparable performance with the former being more computationally expensive. On the other hand, $m$-$N$-BS only achieves the nominal coverage when $N$ is very large. As $N$ increases, we observe narrower CIs for all methods.

\begin{table}[H]
\centering
\caption{The averaged coverage and length of 95\% CIs. The numbers in parentheses are the standard errors ($\times 10^{-2}$)}
\label{tab_CIs_null}
\renewcommand{\arraystretch}{1.25}{
	\resizebox{\textwidth}{!}{
		\begin{tabular}{c cccc|cccc}\hline
			& \multicolumn{4}{c|}{Coverage} & \multicolumn{4}{c}{The averaged length of 95\% CIs} \\
			$N$ &  $100$ & $500$ &  $1000$ &  $3000$ &  $100$ & $500$ &  $1000$ &  $3000$\\\hline 
			$m$-$N$-BS & 0.865 & 0.905 & 0.935 & 0.945 & 0.153 (0.53) & 0.079 (0.23) &   0.059 (0.18)  & 0.036 (0.10)\\
			Deriv-BS & 0.94 & 0.93 & 0.945 & 0.95 & 0.194 (0.77) &  0.090 (0.31) &  0.064 (0.21) &  0.037 (0.13)\\
			Plug-in & 0.96 & 0.93 & 0.94 & 0.945 & 0.195 (0.68) & 0.089 (0.30) &  0.063 (0.21) &  0.037 (0.11)\\\hline
		\end{tabular} 
}}
\end{table}


\subsubsection{Confidence intervals of the W-distance for $\alpha^{(i)} \ne \alpha^{(j)}$}

In this section we further compare the three procedures in the previous section for $\alpha^{(i)} \ne \alpha^{(j)}$. Holding the mixture components $A$ fixed, we generate $10$ pairs of $(\alpha^{(i)},\alpha^{(j)})$ and their corresponding $(r^{(i)}, r^{(j)})$. For each pair, as in previous section, we generate 200 pairs of $(X^{(i)}, X^{(j)})$ based on $(r^{(i)}, r^{(j)})$ for each $N$, and record the coverage and averaged length of 95\% CIs for each procedure. By further averaging the coverages and averaged lengths over the 10 pairs of $(r^{(i)}, r^{(j)})$, \cref{tab_CIs_alternative} contains the final coverage and averaged length of 95\% CIs for each procedure. To save computation, we reduce both $M$ and the number of repeated bootstrap samples to $500$. As we can see in \cref{tab_CIs_alternative}, Plug-in has the nominal coverage in all settings. Deriv-BS has similar performance as Plug-in for $N\ge 500$ but tends to have under-coverage for $N = 100$. On the other hand, $m$-$N$-BS is outperformed by both Plug-in and Deriv-BS.

\begin{table}[H]
\centering
\caption{The averaged coverage and length of 95\% CIs. The numbers in parentheses are the standard errors ($\times 10^{-2}$)}
\label{tab_CIs_alternative}
\renewcommand{\arraystretch}{1.25}{
	\resizebox{\textwidth}{!}{
		\begin{tabular}{c cccc|cccc}\hline
			& \multicolumn{4}{c|}{Coverage} & \multicolumn{4}{c}{The averaged length of 95\% CIs} \\
			$N$ &  $100$ & $500$ &  $1000$ &  $3000$ &  $100$ & $500$ &  $1000$ &  $3000$\\\hline 
			$m$-$N$-BS & 0.700 & 0.666  & 0.674  & 0.705 & 0.156 (0.67) & 0.082 (0.33) &   0.062 (0.25)  & 0.039 (0.16)\\
			Deriv-BS & 0.912 & 0.943 &  0.952 & 0.953 & 0.299 (2.71) &  0.140 (0.95) &  0.100 (0.60) &  0.058 (0.31)\\
			Plug-in & 0.946  &  0.950   &  0.952  &   0.951 & 0.304 (2.43) & 0.140  (0.90) &  0.100  (0.57) &  0.058 (0.30)\\\hline
		\end{tabular} 
}}
\end{table}


\subsection{Speed of convergence in distribution of the proposed distance estimator}\label{sec_sim_LT} 

We focus on the null case, $\alpha^{(i)} = \alpha^{(j)}$, to show how the speed of convergence of our estimator $\wt W$ in \cref{thm_limit_distr} depends on $N$, $p$ and $K$.  Specifically, we consider (i) $K = 10$, $p = 300$ and $N\in \{10, 100, 1000\}$, (ii) $N = 100$, $p = 300$ and $K \in \{5, 10, 20\}$ and (iii) $N = 100$, $K = 10$ and $p\in \{50, 100, 300, 500\}$. For each combination of $K, p$ and $N$, we generate  10 mixture distributions $r^{(\ell)} = A\alpha^{(\ell)}$ for $\ell\in\{1,\ldots, 10\}$, with each $\alpha^{(\ell)}$ generated according to the dense setting. Then for each $r^{(\ell)}$, we repeatedly draw 10,000 pairs of $X^{(i)}$ and $X^{(j)}$, sampled from $\text{Multinomial}_p(N; r^{(\ell)})$ independently, and compute our estimator $\wt W$. To evaluate the speed of convergence of $\wt W$, we
compute the Kolmogorov-Smirnov (KS) distance, as well as the p-value of the two-sample KS test, between our these $10,000$ estimates of $\sqrt{N}~\wt W$ and its limiting distribution. Because the latter is not given in closed form, we mimic the theoretical c.d.f. by a c.d.f. based on $10,000$ draws from it. Finally, both the averaged KS distances and p-values,  over $\ell\in \{1,2,\ldots, 10\}$, are shown in Table \ref{tab_speed_LT}. We can see that the speed of convergence of our distance estimator gets faster as $N$ increases or $K$ decreases, but seems  unaffected by the ambient dimension $p$ (we note here that  the sampling variations across different settings of varying $p$ are larger than those of varying $N$ and $K$). In addition, our distance estimator already seems to converge in distribution to the established limit as soon as $N \ge K$. 

\begin{table}[ht]
\caption{The averaged KS distances and p-values of the two sample  KS test between our estimates and samples of the limiting distribution}
\label{tab_speed_LT}
\centering
\renewcommand{\arraystretch}{1.25}{
	\resizebox{\textwidth}{!}{
		\begin{tabular}{l  ccc | cc c | cccc }\hline
			& \multicolumn{3}{c|}{$K = 10$, $p = 300$} & \multicolumn{3}{c|}{$N = 100$, $p = 300$} & \multicolumn{4}{c}{$N = 100$, $K = 10$}\\
			&   $N = 10$  & $N = 100$ & $N = 1000$ & $K = 5$ & $K = 10$ & $K = 20$ & $p = 50$ & $p = 100$ & $p = 300$ & $p = 500$\\\cline{2-4} \cline{5-7} \cline{8-11}  
			KS distances $(\times 10^{-2})$ & 1.02 & 0.94 & 0.87 &  0.77 & 0.92 & 0.97 & 0.95 & 0.87 & 1.01 & 0.92 \\ 
			P-values  of KS test & 0.34 & 0.39 & 0.50 & 0.58 & 0.46 & 0.36  & 0.37 & 0.47 & 0.29 & 0.48\\\hline
		\end{tabular}
	}
}
\end{table}

\subsection{Details of $m$-out-of-$N$ bootstrap and the derivative-based bootstrap}\label{app_BS_procedure}

In this section we give details of the $m$-out-of-$N$ bootstrap and the derivative-based bootstrap mentioned in \cref{sec_BS}. For simplicity, we focus on $N_i = N_j = N$. Let $B$ be the total number of bootstrap repetitions and $\wh A$ be a given estimator of $A$.  

For a given integer $m<N$ and any $b\in [B]$, let $X_{*b}^{(i)}$ and $X_{*b}^{(j)*}$ be the bootstrap samples obtained as 
\[
m X_{*b}^{(\ell)} \sim \text{Multinomial}_p(m; X^{(\ell)}),\quad \ell \in \{i,j\}.
\]
Let $\wt W_b$ be defined in \eqref{SWD_bar} by using $\wh A$ as well as $\wt \alpha_b^{(i)}$ and $\wt \alpha_b^{(j)}$. The latter are obtained from \eqref{def_alpha_td} based on $X_{*b}^{(i)}$ and $X_{*b}^{(j)*}$. The $m$-out-of-$N$ bootstrap estimates the distribution of $\sqrt{N}(\wt W - W(\balpha^{(i)}, \balpha^{(j)}; d))$ by the empirical distribution of $\sqrt{m}( \wt W_b -\wt W)$ with $b\in [B]$.

The derivative-based bootstrap method on the other hand aims to estimate  the limiting distribution in \cref{thm_limit_distr}. Specifically, let $\wh \cF'_\delta$ be some pre-defined estimator of $\cF_{ij}'$. The derivative-based bootstrap method estimates the distribution of 
\[
\sup_{f\in \cF'_{ij}}f^\T Z_{ij}
\]
by the empirical distribution of 
\[
\sup_{f\in \wh \cF'_\delta}f^\T  B_{ij}^{(b)},\quad \text{with } b\in [B].
\]
Here
\[
B_{ij}^{(b)} = \sqrt{N}\Bigl(\wt\alpha_b^{(i)} - \wt\alpha_b^{(j)} - \wt \alpha^{(i)} + \wt \alpha^{(j)}\Bigr)
\]
uses the estimators $\wt\alpha_b^{(i)}$ and $\wt\alpha_b^{(j)}$ obtained from \eqref{def_alpha_td} based on the bootstrap samples 
\[
N X_{*b}^{(\ell)} \sim \text{Multinomial}_p(N; X^{(\ell)}),\quad \ell \in \{i,j\}.
\]

\subsection{Selection of $m$ for $m$-out-of-$N$ bootstrap}\label{sec_sim_BS}
In our simulation, we follow \cite{sommerfeld2018inference} by setting $m = N^\gamma$ for some $\gamma \in (0, 1)$. 
To evaluate its performance, consider the null case, $r = s$, and choose $K = 10$, $p = 500$, $N \in \{100, 500, 1000, 3000\}$ and $\gamma \in \{0.1, 0.3, 0.5, 0.7\}$. The number of repetitions of Bootstrap is set to 500 while for each setting we repeat 200 times. 
The KS distance is used to evaluate the closeness between Bootstrap samples and the limiting distribution. Again, we approximate the theoretical c.d.f. of the limiting distribution based on 20,000 draws from it.   Figure \ref{fig_BS_N} shows the KS distances for various choices of $N$ and $\gamma$. 
As we we can see, the result of $m$-out-of-$N$ bootstrap gets better as $N$ increases. Furthermore, it suggests $\gamma = 0.3$ which 
gives the best performance over all choices of $N$. 

\begin{figure}[ht]
\centering
\includegraphics[width = 0.6\textwidth]{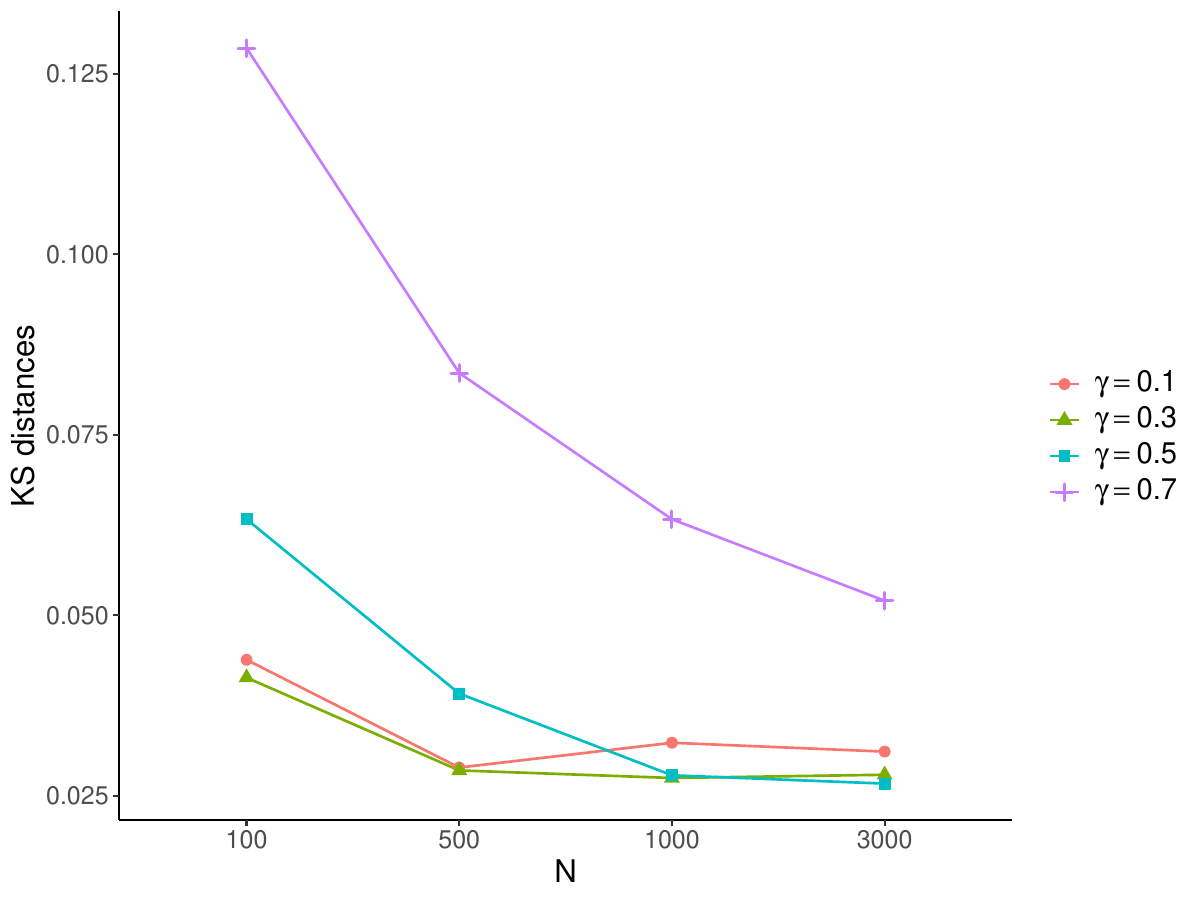}
\caption{The averaged KS distances of $m$-out-of-$N$ bootstrap with $m = N^\gamma$}
\label{fig_BS_N}
\end{figure}

\section{Review of the $\ell_1$-norm error bound of the MLE-type estimators of the mixture weights in \cite{bing2021likelihood}}\label{app_sec_alpha}

In this section, we state in details the theoretical guarantees from \cite{bing2021likelihood} on the MLE-type estimator $\wh \alpha^{(i)}$ in \eqref{MLE}  of $\alpha^{(i)}$. Since it suffices to focus on one $i\in [n]$, we drop the superscripts in this section for simplicity. We assume $\min_{j\in[p]}\|A_{j\cdot}\|_\i>0$. Otherwise, a pre-screening procedure could be used to reduce the dimension $p$ such that this condition is met.

Recalling from \eqref{MLE} that $\wh \alpha$ depends on a given estimator $\wh A$ of $A$.  Let $\wh A$ be a generic estimator of $A$ satisfying  
\begin{align} \label{ass_A_simplex}
&\wh A_k \in \Delta_p,  \quad \text{for } \  k \in [K];\\\label{ass_A_error_col}  
&\max_{k\in[K]}\|\wh A_{k} - A_k\|_1 \le \epsilon_{1,\infty};\\\label{ass_A_error_row} 
& \max_{j\in [p]}{\|\wh A_{j\cdot} - A_{j\cdot}\|_\i \over \|A_{j\cdot}\|_\i} \le \epsilon_{\infty,\i};\\\label{ass_A_error}
&\max_{j\in[p]}{\|\wh A_{j\cdot} - A_{j\cdot}\|_1 \over \|A_{j\cdot}\|_1} \le \epsilon_{\infty,1}.
\end{align}
for  some deterministic sequences  $\epsilon_{1,\infty}$, $\epsilon_{\infty, \i}$ and $\epsilon_{\infty, 1}$.\footnote{Displays \eqref{ass_A_error_col} -- \eqref{ass_A_error} only need to hold up to a label permutation among columns of $A$.}

To state the results in \cite{bing2021likelihood}, we need a few more notation. 
For any $r = A\alpha$ with $\alpha \in \Theta_\alpha(\tau)$, define 
$$
\uJ = \left\{j\in [p]: r_j > {5\log(p) \over N}\right\}.
$$
The set $\uJ$ in conjunction with $\oJ$ in \eqref{def_oJ} are chosen such that 
$$
\PP\left\{\uJ  \subseteq \supp(X) \subseteq \oJ \right\} \ge 1-2p^{-1},
$$ 
see, for instance,  \cite[Lemma I.1]{bing2021likelihood}. Here $X$ is the empirical estimator of $r$.
Further define
\begin{equation}\label{def_xi}
\Tm := \min_{k\in \supp(\alpha)}  \alpha_k,\qquad \xi \coloneqq \max_{j\in \oJ} {\max_{k\notin \supp(\alpha)} A_{jk} \over \sum_{k\in \supp(\alpha)} A_{jk}}.
\end{equation}
Recall that \eqref{def_kappa_A} and \eqref{def_kappas}. Write 
\begin{align}\label{def_kappa_uk}
\uk_\tau = \min_{\alpha\in \Theta_\alpha(\tau)}\kappa(A_{\uJ}, \|\alpha\|_0).
\end{align}  
We have $\uk_\tau \le \ok_\tau$.
Finally, let 
\begin{align*}
M_1 & :=  {\log(p) \over \uk_\tau^4  }{(1\vee \xi)^3 \over \Tm^3},\qquad 
M_2  :=  {\log(p) \over \ok_K^2 }{(1\vee \xi) \over \Tm}{(1 + \xi \sqrt{K-\|\alpha\|_0}) \over \Tm}.
\end{align*}

The following results are proved in \cite[Theorems 9 \& 10]{bing2021likelihood}.

\begin{theorem}\label{thm_mle_alpha}
Let $\alpha \in \Theta_\alpha(\tau)$.
Assume there exists sufficiently large constant $C,C'>0$ such that 
\begin{equation}\label{cond_M1M2}
	N \ge C \max\{M_1, M_2\}
\end{equation}
and	$|\oJ\setminus \uJ| \le C'$. Let $\wh A$ be any estimator such that (\ref{ass_A_simplex}) -- (\ref{ass_A_error_row}) hold with 
\begin{equation}\label{event_Ahat_col_prime}
	{ 2 (1\vee \xi) \over  \Tm} \epsilon_{\i,\i} \le 1,\qquad { 2 (1\vee \xi) \over  \Tm} \epsilon_{1,\i} \le   \uk^2_\tau.
\end{equation}
Then we have,  with probability $1-8p^{-1}$, 
\[
\|\wh \alpha - \alpha\|_1 ~  \lesssim ~ {1\over \ok_\tau} \sqrt{K\log(p) \over N} + {1 \over \ok_\tau^2} \epsilon_{1,\i} .
\]
Furthermore, if additionally
\begin{align*}
	\sqrt{\xi \log(p) \over  \Tm N}\left(1 + {1\over \ok_\tau}\sqrt{\xi \tau \over \Tm}\right)  &+ {\log(p)\over N}+\left(1 + {\xi \over \Tm \ok_\tau^2}\right) \epsilon_{1,\i} \le c \min_{k\notin \supp(\alpha)} \sum_{j\in \oJ^c}A_{jk}
\end{align*}
holds
for some  sufficiently small constant $c>0$, then with probability at least $1-\cO(p^{-1})$
$$
\|\wh \alpha - \alpha \|_1 \lesssim 
{1\over \ok_\tau}\sqrt{\tau\log(p) \over N} +   {1\over \ok_\tau^2} \epsilon_{1,\i} .
$$
\end{theorem}

Specializing to the estimator $\wh A$ in \cite{bing2020fast}, \cref{thm_mle_alpha} can be invoked for $\epsilon_{1,\i}$ and $\epsilon_{\i,\i}$ given by \eqref{rate_A} and \eqref{rate_A_sup}. For future reference, when $K$ is fixed, it is easy to see from \eqref{rate_A_sup} that the same estimator $\wh A$ also satisfies \eqref{ass_A_error} for 
\begin{equation}\label{rate_A_sup_1}
\epsilon_{\i, 1} =  \sqrt{p\log(L)\over nN}
\end{equation}
with probability tending to one. 

\section{Results on the W-distance estimator based on the weighted least squares estimator of the mixture weights}\label{app_least_squares}

In this section we study the W-distance estimator, $\wt W_{LS}$ in \eqref{SWD_bar}, of $W(\balpha^{(i)}, \balpha^{(j)}; d)$, by using the weighted least squares estimators 
\begin{equation}\label{LS}
\wt \alpha_{LS}^{(\ell)} = \wh A^+ X^{(\ell)},\qquad \text{for}\quad \ell \in \{i,j\}
\end{equation}
where
\begin{equation}\label{def_hat_Ainv_DA}
\wh A^+ := (\wh A^\top \wh D^{-1} \wh A)^{-1} \wh A^\top\wh D^{-1} \quad \ \text{and} \quad   \wh D := \diag(\|\wh A_{1\cdot}\|_1, \ldots, \|\wh A_{p\cdot}\|_1),
\end{equation}
with $\wh A$ being a generic estimator of $A$. Due to the weighting matrix $\wh D$,  $\wh A^\top \wh D^{-1} \wh A$ is a doubly stochastic matrix, and its largest eigenvalue is $1$ (see also Lemma~\ref{lem_M_hat}, below).


Fix arbitrary $i\in [n]$. Let  
\begin{equation}\label{def_Vr}
V_{LS}^{(i)} = A^\T D^{-1} \left(\diag(r^{(i)}) - r^{(i)}r^{(i)\T}\right) D^{-1} A
\end{equation}
and 
\begin{equation}\label{def_Sigma_i_LS}
\Sigma_{LS}^{(i)} =  M^{-1}  V_{LS}^{(i)} M^{-1} = A^+ \diag(r^{(i)}) A^{+\T} - \alpha^{(i)}\alpha^{(i)\T}
\end{equation}
with 
$$
M = A^\T D^{-1} A. 
$$
The following theorem establishes the asymptotic normality of $\wt\alpha^{(i)}_{LS}$.

\begin{theorem}\label{ASN_LS}
Under \cref{ass_id} and condition (\ref{multinomial}), let $\wh A$ be any estimator such that \eqref{ass_A_simplex}, \eqref{ass_A_error_col} and \eqref{ass_A_error} hold. 
Assume $\lambda_K^{-1}(M) = \cO(\sqrt{N})$,  $\lambda_K^{-1}(V_{LS}^{(i)} ) = o(N)$ and 
\begin{equation}\label{cond_A_ASN}
	(\epsilon_{1,\infty} + \epsilon_{\infty,1})\sqrt{N \over \lambda_K(V_{LS}^{(i)})} = o(1),\quad \textrm{as $n,N\to \infty$.}
\end{equation}
Then, we have the following convergence in distribution as $n,N\to \infty$,
\[
\sqrt{N}\bigl(\Sigma_{LS}^{(i)}\bigr)^{-1/2}\left(\wt\alpha_{LS}^{(i)} - \alpha^{(i)}\right) \overset{d}{\to} \cN_K(0, \bI_K).
\]
\end{theorem}
\begin{proof}
The proof can be found in \cref{app_proof_ASN_LS}.
\end{proof}

\begin{remark}
We begin by observing that if $p$ is finite, independent of $N$, and $A$ is known, the conclusion of the theorem follows trivially  from the fact that 
$$
\sqrt{N}(X^{(i)} - r^{(i)})\overset{d}{\to} \cN_p(0,  \diag(r^{(i)}) - r^{(i)}r^{(i)\T})
$$   and $\wt \alpha^{(i)} = A^{+} X^{(i)}$, under no conditions beyond the multinomial sampling scheme assumption.   Our theorem generalizes this to the practical situation when $A$ is estimated,  from a  data set that is potentially dependent on that used to construct $\wt \alpha^{(i)}$, and when $p = p(N)$. 

Condition $\lambda_K^{-1}(V_{LS}^{(i)}) = o(N)$ is a mild regularity condition that is needed even if $A$ is known, but both $p = p(N)$ and parameters in $A$ and $\alpha^{(i)}$ are allowed to grow with $N$.

If $V_{LS}^{(i)}$ is rank deficient, so is the asymptotic covariance matrix $ \Sigma_{LS}^{(i)}$. This happens, for instance, when the mixture components $A_k$ have disjoint supports and the weight vector $\alpha^{(i)}$ is sparse. Nevertheless, a straightforward modification of our analysis  leads to the same conclusion, except that the inverse of $\Sigma_{LS}^{(i)}$ gets replaced by a generalized inverse and the limiting covariance matrix becomes $\bI_s$ with $s = \rank(\Sigma_{LS}^{(i)})$ and zeroes elsewhere.

Condition (\ref{cond_A_ASN}), on the other hand, ensures that the error of estimating the mixture components becomes negligible. For the choice of $\wh A$ in \cite{bing2020fast}, in view of \eqref{rate_A} and \eqref{rate_A_sup_1}, \eqref{cond_A_ASN} requires  
\[
\sqrt{p\log(L) \over n \lambda_K(V_{LS}^{(i)})} = o(1), \quad \text{as }n,N\to \i
\]
with $L = n\vee p\vee N$.\\
\end{remark}

The asymptotic normality in \cref{ASN_LS} in conjunction with \cref{LT_SWD} readily yields the limiting distribution of $\wt W_{LS}$ in \eqref{SWD_bar} by using \eqref{LS}, stated in the following theorem.   For arbitrary pair $(i,j)$, let 
\begin{equation}\label{def_Z}
G_{ij} \sim \cN_K(0, Q_{LS}^{(ij)})
\end{equation}
where we assume that the following limit exists
\begin{equation*} 
Q_{LS}^{(ij)} \defeq
\lim_{N\to\infty}  ~   \Sigma_{LS}^{(i)} +  \Sigma_{LS}^{(j)} ~ \in \RR^{K\times K}.
\end{equation*} 

\begin{theorem}\label{thm_limit_distr_LS}
Under Assumption \ref{ass_id} and the multinomial sampling assumption \eqref{multinomial},  let $\wh A$ be any estimators such that \eqref{two}, \eqref{ass_A_simplex}, \eqref{ass_A_error_col} and \eqref{ass_A_error} hold. Assume $\lambda_K^{-1}(M) = \cO(\sqrt{N})$,  $\epsilon_A\sqrt{N}= o(1)$, $ \lambda_K^{-1}(V_{LS}^{(i)}) \vee  \lambda_K^{-1}(V_{LS}^{(j)})  = o(N)$  and 
\begin{equation}\label{cond_A_ASN_rs}
	(\epsilon_{1,\infty} + \epsilon_{\infty,1})\sqrt{N \over \lambda_K(V_{LS}^{(i)}) \wedge \lambda_K(V_{LS}^{(j)})} = o(1),\quad \textrm{as $n,N\to \infty$. }
\end{equation}
Then, as $n, N\to \infty$, we have the following convergence in distribution,
\begin{equation}\label{limit_distr_LS}
	\sqrt{N}\left(
	\wt W_{LS} - W(\balpha^{(i)},\balpha^{(j)}; d)
	\right) \overset{d}{\to} \sup_{f\in \cF'_{ij}}f^\top G_{ij}
\end{equation}
where $\cF'_{ij}$ is defined in \eqref{space_F_prime}. 
\end{theorem}
\begin{proof}
Since \cref{ASN_LS} ensures that
\[
\sqrt{N}\left[(\wt \alpha_{LS}^{(i)} - \wt \alpha_{LS}^{(j)} ) - (\alpha^{(i)} - \alpha^{(i)})\right] \overset{d}{\to}  G_{ij},\quad \textrm{as }n,N\to \infty,
\]
the result follows by invoking \cref{LT_SWD}.
\end{proof}

Consequently, the following corollary provides one concrete instance for which conclusions  in \cref{thm_limit_distr_LS} hold for $\wt W_{LS}$ based on $\wh A$ in \cite{bing2020fast}. 

\begin{corollary}
Grant topic model assumptions and \cref{ass_anchor} as well as conditions listed in \cite[Appendix K.1]{bing2021likelihood}. For any $r^{(i)}, r^{(j)} \in \cS$ with $i,j\in [n]$, suppose $p\log(L) = o(n)$ and
\begin{equation}\label{cond_regularity}
	\max\left\{\lambda_K^{-1}(M),\lambda_K^{-1}(V_{LS}^{(i)}),\lambda_K^{-1}(V_{LS}^{(j)})\right\} = \cO(1). 
\end{equation}
For any $d$ satisfying \eqref{lip_d_upper} with $C_d = \cO(1)$,  then \cref{limit_distr_LS} holds as $n,N\to \infty$.
\end{corollary}


\begin{remark}[The weighted restricted least squares estimator of the mixture weights]\label{rem_WLS}
We argue here that the WLS estimator $\wt \alpha^{(\ell)}_{LS}$ in \eqref{LS} can be viewed as the de-biased estimators of the following {\em restricted} weighted least squares estimator  
\begin{equation}\label{def_WRLS}
	\wh \alpha_{LS}^{(\ell)} := \argmin_{\alpha \in \Delta_K} {1\over 2} \|
	\wh D^{-1/2}(X^{(\ell)} - \wh A\alpha)
	\|_2^2.
\end{equation} 
Suppose that $\wh A = A$ and let us drop the superscripts for simplicity. The K.K.T. conditions of (\ref{def_WRLS}) imply 
\begin{equation}\label{KKTa}
	\sqrt{N}\left(\wh \alpha_{LS} - \alpha\right) = \sqrt{N} A^{+}(X - r) + \sqrt{N}M^{-1}\left(\lambda - \mu \1_K\right).
\end{equation}
Here $\lambda \in \RR_+^K$ and $\mu \in  \RR$ are Lagrange variables corresponding to the  restricted optimization in \eqref{def_WRLS}. 
We see that the first term on the right-hand side in \cref{KKTa} is asymptotically normal, while the second one  does not always converge to a normal limit. Thus, it is natural to consider a modified estimator 
$$ 
\wh \alpha_{LS} - M^{-1}\left(\lambda - \mu \1_K\right)
$$
by removing the second term in \eqref{KKTa}. Interestingly, this new estimator is nothing but $\wt \alpha_{LS}$  in \eqref{LS}. This suggests that $\sqrt{N}(\wt \alpha_{LS}-\alpha)$ converges to a normal limit faster than that of $\sqrt{N}(\wh\alpha_{LS}-\alpha)$, as confirmed in Figure \ref{fig_ASN_T} of \cref{app_sim_LS}.


\end{remark}

\subsection{Proof of \cref{ASN_LS}}\label{app_proof_ASN_LS}

\begin{proof}
For simplicity, let us drop the superscripts. By definition, we have 
\begin{align*}
	\wt \alpha_{LS} - \alpha
	& = A^+ (X - r) + (\wh A^+ - A^+)(X - r) +(\wh A^+ - A^+) ~ r
\end{align*}
where we let
\begin{equation}\label{def_A_Ahat_inv}
	A^+ = M^{-1} A^\T D^{-1},\qquad \wh A^+ = \wh M^{-1} \wh A^\T \wh D^{-1},
\end{equation}
with $M = A^\T D^{-1} A$ and $\wh M = \wh A^\T \wh D^{-1}\wh A$. Recall that
\begin{equation}\label{def_Qr}
	\Sigma_{LS} = M^{-1}V_{LS} M^{-1}.
\end{equation}
It suffices to show that, for any $u\in \RR^{K}\setminus\{0\}$, 
\begin{flalign}\label{task_1_LS}
	&\sqrt{N}{u^\T A^+ (X - r) \over \sqrt{u^\T \Sigma_{LS} u}} \overset{d}{\to} \cN(0,1),\\\label{task_2_LS}
	&\sqrt{N}{u^\T (\wh A^+ - A^+)(X - r)\over \sqrt{u^\T \Sigma_{LS}u}} = o_\PP(1),\\\label{task_3_LS}
	&\sqrt{N}{u^\T (\wh A^+ - A^+) ~ r \over \sqrt{u^\T \Sigma_{LS} u}} = o_\PP(1).
\end{flalign}
Define the $\ell_2\to\ell_\infty$ condition number of $V_{LS}$ as 
\begin{equation}\label{def_kappa}
	\kappa(V_{LS}) = \inf_{u\ne 0} {u^\T V_{LS} u \over \|u\|_\infty^2}.
\end{equation}
Notice that 
\begin{equation}\label{ineq_kappa}
	{\kappa(V_{LS}) \over K} \le  \lambda_K(V_{LS}) \le \kappa(V_{LS}).
\end{equation}

{\bf Proof of \cref{task_1_LS}:}
By recognizing that the left-hand-side of \eqref{task_1_LS} can be written as 
\[
{1\over \sqrt N}\sum_{t=1}^N {u^\T A^+ (Z_t - r) \over \sqrt{u^\T  \Sigma_{LS} u}} \defeq {1\over \sqrt N}\sum_{t=1}^N Y_t.
\]
with $Z_t$, for $t\in [N]$, being i.i.d. $\textrm{Multinomial}_p(1; r)$. It is easy to see that $\EE[Y_t]  = 0$ and $\EE[Y_t^2] = 1$. To apply the Lyapunov central limit theorem, it suffices to verify 
\[
\lim_{N\to\infty}{\EE[|Y_t|^3]\over \sqrt{N}} = 0.
\]
This follows by noting that 
\begin{align*}
	\EE[|Y_t|^3] &\le 2{\|u^\T A^+\|_\infty \over \sqrt{u^\T \Sigma_{LS} u}} && \textrm{by }\EE[Y_t^2]=1, \|Z_t-r\|_1 \le 2\\
	&\le {2\|u^\T A^\T D^{-1}\|_\infty \over \sqrt{u^\T V_{LS} u}} && \textrm{by \cref{def_Qr}}\\
	&\le {2 \over \sqrt{\kappa(V_{LS}})} && \textrm{by \cref{def_kappa}} 
\end{align*}
and by using $\lambda_K^{-1}(V_{LS}) = o(N)$ and \cref{ineq_kappa}.

{\bf Proof of \cref{task_2_LS}:}
Fix any $u\in\RR^{K}\setminus\{0\}$. Note that 
\begin{align*}
	{u^\T (\wh A^+ - A^+)(X - r)\over \sqrt{u^\T \Sigma_{LS} u}} & = {u^\T M (\wh A^+ - A^+)(X - r)\over \sqrt{u^\T V_{LS} u}}\\
	&\le {\|u\|_{\infty}\|M(\wh A^+ - A^+)(X - r)\|_1  \over  \|u\|_{\infty}\sqrt{\kappa(V_{LS})} } &&\textrm{by \cref{def_kappa}}.
\end{align*}
Further notice that $\|M(\wh A^+ - A^+)(X - r)\|_1$ is smaller than 
\[
\|M-\wh M\|_{1,\infty} \|(\wh A^+ - A^+)(X - r)\|_1 + \|\wh M(\wh A^+ - A^+)(X - r)\|_1.
\]
For the first term, since 
\begin{align}\label{eq_decomp_diff_A_inv}\nonumber
	\wh A^+ - A^+ & =  \wh M^{-1}\wh A^\T (\wh D^{-1} - D^{-1}) + \wh M^{-1}(\wh A- A)^\T D^{-1}\\\nonumber
	&\quad + (\wh M^{-1} - M^{-1}) A^\T D^{-1}\\\nonumber
	& =  \wh M^{-1}\wh A^\T \wh D^{-1}(D - \wh D)D^{-1} + \wh M^{-1}(\wh A- A)^\T D^{-1}\\
	&\quad + \wh M^{-1}(M - \wh M)M^{-1} A^\T D^{-1},
\end{align}
we obtain 
\begin{align}\label{bd_diff_A+_diff_r_init}\nonumber
	&\|(\wh A^+ - A^+)(X - r)\|_1\\\nonumber
	&\le \|\wh M^{-1}\|_{1,\infty} \|\wh A^\T \wh D^{-1}\|_{1,\infty} \| (\wh D - D)D^{-1}\|_{1,\infty}\|X - r\|_1\\\nonumber
	&\quad + \|\wh M^{-1}\|_{1,\infty}\|(\wh A- A)^\T D^{-1}\|_{1,\infty}\|X - r\|_1\\\nonumber
	&\quad + \|\wh M^{-1}(M - \wh M)\|_{1,\infty} \|A^+(X - r)\|_1\\\nonumber
	&\le 2\|\wh M^{-1}\|_{1,\infty} \epsilon_{\infty,1}\|X - r\|_1  +  \|\wh M^{-1}(M - \wh M)\|_{1,\infty} \|A^+(X - r)\|_1\\
	&\le 4\|M^{-1}\|_{1,\infty}\epsilon_{\infty,1} \|X - r\|_1 + 2\|M^{-1}\|_{\infty,1}(2\epsilon_{\infty,1} + \epsilon_{1,\infty}) \|A^+(X - r)\|_1
\end{align}
where in the last line we invoke \cref{lem_M_hat} by observing that condition \eqref{cond_A_cr} holds under  \eqref{cond_A_ASN}. \cref{bd_diff_A+_diff_r_init} in conjunction with \cref{bd_MA_diff_op} yields 
\begin{align}\label{bd_T1}\nonumber
	&\|M-\wh M\|_{1,\infty} \|(\wh A^+ - A^+)(X - r)\|_1\\
	& \le 2\|M^{-1}\|_{1,\infty}(2\epsilon_{\infty,1} + \epsilon_{1,\infty}) \Bigl[2\epsilon_{\infty,1} \|X - r\|_1  + (2\epsilon_{\infty,1} + \epsilon_{1,\infty}) \|A^+(X - r)\|_1\Bigr].
\end{align}
On the other hand, from \cref{eq_decomp_diff_A_inv} and its subsequent arguments, it is straightforward to show 
\begin{align}\label{bd_T2}\nonumber
	\|\wh M(\wh A^+ - A^+)(X - r)\|_1 &\le  \|\wh A^\T \wh D^{-1}\|_{1,\infty} \| (\wh D - D)D^{-1}\|_{1,\infty}\|X - r\|_1\\\nonumber
	&\quad + \|(\wh A- A)^\T D^{-1}\|_{1,\infty}\|X - r\|_1\\\nonumber
	&\quad + \|M - \wh M\|_{1,\infty} \|A^+(X - r)\|_1\\
	&\le 2 \epsilon_{\infty,1}\|X - r\|_1  +  (2\epsilon_{\infty,1} + \epsilon_{1,\infty}) \|A^+(X - r)\|_1.
\end{align}
By combining \cref{bd_T1} and \cref{bd_T2} and by using \cref{cond_A_ASN} together with the fact that 
\begin{equation}\label{ineq_Minv}
	\|M^{-1}\|_{\infty,1} \le \sqrt{K}\|M^{-1}\|_{\op} =  {\sqrt K\over \lambda_K(M)}
\end{equation}
to collect terms, we obtain
\begin{align*}
	\|M(\wh A^+ - A^+)(X - r)\|_1 \lesssim   \epsilon_{\infty,1} \|X -r \|_1 + (2\epsilon_{\infty,1} + \epsilon_{1,\infty}) \|A^+(X - r)\|_1.
\end{align*}
Finally, since  $\|X -r \|_1  \le 2$ and \cref{lem_B_r_diff} ensures that, for any $t \ge 0$,
\begin{align}\label{bd_A+_diff_r_init}
	\|A^+ (X - r)\|_1  &\le    \max_{j\in[p]}\|(A^+)_j\|_2 \left(\sqrt{ 2K\log(2K/t)~  \over N} + {4K\log(2K/t) \over N} \right),
\end{align}
with probability at least $1-t$, by also using 
\[
\|(A^+)_j\|_2 \le {\|M^{-1}A_{j\cdot}\|_2 \over \|A_{j\cdot}\|_1} \le \lambda_K^{-1}(M),
\]
we conclude
\[
\|M(\wh A^+ - A^+)(X - r)\|_1 = \cO_\PP\left(
\epsilon_{\infty,1} + (\epsilon_{\infty,1} + \epsilon_{1,\infty})\lambda_K^{-1}(M)\sqrt{K\log(K) \over  N}
\right).
\]
This further implies   \cref{task_2_LS} under $\lambda_K^{-1}(M) = \cO(\sqrt N)$,  \cref{cond_A_ASN} and \cref{ineq_kappa}.

{\bf Proof of \cref{task_3_LS}:} By similar arguments, we have 
\[
\sqrt{N}{u^\T (\wh A^+ - A^+)r \over \sqrt{u^\T V_{LS} u}}  \le \sqrt{N}{\|M(\wh A^+ - A^+)r\|_1  \over \sqrt{\kappa(V_{LS})}}.
\]
Since, for $r = A\alpha$,
\begin{align*}
	\|M(\wh A^+ - A^+)r\|_1 &= \|M \wh A^+(A -\wh A)\alpha\|_1\\
	&\le \|M\wh A^+\|_{1,\infty}\|(\wh A  - A)\alpha\|_1\\
	&\le \left(\|\wh A^\T \wh D^{-1}\|_{1,\infty} + 
	\|(\wh M - M) \wh M^{-1}\wh A^\T \wh D^{-1}\|_{1,\infty}
	\right) \|\wh A - A\|_{1,\infty}\\
	&\le (1 + \|\wh M - M\|_{1,\infty} \|\wh M^{-1}\|_{1,\infty}) \epsilon_{1,\infty}\\
	&\le \Bigl(1 + 2\|M^{-1}\|_{1,\infty}(2\epsilon_{\infty,1} + \epsilon_{1,\infty})\Bigr)\epsilon_{1,\infty}
\end{align*}
by using \cref{lem_M_hat} and \cref{bd_MA_diff_op} in the last step, invoking \cref{cond_A_ASN}, (\ref{ineq_Minv}) and $\lambda_K^{-1}(M) = \cO(\sqrt{N})$ concludes 
\[
\|M(\wh A^+ - A^+)r\|_1 = \cO(\epsilon_{1,\infty})
\]
implying \cref{task_3_LS}. 
The proof is complete.
\end{proof}

\subsection{Lemmas used in the proof of \cref{ASN_LS}}

The following lemma controls various quantities related with $\wh M$. 

\begin{lemma}\label{lem_M_hat}
Under \cref{ass_id}, assume 
\begin{equation}\label{cond_A_cr}
	\|M^{-1}\|_{\infty,1} \left( \epsilon_{1,\infty} + 2\epsilon_{\infty,1} \right)  \le 1/2.
\end{equation}
Then on the event that \cref{ass_A_simplex} and \cref{ass_A_error} hold, one has 
\[
\lambda_K(\wh M) \ge \lambda_K(M) / 2,\qquad \|\wh M^{-1}\|_{\infty, 1} \le 2\|M^{-1}\|_{\infty, 1}
\]
and
\[
\max\left\{\|(\wh M - M)\wh M^{-1}\|_{1,\infty}, ~ \|(\wh M - M)\wh M^{-1}\|_{\infty,1} \right\} \le 2 \|M^{-1}\|_{\infty,1}(2\epsilon_{\infty,1} + \epsilon_{1,\infty}).
\]
Moreover, $\lambda_K(\wh M) \ge \lambda_K(M) / 2$ is guaranteed by \cref{ass_id} and $4\epsilon_{\infty,1} + 2\epsilon_{1,\infty} \le \lambda_K(M)$.  
\end{lemma}
\begin{proof}
We assume the identity label permutation for \cref{ass_A_simplex} and \cref{ass_A_error}. First, by using $\|M\|_{\op} \le \|M\|_{\infty,1}$ for any symmetric matrix $M$, we find
\begin{equation}\label{lb_MA_inf_1}
	\|M^{-1}\|_{\infty, 1} \ge \|M^{-1}\|_{\op} = \lambda_K^{-1}(M) \ge 1.
\end{equation}
Then condition (\ref{cond_A_cr}) guarantees, for any $j\in [p]$,
\begin{equation}\label{lbd_A_hat_row}
	\|\wh A_{j\cdot}\|_1 \ge \|A_{j\cdot}\|_1\left( 1 - {\|\wh A_{j\cdot} - A_{j\cdot}\|_1 \over \|A_{j\cdot}\|_1}\right) \ge {1\over 2}\|A_{j\cdot}\|_1.
\end{equation}
To prove the first result, by Weyl's inequality, we have 
\[
\lambda_K(\wh M) \ge \lambda_K(M) - \|\wh M - M\|_{\op}. 
\]
Using $\|M\|_{\op} \le \|M\|_{\infty,1}$ again yields
\begin{align}\label{bd_MA_diff_op}\nonumber
	\|\wh M - M\|_{\op} & \le \|\wh M - M\|_{\infty,1}\\\nonumber
	&\le \|\wh A^\T \wh D^{-1}(\wh A - A)\|_{\infty,1} + \|\wh A^\T (\wh D^{-1} - D^{-1})A\|_{\infty,1}\\\nonumber
	&\quad + \|(\wh A - A)^\T D^{-1}A\|_{\infty,1}\\\nonumber
	&\overset{(i)}{\le}  \|\wh D^{-1}(\wh A - A)\|_{\infty,1} + \|\wh D^{-1}(\wh D-D)D^{-1}A\|_{\infty,1}\\\nonumber
	&\quad + \|(\wh A - A)^\T\|_{\infty,1}  \|D^{-1}A\|_{\infty,1}\\
	&\overset{(ii)}{\le} 2\max_{j\in[p]} {\|\wh A_{j\cdot} -A_{j\cdot}\|_1 \over \|\wh A_{j\cdot}\|_1} +  \max_{k\in [K]}\|\wh A_k - A_k\|_1,
\end{align}
where the step $(i)$ uses $\|\wh A^\T\|_{\infty,1} = 1$ and the step $(ii)$ is due to $\|D^{-1}A\|_{\infty,1} = \max_j \|A_{j\cdot}\|_1 / \|A_{j\cdot}\|_1 =  1$. Invoke \cref{ass_A_error} to obtain 
\begin{align*}
	\|\wh M - M\|_{\op} \le 2\epsilon_{\infty,1} + \epsilon_{1,\infty} \overset{(\ref{cond_A_cr})}{\le} {1\over 2\|M^{-1}\|_{\infty,1}} .
\end{align*}
Together with (\ref{lb_MA_inf_1}), we conclude $\lambda_K(\wh M) \ge \lambda_K(M) / 2$. Moreover, we readily see that $\lambda_K(\wh M) \ge \lambda_K(M) / 2$ only requires
$\lambda_K(M) / 2 \ge 2\epsilon_{\infty,1} + \epsilon_{1,\infty}$.

To prove the second result, start with the decomposition
\begin{align}\label{eq_decom_diff_M_inv}\nonumber
	\wh M^{-1} - M^{-1} &= M^{-1} (M - \wh M) \wh M^{-1}\\\nonumber
	&= M^{-1}A^\T D^{-1}(A - \wh A) +  M^{-1}A^\T(D^{-1} - \wh D^{-1})\wh A \wh M^{-1}\\
	&\quad + M^{-1}(A - \wh A)^\T \wh D^{-1}\wh A \wh M^{-1}.
\end{align}
It then follows from dual-norm inequalities, $\|A^\T\|_{\infty, 1} = 1$ and $\|\wh D^{-1}\wh A\|_{\infty,1}=1$ that 
\begin{align*}
	\|\wh M^{-1}\|_{\infty,1} &\le  \| M^{-1}\|_{\infty,1}  + \|\wh M^{-1} - M^{-1}\|_{\infty,1} \\
	& \le \| M^{-1}\|_{\infty,1} + \| M^{-1}\|_{\infty,1}\|D^{-1}(A - \wh A)\|_{\infty,1}\\
	&\quad  +  \| M^{-1}\|_{\infty,1} \|(D^{-1} - \wh D^{-1})\wh A\|_{\infty,1} \|\wh M^{-1}\|_{\infty,1}\\
	&\quad + \| M^{-1}\|_{\infty,1}\|(A - \wh A)^\T\|_{\infty,1}\|\wh M^{-1}\|_{\infty,1}\\
	&\le \| M^{-1}\|_{\infty,1} + \| M^{-1}\|_{\infty,1} \|D^{-1}(A - \wh A)\|_{\infty,1}\\
	&\quad  +  \| M^{-1}\|_{\infty,1} \|D^{-1}(\wh D - D)\|_{\infty,1} \|\wh M^{-1}\|_{\infty,1}\\
	&\quad + \| M^{-1}\|_{\infty,1}\|(A - \wh A)^\T\|_{\infty,1}\|\wh M^{-1}\|_{\infty,1}\\
	&\le  \| M^{-1}\|_{\infty,1} + \| M^{-1}\|_{\infty,1}  \epsilon_{\infty,1}\\
	&\quad  +  \| M^{-1}\|_{\infty,1}(\epsilon_{\infty,1} + \epsilon_{1,\infty}) \|\wh M^{-1}\|_{\infty,1}.
\end{align*}
Since, similar to \cref{lb_MA_inf_1}, we also have $\|\wh M^{-1}\|_{\infty,1} \ge 1$. 
Invoke condition \cref{cond_A_cr} to conclude 
\[
\|\wh M^{-1} \|_{\infty,1} \le 2  \| M^{-1}\|_{\infty,1},
\]
completing the proof of the second result. 

Finally, \cref{bd_MA_diff_op} and the second result immediately give
\begin{align*}
	\|(\wh M - M)\wh M^{-1}\|_{\infty,1} & = \|\wh M - M\|_{\infty,1} \|\wh M^{-1} \|_{\infty,1}\le 2  \| M^{-1}\|_{\infty,1}(2\epsilon_{\infty,1} + \epsilon_{1,\infty}).
\end{align*}
The same bound also holds for $\|(\wh M - M)\wh M^{-1}\|_{1,\infty}$ by the symmetry of $\wh M - M$ and $\wh M^{-1}$.
\end{proof}

\begin{lemma}\label{lem_B_r_diff}
For any fixed matrix $B\in \RR^{K\times p}$, with probability at least $1-t$ for any $t\in (0,1)$, we have 
\[
\|B(X - r)\|_1 ~ \le  ~ \max_{j\in[p]}\|B_j\|_2 \left(\sqrt{ 2K\log(2K/t)~  \over N} + {4K\log(2K/t) \over N} \right).
\]
\end{lemma}
\begin{proof}
Start with $\|B(X - r)\|_1 = \sum_{k\in [K]} |B_{k\cdot}^\T (X - r)|$ and pick any $k \in [K]$. 
Note that 
\[
N X \sim \textrm{Multinomial}_p(N; r).
\]   
By writing 
\[
X - r = {1\over N}\sum_{t=1}^N \left(Z_t - r\right)
\]
with $Z_t \sim \textrm{Multinomial}_p(1; r)$, we have 
\[
B_{k\cdot}^\T (X - r)  = {1\over N}\sum_{t=1}^N B_{k\cdot}^\T \left(Z_t - r\right).
\]
To apply the Bernstein inequality, recalling that $r\in \Delta_p$, we find that $B_{k\cdot}^\T\left(Z_t - r\right)$ has zero mean, 
$B_{k\cdot}^\T \left(Z_t - r\right) \le 2\|B_{k\cdot}\|_{\infty}$ for all $t\in [N]$ and 
\[
{1\over N}\sum_{t=1}^N \textrm{Var}\left(B_{k\cdot}^\T\left(Z_t - r\right)\right)\le  B_{k\cdot}^\T \textrm{diag}(r) B_{k\cdot}.
\]
An application of the Bernstein inequality gives, for any $t\ge 0$, 
\[
\PP\left\{
B_{k\cdot}^\T (X - r)  \le  \sqrt{t ~ B_{k\cdot}^\T \textrm{diag}(r) B_{k\cdot}  \over N} + {2t \|B_{k\cdot}\|_\infty \over N}
\right\} \ge  1 - 2e^{-t/2}.
\]
Choosing $t = 2\log(2K/\epsilon)$  for any $\epsilon \in (0,1)$ and taking the union bounds over $k\in [K]$ yields
\[
\|B(X - r)\|_1  \le \sqrt{ 2\log(2K/\epsilon)~  \over N} \sum_{k\in [K]}\sqrt{B_{k\cdot}^\T \textrm{diag}(r) B_{k\cdot} } + {4\log(2K/\epsilon) \over N}\sum_{k\in [K]} \|B_{k\cdot}\|_\infty 
\]
with probability exceeding $1-\epsilon$. 
The result thus follows by noting that 
\[
\sum_{k\in [K]}\sqrt{B_{k\cdot}^\T \textrm{diag}(r) B_{k\cdot} }  \le \sqrt{K}\sqrt{
	\sum_{k\in [K]}\sum_{j\in [p]}  r_j B_{kj}^2
} \le \sqrt{K}\max_{j\in[p]}\|B_{j}\|_2
\]
and $\sum_{k\in [K]} \|B_{k\cdot}\|_\infty \le K\max_{j\in[p]}\|B_{j}\|_2 $.
\end{proof}

\end{document}